\documentclass[12pt]{article}

\usepackage[margin=1in]{geometry} 
\usepackage{inputenc} 
\usepackage{authblk} 
\usepackage{hyperref} 
\usepackage{color}
\usepackage{array,multirow}
\usepackage{subcaption}
\usepackage{graphicx}
\usepackage{tabularx}
\usepackage{mathtools,amsthm,amsmath,amssymb}
\usepackage{bm} 
\usepackage{algorithm} 
\usepackage{algpseudocode}
\usepackage[capitalise]{cleveref}

\theoremstyle{definition}
\newtheorem{theorem}{Theorem}[section]
\newtheorem{proposition}[theorem]{Proposition}
\newtheorem{lemma}[theorem]{Lemma}
\newtheorem{corollary}[theorem]{Corollary}
 
\newtheorem{assumption}[theorem]{Assumption}
\newtheorem{counterexample}[theorem]{Counterexample}

\theoremstyle{definition}

\newtheorem{example}[theorem]{Example}

\theoremstyle{remark}
\newtheorem{remark}[theorem]{Remark}

\DeclareMathOperator{\trace}{trace}

\DeclareMathOperator{\Dkl}{D_\mathrm{KL}}
\DeclareMathOperator{\Ker}{Ker}
\DeclareMathOperator{\rank}{rank}

\newcommand{\Ex}{\mathbb{E}}
\newcommand{\nut}{\tilde{\nu}}
\newcommand{\nulis}{\nu^{\scriptscriptstyle \text{LIS}}}
\newcommand{\nuas}{\nu^{\scriptscriptstyle \text{AS}}}
\newcommand{\Hlis}{H_{\scriptstyle \text{LIS}}}
\newcommand{\Has}{H_{\scriptstyle \text{AS}}}

\usepackage[textsize=tiny,textwidth=1.05in]{todonotes}

\providecommand{\keywords}[1]
{
  \small	
  \textbf{Keywords } #1
}

\begin{document}

\title{Certified dimension reduction in nonlinear Bayesian inverse problems}

\author{\sc 
Olivier Zahm\footnote{Univ. Grenoble Alpes, Inria, CNRS, Grenoble INP, LJK, 38000 Grenoble, France (Corresponding author \nolinkurl{olivier.zahm@inria.fr})}, 
Tiangang Cui\footnote{School of Mathematical Sciences, Monash University, Victoria 3800, Australia.}, 
Kody Law\footnote{School of Mathematics, University of Manchester, Manchester, M13 9PL, UK.}, 
Alessio Spantini\footnote{Department of Aeronautics and Astronautics, MIT, Cambridge, MA 02139, USA.}, 
Youssef Marzouk$^\text{\textsection}$
}

\maketitle

\begin{abstract}
 We propose a dimension reduction technique for Bayesian inverse problems with nonlinear forward operators, non-Gaussian priors, and non-Gaussian observation noise. The likelihood function is approximated by a ridge function, \text{i.e.,} a map which depends non-trivially only on a few linear combinations of the parameters. We build this ridge approximation by minimizing an upper bound on the Kullback--Leibler divergence between the posterior distribution and its approximation. This bound, obtained via logarithmic Sobolev inequalities, allows one to certify the error of the posterior approximation. Computing the bound requires computing the second moment matrix of the gradient of the log-likelihood function. In practice, a sample-based approximation of the upper bound is then required. We provide an analysis that enables control of the posterior approximation error due to this sampling. Numerical and theoretical comparisons with existing methods illustrate the benefits of the proposed methodology.
\end{abstract}

\keywords{
dimension reduction, 
nonlinear Bayesian inverse problem,
logarithmic Sobolev inequality,
certified error bound,
non-asymptotic analysis.
}

\section{Introduction}

Solving Bayesian inverse problems \cite{kaipio2006statistical,stuart2010inverse} is a challenging task in many domains of application, due to the complexity of the posterior distribution. One of the primary sources of complexity is the dimension of the parameters to be inferred, which is often high or in principle infinite---e.g., when the posterior is a distribution over functions or their finite-dimensional discretization. High dimensionality presents difficulties for posterior sampling: care is required to design sampling algorithms that mix effectively while remaining robust under refinement of the discretization. High dimensionality also raises significant hurdles to the use of model reduction or approximation schemes (e.g., \cite{brennan2020greedy,conrad2016accelerating,cui2021conditional,cui2015data,manzoni2016accurate,rubio2018fast}) that attempt to reduce the cost of likelihood or forward model evaluations.

Successful strategies for high-dimensional Bayesian inverse problems typically exploit the presence of some low-dimensional structure. A common structure in inverse problems is that the posterior is a low-dimensional \emph{update} of the prior, in the sense that change from prior to posterior is most prominent on a low-dimensional subspace of the parameter space. 
Put another way, the likelihood is influential, relative to the prior, only on a low-dimensional subspace (we will make these intuitive notions more precise later). Sources of such structure include the smoothing properties of the forward operator and limitations on the number or accuracy of the observations.
In the linear--Gaussian case, this structure is already well understood. For instance, low-rank approximations of the (prior-preconditioned) Hessian of the log-likelihood have been used in \cite{flath2011fast} to approximate the posterior covariance. In \cite{spantini2015optimal}, this approach is shown to yield \emph{optimal} approximations of the posterior covariance and of the posterior mean. A heuristic extension of this approach to nonlinear forward models, known as the \emph{likelihood-informed subspace} (LIS) method, is proposed in \cite{cui2014likelihood} and shown to perform well in many applications \cite{beskos2018multilevel,cui2021data,lamminpaa2019likelihood}. A similar idea underlies the \emph{active subspace} method applied to Bayesian problems \cite{constantine2016accelerating}. Identifying the subspace on which changes from the prior to posterior are most prominent also provides a foundation for infinite-dimensional MCMC algorithms that ``split'' the parameter space, such as the \emph{dimension-independent likelihood-informed} (DILI) MCMC samplers of \cite{cui2016dimension}.
In \cite{cui2016scalable}, LIS-based dimension reduction is shown to be an effective prelude to model order reduction, where the forward model is replaced with a computationally inexpensive approximation. Yet none of these dimension reduction approaches, outside of the linear--Gaussian setting of \cite{spantini2015optimal}, comes with a rigorous error analysis. The subject of this paper is a new dimension reduction method for the nonlinear\slash non-Gaussian setting that is, in contrast, \emph{certified}---in the sense that the error induced by the approximation is endowed with a computable upper bound.

Other forms of dimension reduction for Bayesian inverse problems have also been proposed, besides the ``update'' form of dimension reduction described above. For instance, in \cite{li2006efficient,marzouk2009dimensionality}, a truncated Karhunen-Lo\`eve decomposition of the prior distribution is used to reduce the parameter dimension of the entire inverse problem. This approach exploits only the low-dimensional structure of the prior distribution, however, and does not take advantage of structure in the likelihood function or forward model. In \cite{lieberman2010parameter}, a greedy algorithm is used to identify a parameter subspace capable of reproducing the forward model. This approach, in contrast, does not take advantage of the prior correlation structure, and moreover can remove directions that are uninformed by the likelihood but that still retain large variation under the true posterior. More recent results on the intrinsic dimension  of linear Bayesian inverse problems \cite{agapiou2017importance} reinforce the idea that the \emph{update} from prior to posterior should instead be a central object, e.g., when characterizing the performance of importance sampling schemes.

\subsection{Contribution}
In this paper we propose a methodology to detect and exploit the low-dimensional structure of the update from prior to posterior. Our approach addresses Bayesian inverse problems with nonlinear forward operators, non-Gaussian priors, and non-Gaussian observation noise: the posterior $\nu$ is a measure on $\mathbb{R}^d$ given by $\text{d}\nu\propto f\text{d}\mu$ where $\mu$ is the prior measure and $f$ is any unnormalized likelihood function defined on $\mathrm{supp}(\mu)$.
Our approach consists in approximating $x\mapsto f(x)$ by a \emph{ridge function}, i.e., a function of the form $x\mapsto h(Ax)$ which depends non-trivially only on $r\ll d$ linear combinations of the parameters $Ax$, with $A\in\mathbb{R}^{r\times d}$. 

More precisely, we seek a \emph{controlled approximation} such that the Kullback--Leibler (KL) divergence $\Dkl( \nu || \nu_r )$ from the resulting posterior approximation $\text{d}\nu_r(x) \propto h(Ax)\text{d}\mu(x)$ to the posterior $\nu$ is below some user-defined threshold. 
As our main contribution, we derive in Theorem \ref{prop:subspaceLogSob} and Corollary \ref{cor:KL_bound} an upper bound on the KL divergence. Minimizing this bound permits us to construct the ridge approximation so that the bound falls below the threshold.

This bound is obtained via \emph{logarithmic Sobolev inequalities}, an important class of inequalities in measure theory \cite{gross1975logarithmic,guionnet2003lectures,otto2000generalization} with many implications on concentration of measure phenomena \cite{boucheron2013concentration,ledoux1999concentration}. 
Using logarithmic Sobolev inequalities requires some assumptions on the prior distribution and on the regularity of the likelihood function. In particular, we need the gradient of the log-likelihood to be square-integrable over the posterior distribution.
A similar methodology has been proposed in \cite{zahm2020gradient} to reduce the input dimension of multivariate functions. In that paper, a bound on the function approximation error in $L^2$ norm is obtained via Poincar\'e inequalities, another class of Sobolev inequality.

In the proposed dimension reduction method, the informed subspace is constructed as the dominant eigenspace of the matrix
$$
  H = \int (\nabla\log f)(\nabla\log f)^T  \mathrm{d}\nu.
$$
Once the informed subspace is identified, the posterior approximation is constructed by computing a conditional expectation of the likelihood function.
We prove that the resulting divergence $\Dkl( \nu || \nu_r )$ is bounded by the sum of the $d-r$ smallest eigenvalues of $H$: a quickly decaying spectrum in $H$ thus reveals the low-dimensional structure in the posterior.

The proposed method requires computing (i) the matrix $H$ and (ii) a conditional expectation of the likelihood function. 
In practice, we approximate both quantities with Monte Carlo  estimates, requiring samples from the posterior for (i) and samples from the prior for (ii). 
Provided the sample sizes are sufficiently large, we prove that the resulting random posterior approximation is a quasi-optimal approximation compared to the theoretical posterior approximation obtained by exact integration.
This quasi-optimality result is given in expectation for (ii) and in high probability for (i).  In particular, we show that the number of posterior samples to approximate (i) should scale in proportion to the \textit{rank} of the matrix to be estimated, which can be much smaller than the ambient dimension.
Finally, even if the method only requires a limited number of posterior samples, it can be difficult to obtain the posterior samples in (i). We thus propose several alternatives for computing (i) that do not require sampling from the exact posterior. These alternatives include an iterative procedure that builds a sequence of low-dimensional posterior approximations in order to obtain an accurate final estimate.

The outline of the paper is as follows. 
In Section \ref{sec:2} we introduce the theoretical tools to derive and to minimize the error bound. We demonstrate the benefits of this method on an analytical example.
In Section \ref{sec:3} we propose algorithms for the numerical construction of the low-dimensional posterior approximation.
We give a theoretical analysis of the convergence of these sample-based estimators.
In order to provide some context for our developments, we show in Section \ref{sec:4} how the proposed methodology compares with existing dimension reduction methods, including a comparison with a Karhunen-Lo\`eve decomposition.
Finally, Section \ref{sec:numerics} illustrates the benefits of our approach on two numerical examples.

\section{Dimension reduction for the approximation of high-dimensional distributions}\label{sec:2}

Let $\mu$ be a probability distribution defined over the Borel sets $\mathcal{B}(\mathbb{R}^d)$ of $\mathbb{R}^d$. Given a measurable function $f:\mathrm{supp}(\mu)\rightarrow\mathbb{R}^+$ such that $\int f ~\mathrm{d}\mu < \infty$, let $\nu$ be the probability distribution such that
$$
 \frac{\mathrm{d}\nu}{\mathrm{d}\mu} \propto f.
$$
In the context of Bayesian inverse problems, one can view $\mu$ as the \emph{prior distribution} and $\nu$ as the \emph{posterior distribution}, while $f$ represents, up to a multiplicative constant, the \emph{likelihood function}. We consider the problem of approximating the posterior $\nu$ by a probability distribution $\nu_r$ such that
\begin{equation}\label{eq:nuTilde}
 \frac{\mathrm{d}\nu_r}{\mathrm{d}\mu} \propto g \circ P_r ,
\end{equation}
where $P_r:\mathbb{R}^d\rightarrow\mathbb{R}^d$ is a linear projector with rank $r$ and $g:\text{Im}(P_r)\rightarrow\mathbb{R}^+$ is a Borel function called the \emph{profile function}. 
Throughout this paper we identify the projector with its matrix representation $P_r\in\mathbb{R}^{d\times d}$.
Notice that $P_r$ is not restricted to be orthogonal: it can be any matrix $P_r\in\mathbb{R}^{d\times d}$ which satisfies $P_r^2=P_r$ and $\text{rank}(P_r) = r$, but not necessarily $P_r^T=P_r$. 
Any vector $x\in \mathbb{R}^d$ can be uniquely decomposed as
$$
 x = x_r + x_\perp \quad\text{ with }
 \left\{\begin{array}{l}
  x_r = P_r x  \\
  x_\perp = (I_d-P_r)x 
 \end{array}
 \right.
$$
where $I_d\in\mathbb{R}^{d\times d}$ denotes the identity matrix. If $r \ll d$, the approximation of $\nu$ by $\nu_r$ consists essentially in replacing the high-dimensional likelihood $f$ by a function of fewer variables. Indeed, $x\mapsto g (P_r x) = g(x_r)$ depends only on the variable $x_r\in\text{Im}(P_r) \cong \mathbb{R}^r$ and is constant along $\text{Ker}(P_r)\cong \mathbb{R}^{d-r}$.  Yet $\nu_r$ cannot itself  be considered a low-dimensional distribution, since its support could be the same as that of $\nu$. 

We use the Kullback--Leibler divergence $\Dkl( \, \cdot \, || \, \cdot \, )$ to measure the dissimilarity between probability distributions. It is defined as
\begin{equation}\label{eq:defKL}
 \Dkl( \nu_1 || \nu_2 ) = \int \log \left( \frac{\mathrm{d}\nu_1}{\mathrm{d}\nu_2} \right)  \mathrm{d}\nu_1 ,
\end{equation}
for any probability distributions $\nu_1$ and $\nu_2$ such that $\nu_1 $ is absolutely continuous with respect to $ \nu_2 $, and $\Dkl( \nu_1 || \nu_2 )=\infty$ otherwise. 
Given a prescribed tolerance $\varepsilon\geq 0$, our goal is to build an approximation $\nu_r$ of $\nu$ of the form \eqref{eq:nuTilde} such that
\begin{equation}\label{eq:epsilon_proximality}
 \Dkl( \nu || \nu_r ) \leq \varepsilon .
\end{equation}
Of course if $P_r$ is the identity matrix and $g=f$, the distribution $\nu_r$ is exactly the posterior $\nu$, so that \eqref{eq:epsilon_proximality} is trivially satisfied. In that case, the rank of $P_r=I_d$ is $d$ and there is no dimension reduction. The goal of this section is to give sufficient conditions under which we can build an approximation $\nu_r$ such that \eqref{eq:epsilon_proximality} holds with a rank $r=r(\varepsilon)$ that is significantly smaller than $d$.

\begin{remark}
 Instead of approximating $\nu$ by $\nu_r$ as given by \eqref{eq:nuTilde}, we could have considered an approximation $\widetilde\nu_r$ of the form
 $$
  \frac{\mathrm{d}\widetilde\nu_r}{\mathrm{d}\mu} \propto h \circ A_r,
 $$
 with $h:\mathbb{R}^r\rightarrow \mathbb{R}^+$ a Borel function and $A_r\in\mathbb{R}^{r\times d}$ a matrix of full row rank. Functions of the form $x\mapsto h(A_rx)$ are a particular kind of \emph{ridge function} \cite{pinkus2015ridge}. They are used in various domains to approximate multivariate functions; see \cite{cohen2012capturing,fornasier2012learning,tyagi2014learning}, for example. However, any function of the form $h\circ A_r$ can be written as $g\circ P_r$ and vice-versa, where $g:\text{Im}(P_r)\rightarrow\mathbb{R}^+$ is some Borel function and $P_r$ is some rank-$r$ projector. More precisely, for any $r\leq d$ we have
 \begin{align*}\label{eq:equivalenceArPr}
  &\left\{ 
  h \circ A_r 
  \Big|
  \begin{array}{l}
  h:\mathbb{R}^r \rightarrow \mathbb{R}^+, \text{Borel function}\\
  A_r \in\mathbb{R}^{r\times d}, \text{rank-$r$ matrix}
  \end{array}
  \right\} \\
  =&
  \left\{ 
  g \circ P_r 
  \Big|
  \begin{array}{l}
  g:\text{Im}(P_r) \rightarrow \mathbb{R}^+ , \text{Borel function}\\
  P_r \in\mathbb{R}^{d\times d}, \text{rank-$r$ projector}
  \end{array}
  \right\}.
 \end{align*}
 This means that approximating $\nu$ by $\nu_r$ or by $\widetilde\nu_r$ are essentially the same problem: we do not gain approximation power in using $\widetilde\nu_r$ instead of $\nu_r$.

\end{remark}

\subsection{Optimal profile function}\label{sec:2.1}

To begin our analysis, we assume that the projector $P_r$ is given (fixed); we will address the problem of constructing $P_r$ later in Section~\ref{sec:ControlledApprox}.
In this subsection, we characterize the \emph{optimal} profile function $g^*$ for a given $P_r$, defined as a minimizer of $g\mapsto \Dkl( \nu || \nu_r )$ over the set of positive and measurable functions. This function $g^*$ yields an optimal probability measure $\nu_r^*$ defined as $\frac{\mathrm{d}\nu_r^*}{\mathrm{d}\mu}\propto g^*\circ P_r$. We will also show that, for any probability measure $\nu_r$ with $\frac{\mathrm{d}\nu_r}{\mathrm{d}\mu}\propto g\circ P_r$, the Kullback-Leibler divergence can then be decomposed as $\Dkl( \nu || \nu_r ) = \Dkl( \nu || \nu_r^* ) + \Dkl( \nu_r^* || \nu_r )$.

Let us denote by $\sigma(P_r)$ the $\sigma$-algebra generated by $P_r$. It is defined by
$
 \sigma(P_r) = \{ P_r^{-1} B ~|~ B \in \mathcal{B}(\mathbb{R}^d) \},
$
where $P_r^{-1} B = \{ x\in\mathbb{R}^d : P_r x \in B \}$ denotes the pre-image of $B\in\mathcal{B}(\mathbb{R}^d)$ under $P_r$.
The following lemma corresponds to the Doob--Dynkin lemma; see, for example, Lemma 1.13 in \cite{kallenberg1997foundations}. It states that, for any projector $P_r$, the set of all functions of the form $g\circ P_r$ for some measurable function $g$ is exactly the set of all $\sigma(P_r)$-measurable functions.

\begin{lemma}\label{cor:Doob}
 Let $P_r\in\mathbb{R}^{d\times d}$ be a projector. Given a measurable function $g$ defined on $\mathbb{R}^d$, the function $h = g \circ P_r $ is $\sigma(P_r)$-measurable. Conversely, given a $\sigma(P_r)$-measurable function $h$, there exists a measurable function $g$ defined on $\mathbb{R}^d$ such that $h = g \circ P_r $.
\end{lemma}

We denote by $\mathbb{E}_\mu(f|\sigma(P_r)): \mathbb{R}^d \rightarrow\mathbb{R}$ the conditional expectation of $f$ given $\sigma(P_r)$ under the distribution $\mu$. It is the unique $\sigma(P_r)$-measurable function that satisfies
\begin{equation}\label{eq:CondExp_VarForm}
 \int \mathbb{E}_\mu(f|\sigma(P_r)) \, h \,\mathrm{d}\mu
 = 
 \int f \, h \,\mathrm{d}\mu ,
\end{equation}
for any $\sigma(P_r)$-measurable function $h:\mathbb{R}^d\rightarrow \mathbb{R}$. We consider the probability distribution $\nu_r^*$ defined by
\begin{equation*}\label{eq:NuStar}
 \frac{\mathrm{d}\nu_r^*}{\mathrm{d}\mu} \propto \mathbb{E}_\mu(f|\sigma(P_r)).
\end{equation*}
This distribution is well defined because $\mathbb{E}_\mu(f|\sigma(P_r))$ is a positive function (the conditional expectation preserves the sign) and because $\int \mathbb{E}_\mu(f|\sigma(P_r)) \,\mathrm{d}\mu = \int f\, \mathrm{d}\mu <\infty$ (by letting $h=1$ in \eqref{eq:CondExp_VarForm}). By the definition of the Kullback--Leibler divergence \eqref{eq:defKL}, we can write
\begin{align}
 \Dkl(\nu || \nu_r ) - \Dkl(\nu || \nu_r^* ) 
  &= \int \frac{f}{Z_1} \log \frac{f/Z_1}{g\circ P_r / Z_2} ~\mathrm{d}\mu - \int \frac{f}{Z_1} \log \frac{f/Z_1 }{\mathbb{E}_\mu(f|\sigma(P_r)) / Z_1} ~\mathrm{d}\mu  \nonumber\\
  &= \int \frac{f}{Z_1} \log \frac{\mathbb{E}_\mu(f|\sigma(P_r))/Z_1}{g\circ P_r/Z_2} ~\mathrm{d}\mu  \label{eq:tmp2863}\\
  &= \int \frac{\mathbb{E}_\mu(f|\sigma(P_r))}{Z_1} \log \frac{\mathbb{E}_\mu(f|\sigma(P_r))/Z_1}{g\circ P_r/Z_2} ~\mathrm{d}\mu \label{eq:tmp2866}\\
  &= \Dkl(\nu_r^* || \nu_r ), \label{eq:pythagoras}
\end{align}
where $Z_1$ and $Z_2$ are two normalizing constants defined by $Z_1 = \int f ~\mathrm{d}\mu = \int \mathbb{E}_\mu(f|\sigma(P_r)) \mathrm{d}\mu$ and $Z_2 = \int g\circ P_r ~\mathrm{d}\mu$. 
To go from \eqref{eq:tmp2863} to \eqref{eq:tmp2866}, we used relation \eqref{eq:CondExp_VarForm} with $ h = \log \frac{(Z_1)^{-1}\mathbb{E}_\mu(f|\sigma(P_r))}{(Z_2)^{-1}g\circ P_r} $, which is a $\sigma(P_r)$-measurable function.
By \eqref{eq:pythagoras} we have
$$
 \Dkl(\nu || \nu_r ) \geq \Dkl(\nu || \nu_r^* ) ,
 \quad \text{with} \quad 
 \left\{
 \begin{array}{l}
  \frac{\mathrm{d} \nu_r}{\mathrm{d}\mu} \propto g \circ P_r ,\\
  \frac{\mathrm{d}\nu_r^*}{\mathrm{d}\mu} \propto \mathbb{E}_\mu(f|\sigma(P_r)) ,
 \end{array}
 \right.
$$
for any measurable function $g$. From this inequality we deduce that any function $g^*$ which satisfies $g^* \circ P_r = \mathbb{E}_\mu(f|\sigma(P_r))$ is a minimizer of $g\mapsto \Dkl(\nu || \nu_r )$. By Lemma \ref{cor:Doob} such a function $g^*$ exists because $\mathbb{E}_\mu(f|\sigma(P_r))$ is a $\sigma(P_r)$-measurable function.

\begin{remark}
The conditional expectation $\mathbb{E}_\mu(f|\sigma(P_r))$ is known to be the best approximation of a function $f$ with respect to the $L^2_\mu$-norm, meaning that it minimizes $h\mapsto \int (f-h)^2\mathrm{d}\mu$ among all $\sigma(P_r)$-measurable functions $h$. Here we showed that it is also optimal with respect to the Kullback--Leibler divergence. As pointed out in \cite{banerjee2005optimality}, $\mathbb{E}_\mu(f|\sigma(P_r))$ is also an optimal approximation of $f$ with respect to the class of \textit{expected Bregman divergences}, which includes the Kullback--Leibler divergence and the $L^2_\mu$-norm distance.
\end{remark}

The following proposition recalls the well-known analytical expression for the conditional expectation. For the sake of completeness, the proof is given in Section \ref{proof:ExplicitCondExp}.

\begin{proposition}[Explicit expression for the conditional expectation]\label{prop:ExplicitCondExp}
 Let $\mu$ be a probability distribution on $\mathbb{R}^d$ which admits a Lebesgue density $\rho$. Given a rank-$r$ projector $P_r$, we denote by $U_\perp\in\mathbb{R}^{d\times (d-r)}$ a matrix whose columns form a basis of $\Ker(P_r)$. Let $p_\perp( \, \cdot \, | P_r x )$ be the conditional probability density on $\mathbb{R}^{d-r}$ defined by
 \begin{equation}\label{eq:defCondDensity}
  p_\perp(\xi_\perp| P_r x ) 
 = \frac{\rho( P_r x + U_\perp \xi_\perp  )}{\int_{\mathbb{R}^{d-r}} \rho( P_r x + U_\perp \xi_\perp'  ) \mathrm{d}\xi_\perp'},
 \end{equation}
 for all $\xi_\perp\in\mathbb{R}^{d-r}$ and any $x\in\mathbb{R}^d$, with the convention $p_\perp(\xi_\perp| P_r x ) =0$ whenever the denominator of \eqref{eq:defCondDensity} is zero.
 Then, for any measurable function $f$, the conditional expectation $\mathbb{E}_\mu(f|\sigma(P_r))$ can be written as
 \begin{equation}\label{eq:CondExpExplicit}
   \mathbb{E}_\mu(f|\sigma(P_r)) : x \mapsto \int_{\mathbb{R}^{d-r}} f(P_r x + U_\perp \xi_\perp) ~p_\perp(\xi_\perp| P_r x ) \, \mathrm{d}\xi_\perp.
 \end{equation}
\end{proposition}

We conclude this section with a remarkable property of the conditional expectation $\mathbb{E}_\mu(f|\sigma(P_r))$. Let $Q_r$ be any projector whose kernel is the same as that of $P_r$. Then the relations $P_r = P_r \circ Q_r$ and $Q_r = Q_r \circ P_r$ hold; see, for instance, the proof of Proposition 2.2 in \cite{zahm2020gradient}. 
Then Lemma \ref{cor:Doob} ensures that any $\sigma(P_r)$-measurable function is $\sigma(Q_r)$-measurable and vice-versa. In other words, being $\sigma(P_r)$-measurable or $\sigma(Q_r)$-measurable are equivalent such that, by definition of the conditional expectation \eqref{eq:CondExp_VarForm}, we have
\begin{equation*}\label{eq:ImageInvariance}
 \text{Ker}(P_r) = \text{Ker}(Q_r)
 \quad\Rightarrow\quad
 \mathbb{E}_\mu(f|\sigma(P_r)) = \mathbb{E}_\mu(f|\sigma(Q_r)).
\end{equation*}
In other words, the conditional expectation $\mathbb{E}_\mu(f|\sigma(P_r))$ is invariant with respect to the image of $P_r$---so that the error $\Dkl( \nu || \nu_r^* )$, seen as a function of the projector $P_r$, is actually only a function of $\text{Ker}(P_r)$.
In the context of Bayesian inference, this property means that when the optimal profile function $g^*$ is used in \eqref{eq:nuTilde}, the important feature of $P_r$ to be discovered is its kernel, and not its image.  Thus reducing the dimension of a Bayesian inverse problem consists in identifying the subspace $\text{Ker}(P_r)$ on which the data are \textit{not informative}, meaning the subspace where the data do not modify the prior knowledge one has about the parameter.

\subsection{A controlled approximation}\label{sec:ControlledApprox}

In this section we show how to build a projector $P_r$ with a sufficiently large rank so that $\nu_r^*$, defined by $\frac{\mathrm{d}\nu_r^*}{\mathrm{d}\mu} \propto \mathbb{E}_\mu(f|\sigma(P_r))$, satisfies
$
 \Dkl( \nu || \nu_r^* ) \leq \varepsilon
$
for some prescribed tolerance $\varepsilon$. We argue that under some conditions on the likelihood $f$, the rank of $P_r$ can be significantly smaller than $d$. 
\\

We make the following assumption on the prior distribution $\mu$.
Here, the notation $A \succeq B$ means that the matrix $A-B$ is positive semidefinite.
\begin{assumption}\label{assu:convexityOfV}
 The distribution $\mu$ is supported on a convex set $\mathrm{supp}(\mu)\subseteq\mathbb{R}^d$ and admits a Lebesgue density $\rho$ such that $\rho \propto \exp( - V - \Psi)$, where $V$ and $\Psi$ are two functions defined on $\mathrm{supp}(\mu)$ which satisfy the following two properties.
 \begin{itemize}
  \item $V$ is twice continuously differentiable and there exists a symmetric positive definite matrix $\Gamma\in\mathbb{R}^{d\times d}$ such that for all $x\in \mathrm{supp}(\mu)$ we have
  \begin{equation}\label{eq:convexityOfV}
   \nabla^2 V(x) \succeq \Gamma.
  \end{equation}
  
 \item $\Psi$ is a bounded function over $\mathrm{supp}(\mu)$ such that
 \begin{equation}\label{eq:kappa}
  \exp( \sup \Psi - \inf \Psi ) \leq \kappa ,
 \end{equation}
 for some $\kappa \geq 1$.
 \end{itemize}
\end{assumption}

Let us make some comments on this assumption. First consider the case $\kappa=1$, which means that $\Psi$ is a constant function so that $\rho\propto \exp(-V-\Psi)\propto \exp(-V)$. Assumption \eqref{eq:convexityOfV} corresponds to a strong convexity property of the function $V$. Intuitively, it means that $V$ is ``at least quadratically convex.'' Now, when $\kappa > 1$, Assumption \eqref{eq:kappa} means that $c \exp(-V(x)) \leq \rho(x) \leq C \exp(-V(x))$ holds for all $x\in \mathrm{supp}(\mu)$, where $c>0$ and $C<\infty$ are two constants that satisfy $C/c \leq \kappa <\infty$.
We note that if Assumption \ref{assu:convexityOfV} is satisfied, then $\kappa$ and $\Gamma$ are not unique. Indeed, if $\kappa$ and $\Gamma$ satisfy \eqref{eq:convexityOfV} and \eqref{eq:kappa}, then any $\kappa'\geq \kappa$ and symmetric positive definite matrix $\Gamma'\preceq \Gamma$ also satisfy \eqref{eq:convexityOfV} and \eqref{eq:kappa}. 
In the results that follow, tighter bounds will be obtained by having the smallest $\kappa$ and the largest $\Gamma$, but this optimality can be difficult to verify in general.

\begin{example}[Gaussian distribution]\label{Ex:GaussianPrior}
 Any Gaussian prior $\mu =\mathcal{N}( m,\Sigma )$ with mean $m\in\mathbb{R}^d$ and invertible covariance matrix $\Sigma\in\mathbb{R}^{d\times d}$ satisfies Assumption \ref{assu:convexityOfV} with $\mathrm{supp}(\mu)=\mathbb{R}^d$, $\kappa=1$ and $\Gamma = \Sigma^{-1}$. Indeed, the density $\rho$ associated to $\mathcal{N}( m,\Sigma )$ is such that $\rho \propto \exp (- V)$ where the function $V:x\mapsto \frac{1}{2}(x-m)^T\Sigma^{-1}(x-m)$ satisfies $\nabla^2 V(x)=\Gamma$ for all $x\in \mathrm{supp}(\mu)$.
\end{example}

\begin{example}[Gaussian mixture]
 Let
 $
  \mu \propto \sum_{i=1}^n \alpha_i \, \mathcal{N}(m_i,\Sigma_i)
 $
 with $\alpha_1>0$ and $0 \prec \Sigma_1^{-1} \prec \Sigma_i^{-1}$ 
for all $1<i\leq n$, where $A\prec B$ means that $B-A$ is a positive definite matrix. The density $\rho$ of the Gaussian mixture $\mu$ can be written as follows,
 $$
  \rho
  \propto \sum_{i=1}^n \alpha_i \frac{\exp(-V_i) }{Z_i}
  = \exp(-V_1) \Big(\frac{\alpha_1}{Z_1} + \sum_{i=2}^n \frac{\alpha_i }{Z_i}\exp(-V_i + V_1 ) \Big) ,
 $$
 where $V_i: x\mapsto \frac{1}{2}(x-m_i)^T\Sigma_i^{-1}(x-m_i) $ and $Z_i = \int \exp (-V_i) \mathrm{d}x$. For any $1<i\leq n$ the function $-V_i + V_1$ is quadratic and, by assumption, its Hessian satisfies $- \Sigma_i^{-1} + \Sigma_1^{-1} \prec 0$. This ensures that $-V_i + V_1$ is bounded from above, and so is the function $\Psi = \log(\frac{\alpha_1}{Z_1} + \sum_{i=2}^n \frac{\alpha_i}{Z_i} \exp(-V_i + V_1 ) )$. Furthermore, the relation $\Psi(x) \geq \log(\frac{\alpha_1}{Z_1})>-\infty$ holds for all $x\in\mathbb{R}^d$, which means that $\Psi$ is bounded from below. As a consequence we have $\rho\propto \exp(-V_1-\Psi)$ so that $\mu$ satisfies Assumption \ref{assu:convexityOfV} with $\Gamma = \Sigma_1^{-1}$ and for some finite $\kappa$.
\end{example}

\begin{example}[Uniform distribution]\label{Ex:UniformPrior}
 The uniform measure $\mu$ on a convex and bounded domain $\mathcal{D}\subset\mathbb{R}^d$ satisfies Assumption \ref{assu:convexityOfV} with 
 $$ 
  \Gamma = \frac{4}{\text{diam}(\mathcal{D})^2} I_d, \text{ and}\quad\kappa = \exp(1),
 $$
 where $\text{diam}(\mathcal{D})=\max\{\|x-y\|_2 : x,y\in \mathcal{D}\}$. To show this, let us first denote by $\mathcal{B}(m,r)$ the minimum enclosing ball of $\mathcal{D}$, that is, the ball with minimal radius $r\geq0$ such that $\mathcal{D}\subset\mathcal{B}(m,r)$. 
 By Jung's theorem \cite{jung1901ueber}, we have $r\leq\text{diam}(\mathcal{D})\sqrt{d/(2(d+1))}$ which yields $r\leq \text{diam}(\mathcal{D})/\sqrt{2}$.
 Second, observe that the density $\rho\propto 1_{\mathcal{D}}$ of $\mu$ can be written as
 $\rho\propto\exp(-V-\Psi)$ where $V(x)=\|x-m\|^2/r^2$ and $\Psi = -V$. To conclude, we recognize that $\nabla^2V(x)\succeq\frac{4}{\text{diam}(\mathcal{D})^2}I_d$ and $\exp(\sup\Psi - \inf\Psi ) \leq \exp(\sup_{x\in\mathcal{B}(m,r)}\|x-m\|^2/r^2)=\exp(1)$.
\end{example}

\begin{counterexample}[Laplace distribution]
 The one-dimensional Laplace distribution $\mu$, with $\text{supp}(\mu)=\mathbb{R}$ and density $\rho(x)\propto \exp(-|x|)$, does not satisfy Assumption \ref{assu:convexityOfV}. To verify this, suppose that we can express the Laplace density in the form $\rho(x)\propto \exp(-V(x)-\Psi(x))$, where $V$ and $\Psi$ satisfy \eqref{eq:convexityOfV} and \eqref{eq:kappa}, respectively. Then we have $V(x)+\Psi(x) = c+|x|$ for some $c\in\mathbb{R}$. By \eqref{eq:convexityOfV} we have $V''(x) \geq \Gamma$ for some $\Gamma > 0$, which implies $V(x)\geq \tfrac{\Gamma}{2}x^2+\beta x + \gamma$ for some $\beta,\gamma\in\mathbb{R}$. 
 This leads to $\Psi(x) \leq c+|x| -\tfrac{\Gamma}{2}x^2-\beta x - \gamma$, which contradicts condition \eqref{eq:kappa} requiring that $\Psi$ be bounded.
\end{counterexample}

Assumption \ref{assu:convexityOfV} provides sufficient conditions for the prior distribution $\mu$ to satisfy the following \textit{logarithmic Sobolev inequality:}
\begin{equation}\label{eq:logSob}
 \int h^2 \log \frac{h^2}{\int h^2 ~\mathrm{d}\mu} ~\mathrm{d}\mu \leq 2\kappa \int \| \nabla h \|_{\Gamma^{-1}}^2 ~\mathrm{d}\mu,
\end{equation}
for any continuously differentiable function $h:\text{supp}(\mu)\rightarrow\mathbb{R}$, where $\| \cdot \|_{\Gamma^{-1}}$ denotes the norm on $\mathbb{R}^d$ such that $\|x\|_{\Gamma^{-1}}^2 = x^T\Gamma^{-1}x$ for all $x\in\mathbb{R}^d$. Logarithmic Sobolev inequalities are a fundamental tool in the analysis of Markov semigroups \cite{bakry2013analysis} and in concentration inequalities; see \cite{boucheron2013concentration}. The inequality \eqref{eq:logSob} results from the Bakry--\'Emery theorem \cite{bakry1985diffusions,bobkov2000brunn,otto2000generalization}, which uses Assumption  \eqref{eq:convexityOfV}, and from the Holley--Stroock perturbation lemma \cite{holley1987logarithmic}, which uses Assumption  \eqref{eq:kappa}. 
More precisely, Proposition 3.1 in \cite{bobkov2000brunn} states that in the case $\kappa=1$ (i.e., $\Psi$ is constant), Assumption \eqref{eq:convexityOfV} is sufficient to have \eqref{eq:logSob}. Then, following the lines of the original proof in \cite{holley1987logarithmic} (see also the proof of Theorem 1.9 in \cite{gozlan2013characterization}, for instance), we have that \eqref{eq:logSob} still holds when $\Psi$ is such that $\kappa>1$. 

The logarithmic Sobolev inequality \eqref{eq:logSob} permits control of the KL divergence from prior to posterior, $\Dkl( \nu || \mu )$, provided that $f$ is sufficiently regular. Indeed, by letting $h=\sqrt{f/(\int f \mathrm{d}\mu)}$ in \eqref{eq:logSob}, we obtain
\begin{equation*}\label{eq:KL_bound_init}
 \Dkl( \nu || \mu ) \leq  \frac{\kappa}{2} \int \| \nabla \log f \|_{\Gamma^{-1}}^2 ~\mathrm{d}\nu.   
\end{equation*}
In order to obtain a similar bound for the KL divergence $\Dkl( \nu || \nu_r^* )$, we need the following \textit{subspace logarithmic Sobolev inequalities}.

\begin{theorem}[Subspace logarithmic Sobolev inequality]\label{prop:subspaceLogSob}
 Let $\mu$ be a probability distribution on $\mathbb{R}^d$ that satisfies Assumption \ref{assu:convexityOfV} for some $\Gamma\in\mathbb{R}^{d\times d}$ and $\kappa\geq1$. Then the relation
 \begin{equation}\label{eq:subspaceLogSob}
  \int h^2 \log \frac{h^2}{\mathbb{E}_\mu( h^2 |\sigma(P_r))} ~\mathrm{d}\mu \leq 2 \kappa \int \| (I_d - P_r^T)\nabla h \|_{\Gamma^{-1}}^2 ~\mathrm{d}\mu,
 \end{equation}
 holds for any projector $P_r\in\mathbb{R}^{d\times d}$ and for any continuously differentiable function $h:\mathrm{supp}(\mu) \rightarrow \mathbb{R}$ such that $\int\|\nabla h\|_{\Gamma^{-1}}^2\mathrm{d}\mu <\infty$.
\end{theorem}

\begin{proof}
 The proof is given in Section \ref{proof:subspaceLogSob}.
\end{proof}

\begin{corollary}\label{cor:KL_bound}
 Let $\mu$ be a distribution which satisfies the subspace logarithmic Sobolev inequality \eqref{eq:subspaceLogSob} for some $\Gamma\in\mathbb{R}^{d\times d}$ and $\kappa\geq1$, and let $\nu$ be such that $\frac{\mathrm{d}\nu}{\mathrm{d}\mu} \propto f$ for some continuously differentiable function $f$ such that $\int\|\nabla \log f\|_{\Gamma^{-1}}^2\mathrm{d}\mu <\infty$. Then for any projector $P_r\in\mathbb{R}^{d\times d}$, the distribution $\nu_r^*$ such that $\frac{\mathrm{d}\nu_r^*}{\mathrm{d}\mu} \propto \mathbb{E}_\mu(f|\sigma(P_r))$ satisfies
 \begin{equation}
  \Dkl( \nu || \nu_r^* )
  \leq \frac{\kappa}{2} \mathcal{R}_{\Gamma}(P_r,H) \, , 
  \label{eq:KL_bound}
 \end{equation}
 where 
 \begin{equation}\label{eq:reconstructionError}
  \mathcal{R}_{\Gamma}(P_r,H) = \trace \big( \Gamma^{-1} (I_d-P_r)^T H (I_d-P_r) \big)
 \end{equation}and 
 \begin{equation} \label{eq:defH}
  H = \int (\nabla\log f)(\nabla\log f)^T  \mathrm{d}\nu.
\end{equation}
 
\end{corollary}

\begin{proof}
 The proof consists in rewriting \eqref{eq:subspaceLogSob} with $h=(f/Z)^{1/2}$, where $Z = \int \mathbb{E}_\mu (f|\sigma(P_r)) \, \mathrm{d}\mu = \int f \,\mathrm{d}\mu $. We have $\nabla h = \frac{1}{2}(f/Z)^{1/2}  \nabla \log f $ so that inequality \eqref{eq:subspaceLogSob} becomes
 $$
  \int \frac{f}{Z} \log \frac{f/Z}{\mathbb{E}_\mu( f |\sigma(P_r))/Z} ~\mathrm{d}\mu \leq \frac{\kappa}{2} \int \| (I_d - P_r^T)\nabla \log f \|_{\Gamma^{-1}}^2 \frac{f}{Z} \,\mathrm{d}\mu .
 $$
 and yields $\Dkl( \nu || \nu_r^* ) \leq \frac{\kappa}{2} \int \| (I_d - P_r^T) \nabla\log f \|_{\Gamma^{-1}}^2 \mathrm{d}\nu $. Finally,
  \begin{align}
\label{eq:reconstructionError_tmp}
  \mathcal{R}_{\Gamma}(P_r,H) &= \int \trace \big( \Gamma^{-1} (I_d - P_r^T) \nabla \log f  ( \nabla \log f)^T (I_d - P_r) \big) \mathrm{d}\nu \\
  &= \int \trace \big(  ( \nabla \log f)^T (I_d - P_r) \Gamma^{-1} (I_d - P_r^T) \nabla \log f \big) \mathrm{d}\nu \nonumber\\
  &= \int \| (I_d -  P_r^T)  \nabla  \log f  \|_{\Gamma^{-1}}^2 \mathrm{d}\nu , \nonumber 
 \end{align}
 holds. This concludes the proof.
\end{proof}

Corollary \ref{cor:KL_bound} provides an upper bound for $\Dkl( \nu || \nu_r^* )$ that is proportional to the \textit{reconstruction error} $\mathcal{R}_{\Gamma}(P_r,H) = \int \| \nabla  \log f -  P_r^T\nabla  \log f   \|_{\Gamma^{-1}}^2 \, \mathrm{d}\nu$, see \eqref{eq:reconstructionError_tmp}. Minimizing this bound over $P_r$ corresponds to finding the projector $P_r$ so that $P_r^T\nabla  \log f $ best approximates $\nabla \log f $ in the mean squared sense: this corresponds to  \emph{principal component analysis of $\nabla\log f$}; see \cite{reiss2020nonasymptotic}. 
The following proposition gives a closed form expression for a minimizer of $P_r \mapsto \mathcal{R}_{\Gamma}(P_r,H)$ over the set of the rank-$r$ projectors.
It corresponds to Proposition 2.6 in \cite{zahm2020gradient} where $\Sigma$ is replaced by $\Gamma^{-1}$. The proof relies on the Eckart--Young theorem, which provides the optimality properties of the singular value decomposition.

\begin{proposition}\label{prop:bestPr}
 Let $\Gamma\in\mathbb{R}^{d\times d}$ be a symmetric positive definite matrix and $H\in\mathbb{R}^{d\times d}$ be a symmetric positive semidefinite matrix. Denote by $(\lambda_i,v_i)\in\mathbb{R}_{\geq0}\times\mathbb{R}^d$ the $i$-th generalized eigenpair of the matrix pencil $(H,\Gamma)$, meaning $Hv_i = \lambda_i \Gamma v_i$ with $\lambda_i\geq\lambda_{i+1}$ and $\|v_i\|_{\Gamma}=1$, where $\| \cdot \|_{\Gamma}=\sqrt{(\cdot)^T\Gamma(\cdot)}$. For any $r \leq d$ we have
 \begin{equation}\label{eq:bestErrorPr}
  \min_{\substack{ P_r\in\mathbb{R}^{d\times d} \\ \text{rank-$r$ projector}}} 
  \mathcal{R}_{\Gamma}(P_r,H)
  = \sum_{i=r+1}^d \lambda_i \, .
 \end{equation}
 Furthermore, a solution to \eqref{eq:bestErrorPr} is the following $\Gamma$-orthogonal projector (i.e., satisfying  $x^T \Gamma x = \| x \|_{\Gamma}^2 = \| P_r x \|_{\Gamma}^2 + \| (I_d-P_r)x \|_{\Gamma}^2$ for all $x\in\mathbb{R}^d$) given by
 \begin{equation}\label{eq:BestPr}
  P_r = \Big(\sum_{i=1}^r v_iv_i^T \Big) \Gamma \,.
 \end{equation}

\end{proposition}

Corollary \ref{cor:KL_bound} and Proposition \ref{prop:bestPr} ensure that, provided $\mu$ satisfies the subspace logarithmic Sobolev inequality \eqref{eq:subspaceLogSob} and provided $P_r$ is defined as in \eqref{eq:BestPr}, the approximation $\nu_r^*$ of $\nu$ defined by $\frac{\mathrm{d}\nu_r^*}{\mathrm{d}\mu} \propto \mathbb{E}_\mu(f|\sigma(P_r))$ is such that 
$$
 \Dkl(\nu||\nu_r^*)\leq \frac{\kappa}{2} \sum_{i=r+1}^d \lambda_i,
$$
where $\lambda_i$ is the $i$-th generalized eigenvalue of $H = \int (\nabla\log f)(\nabla\log f)^T  \mathrm{d}\nu$. This relation holds for any $r\leq d$. Then for any $\varepsilon\geq 0$, the choice $r=r(\varepsilon) \coloneqq \min\{ r :  \frac{\kappa}{2} \sum_{i=r+1}^d \lambda_i \leq \varepsilon \}$ is sufficient to obtain $\Dkl(\nu||\nu_r^*)\leq \varepsilon$. Observe that a strong decay in the generalized eigenvalues of $H$ ensures that $r(\varepsilon)\ll d$.
In particular, if $H$ is rank deficient, we have $\lambda_i=0$ for all $i>\text{rank}(H)$ so that $r \geq \text{rank}(H)$ implies $\Dkl( \nu || \nu_r^* ) = 0$. The spectrum of the matrix $H$ reveals the low effective dimensionality of the posterior distribution: a strong decay in the generalized spectrum of $H$, or certainly a rank deficiency of $H$, ensures that there exists an approximation $\nu_r^*$ of $\nu$ such that $\Dkl(\nu||\nu_r^*)\leq \varepsilon$ with small $r(\varepsilon)\ll d$.

\begin{remark}[Prior-based dimension reduction]\label{sec:PriorBasedDimRed}
For any projector $P_r$ we can write
\begin{align*}
 \mathcal{R}_{\Gamma}(P_r,H)
 &= \trace( \Gamma^{-1} (I_d-P_r^T) H (I_d-P_r) ) \\
 &= \trace( H (I_d-P_r) \Gamma^{-1} (I_d-P_r^T) )  \\
 &\leq  \|H\| \trace( (I_d-P_r) \Gamma^{-1} (I_d-P_r^T) ) \\
 &= \|H\| \mathcal{R}_{I_d}(P_r^T,\Gamma^{-1}) ,
\end{align*}
where $\|H\|=\sup\{ |x^T Hx| : x\in\mathbb{R}^d, \|x\|_2\leq1 \}$ is the spectral norm of $H$. 
In the case of a Gaussian prior, $\mu=\mathcal{N}(m,\Sigma)$, we have $\Gamma=\Sigma^{-1}$ (see Example \ref{Ex:GaussianPrior}) and so Proposition \ref{prop:bestPr} ensures that the minimizer of $P_r\mapsto \mathcal{R}_{I_d}(P_r^T,\Gamma^{-1}) = \mathcal{R}_{I_d}(P_r^T,\Sigma)$ is the orthogonal projector onto the leading eigenspace of the prior covariance $\Sigma$. This projector corresponds to the truncated Karhunen--Loève decomposition; see Section \ref{sec:KarhunenLoeve} for a detailed discussion.

\end{remark}

\subsection{An illustrative example}\label{sec:SimpleExample}

To illustrate how sharp the bound given by Corollary \ref{cor:KL_bound} can be, we consider a simple example for which the Kullback--Leibler divergence $\Dkl( \nu || \nu_r^* )$ is computable in closed form. This allows a comparison of the error $\Dkl( \nu || \nu_r^* ) $ and its upper bound $\frac{\kappa}{2}\mathcal{R}_{\Gamma}(P_r,H)$. 
~\\

Assume that the prior $\mu$ is the standard Gaussian distribution $\mathcal{N}(0,I_d)$ and let the likelihood function be given by
$$
 f: x\mapsto  \exp( -\frac{1}{2} x^T A x ),
$$
where $A\in\mathbb{R}^{d\times d}$ is a symmetric positive semidefinite matrix. The Lebesgue density of $\mu$ is $\rho(x) \propto \exp( -\frac{1}{2} x^T x )$, so that $\mu$ satisfies Assumption \ref{assu:convexityOfV} with $\Gamma=I_d$ and $\kappa=1$. The posterior $\nu$ defined by $\frac{\mathrm{d}\nu}{\mathrm{d}\mu}\propto f$ is also Gaussian with zero mean and covariance $ \Sigma = ( I_d + A )^{-1}$. In this setting, the matrix $H$ defined by \eqref{eq:defH} can be written as follows
$$
 H = \int (\nabla\log f)(\nabla\log f)^T \mathrm{d}\nu = \int (Ax)(Ax)^T \nu(\mathrm{d}x) = A ( I_d + A )^{-1} A.
$$
Consider the generalized eigenvalue problem $Hv_i=\lambda_i \Gamma v_i$ which, since $\Gamma=I_d$, is simply the standard eigenvalue problem $Hv_i=\lambda_i v_i$. Notice that $H$ is a rational function in $A$, so that $H$ and $A$ have the same eigenvectors and their eigenvalues satisfy the relation
$$
 \lambda_i = \frac{\alpha_i^2}{1+\alpha_i},
$$
for all $1\leq i \leq d$, where $\alpha_i$ is the $i$-th largest eigenvalue of $A$. According to Proposition \ref{prop:bestPr}, a projector minimizing $P_r\mapsto \mathcal{R}_{\Gamma}(P_r,H)$ over the set of rank-$r$ projectors is $P_r = \sum_{i=1}^r v_i v_i^T$. Corollary \ref{cor:KL_bound} ensures that the distribution $\nu_r^*$ defined by $\frac{\mathrm{d}\nu_r^*}{\mathrm{d}\mu}\propto \mathbb{E}_\mu(f|\sigma(P_r))$ is such that
\begin{equation}\label{eq:tmp065872}
 \Dkl( \nu || \nu_r^* )  \leq \frac{1}{2} \sum_{i=r+1}^d \lambda_i 
 = \frac{1}{2} \sum_{i=r+1}^d \frac{\alpha_i^2}{1+\alpha_i}.
\end{equation}
To analyze the sharpness of this inequality, we now compute $\Dkl( \nu || \nu_r^* )$. Using Proposition \ref{prop:ExplicitCondExp}, we can express the conditional expectation $\mathbb{E}_\mu(f|\sigma(P_r))$ as follows:
\begin{equation}\label{eq:CondExpGaussianCase}
 \mathbb{E}_\mu(f|\sigma(P_r)) : x\mapsto (\text{det}(\Sigma_r^{-1}\Sigma) )^{1/2} \exp( -\frac{1}{2}x^T P_rAP_r x ),
\end{equation}
where $\Sigma_r = (I_d+P_rAP_r)^{-1}$. Then $\nu_r^*$ such that $\frac{\mathrm{d}\nu_r^*}{\mathrm{d}\mu}\propto \mathbb{E}_\mu(f|\sigma(P_r))$ is Gaussian with zero mean and covariance $\Sigma_r$. 
The Kullback--Leibler divergence from $\nu_r^*=\mathcal{N}(0,\Sigma_r)$ to $\nu=\mathcal{N}(0,\Sigma)$ admits the following closed form expression \cite{kullback1997information}:
\begin{equation}\label{eq:tmp620934}
 \Dkl( \nu || \nu_r^* ) 
 = \frac{1}{2} \big( \trace( \Sigma_r^{-1} \Sigma  ) - \log (\text{det}(\Sigma_r^{-1} \Sigma)) -d \big).
\end{equation}
To continue the calculation, one needs the eigenvalues of $\Sigma_r^{-1} \Sigma$. Let $U$ be the orthogonal matrix containing the eigenvectors of $A$ and let $D=\text{diag}(\alpha_1,\hdots,\alpha_d)$ so that $A=U D U^T$. By construction, $P_rAP_r = U D_r U^T$ where $D_r=\text{diag}(\alpha_1,\hdots,\alpha_r,0,\hdots,0)$, so that
$
 \Sigma_r^{-1} \Sigma = (I_d+P_rAP_r) (I_d+A )^{-1}= U  (I_d+ D_r) (I_d+D )^{-1}U.
$
From this relation, we deduce that the $i$-th eigenvalue of $\Sigma_r^{-1} \Sigma$ is $1$ when $i\leq r$ and $(1+\alpha_i)^{-1}$ otherwise. Substituting these eigenvalues into \eqref{eq:tmp620934} yields
\begin{align*}
 \Dkl( \nu || \nu_r^* ) 
 &=\frac{1}{2} \Big( r + \sum_{i=r+1}^{d} \frac{1}{1+\alpha_i} + \sum_{i=r+1}^d \log(1+\alpha_i) -d \Big)  \\
 &= \frac{1}{2}  \sum_{i=r+1}^{d} \Big( \log(1+\alpha_i) -\frac{\alpha_i}{1+\alpha_i}  \Big).
\end{align*}

We now analyze the deficit in inequality \eqref{eq:tmp065872}. With a Taylor expansion as $\alpha_i$ goes to zero for all $i\geq r$, we can write
\begin{align*}
 \Dkl( \nu || \nu_r^* ) &= \frac{1}{4}  \sum_{i=r+1}^{d} \big( \alpha_i^2  + \mathcal{O}(\alpha_i^3)  \big) , \ \text{and}\\
 \frac{1}{2} \sum_{i=r+1}^d \frac{\alpha_i^2}{1+\alpha_i} &= \frac{1}{2} \sum_{i=r+1}^d \big(  \alpha_i^2 +\mathcal{O}(\alpha_i^3) \big) .
\end{align*}
If the function $f$ is nearly constant ($\alpha_i\approx 0$) along the subspace $\text{span}\{v_{r+1},\hdots,v_d\}$, then the upper bound \eqref{eq:tmp065872} is close to $2\Dkl( \nu || \nu_r^* )$. 

In this particular example, the projector obtained by minimizing the upper bound in fact yields the \textit{optimal} approximation of $\nu$ in Kullback--Leibler divergence, for any given rank $r$. 
This can be shown by Theorem 2.3 in \cite{spantini2015optimal}.
Of course, minimizing the upper bound does not produce the optimal projector in general.

\section{Building the approximation}\label{sec:3}

In this section we propose and analyze algorithms for the numerical construction of a low-rank projector $P_r$ and of a profile function $g$, such that the distribution $\nu_r$ given by $\frac{\mathrm{d}\nu_r}{\mathrm{d}\mu}\propto g\circ P_r$ is a controlled approximation of the posterior distribution $\nu$. 
Recall that in the previous section we obtained the following decomposition of the Kullback--Leibler divergence:
\begin{equation}\label{eq:DecompKL}
 \Dkl( \nu || \nu_r ) \overset{\eqref{eq:pythagoras}}{=} \Dkl( \nu || \nu_r^* ) + \Dkl( \nu_r^* || \nu_r ),
\end{equation}
where $\frac{\mathrm{d} \nu_r^* }{\mathrm{d}\mu} \propto \mathbb{E}_\mu(f|\sigma(P_r))$ gives the optimal (but intractable) probability distribution $\nu_r^*$ we can obtain for a given $P_r$. The term $\Dkl( \nu || \nu_r^* )$ thus measures the error between the target $\nu$ and its best approximation for a given projector.
Using the material presented in Section~\ref{sec:ControlledApprox}, Section~\ref{sec:ConstructionOfTheProjector} below shows how to build $P_r$ so that the error $\Dkl( \nu || \nu_r^* )$---or, more precisely, an upper bound for this error---is arbitrarily small.
The second term $\Dkl( \nu_r^* || \nu_r )$ can be interpreted as a distance between $g\circ P_r$ and $\mathbb{E}_\mu(f|\sigma(P_r))$, so that the construction of $\nu_r$ consists essentially in approximating a conditional expectation. 
In Section~\ref{sec:ApproximationConditionalExpectation}, we propose an approximation $\nu_r=\widehat\nu_r$ that relies on a sample estimate $\widehat F_r$ of the conditional expectation $\mathbb{E}_\mu(f|\sigma(P_r))$.
In the end, the computational strategy we adopt here is to: (i) construct $P_r$ to control the first term; and (ii) given such a projector, build the function $g$ so that the second term is arbitrarily close to zero---or, at least, of the same order of magnitude as the first term.

\subsection{Construction of the projector}\label{sec:ConstructionOfTheProjector}

Let us recall some of the results of Section \ref{sec:ControlledApprox}. Under Assumption \ref{assu:convexityOfV} and provided that $f$ is sufficiently regular, Corollary \ref{cor:KL_bound} provides an upper bound for the first term $ \Dkl( \nu || \nu_r^* )$ in the decomposition \eqref{eq:DecompKL} by means of the reconstruction error $\mathcal{R}_{\Gamma}(P_r,H)=\trace(\Gamma^{-1}(I_d-P_r^T )H (I_d-P_r))$ where $H=\int(\nabla\log f)(\nabla\log f)^T \mathrm{d}\nu$.
This bound holds for any projector $P_r$. We denote by $P_r^*$ a rank-$r$ projector which minimizes the reconstruction error, meaning
\begin{equation}\label{eq:BestPr_2}
 \mathcal{R}_{\Gamma}(P_r^*, H)
 = 
 \min_{\substack{
  P_r\in\mathbb{R}^{d\times d} , \\
  \text{ rank-$r$ projector} }}
 \mathcal{R}_{\Gamma}(P_r,H).
\end{equation}
By Proposition \ref{prop:bestPr}, $P_r^*$ can be obtained by means of the generalized eigendecomposition of $H$. In practice, however, the matrix $H$ may be difficult to compute exactly because it requires computing a high-dimensional integral when $d\gg1$. Instead, we consider the rank-$r$ projector $\widehat P_r$ that is a minimizer of the \emph{approximate} reconstruction error $\mathcal{R}_{\Gamma}(P_r,\widehat H)$, i.e., 
\begin{equation}\label{eq:ApproxPr}
 \mathcal{R}_{\Gamma}(\widehat P_r, \widehat H)
 = 
 \min_{\substack{
  P_r\in\mathbb{R}^{d\times d} , \\
  \text{ rank-$r$ projector} }}
 \mathcal{R}_{\Gamma}(P_r,\widehat H),
\end{equation}
where $\widehat H $ is a Monte Carlo approximation of $H$ defined by
\begin{equation}\label{eq:defHK}
 \widehat H = \frac{1}{K} \sum_{i=1}^K \big( \nabla \log f(X_i) \big) \big( \nabla \log f( X_i ) \big)^T  .
\end{equation}
Here, for the purpose of theoretical analysis, we assume that $X_1,\hdots,X_K$ are independent draws from the posterior distribution $\nu$. Since drawing independent samples from the posterior is computationally infeasible for most problems, we will present in Section~\ref{sec:alg} an iterative algorithm that can be used to estimate $H$ in practice.

The minimization problem in \eqref{eq:ApproxPr} and \eqref{eq:defHK} corresponds to the principal component analysis (PCA) of the random vector $\nabla\log f(X)$, $X\sim \nu$, and its solution follows directly from the solution of a generalized eigenvalue problem; see \cite{kokiopoulou2011trace}. An important question is how large should $K$ be in order to control the reconstruction error $\mathcal{R}_{\Gamma}(\widehat P_r,H)$? We refer to \cite{blanchard2007statistical,reiss2020nonasymptotic} for recent progress in this direction. The following proposition gives a new sufficient condition on $K$ so that $\widehat P_r$ is a quasi-optimal solution to \eqref{eq:BestPr_2}. This result relies on concentration properties of sub-Gaussian random vectors; see \cite{boucheron2013concentration,vershynin_2012}. The proof is given in Section \ref{proof:ControlReconstructionError}.

\begin{proposition}\label{prop:ControlReconstructionError}
 
 Let $\mu,\nu$ be two probability distributions and $f$ be a continuously differentiable function.
 Assume there exists a constant $L\geq 0$ such that the random vector $\nabla \log f(X)$ with $X\sim \nu$ satisfies
 \begin{equation}\label{eq:SubgaussianNablaLogF}
  \| u^T  \nabla \log f(X) \|_{\psi_2} \leq L \, \sqrt{ u^T  H u },
 \end{equation}
 for any $u\in\mathbb{R}^d$, where $H = \int (\nabla \log f)(\nabla \log f)^T  \mathrm{d}\nu$. Here $\| \cdot \|_{\psi_2}$ denotes the sub-Gaussian norm, meaning $\| \xi \|_{\psi_2} = \sup_{k\geq1} k^{-1/2}\mathbb{E}( |\xi|^k  )^{1/k} $ for any real-valued random variable $\xi$. Let $\widehat H$ be a $K$-sample Monte Carlo approximation of $H$. For any $0< \delta ,\eta <1$, the condition
 \begin{equation}\label{eq:KlowerBound}
  K\geq \Omega \delta^{-2} L^4 \big(  \sqrt{\rank(H)} + \sqrt{\log(2\eta^{-1})}\big)^{2},
 \end{equation}
 for some absolute (numerical) constant $\Omega$, ensures that with probability at least $1-\eta$, the following relation holds for any projector $P_r$:
 \begin{equation}\label{eq:ControlReconstructionError}
  (1-\delta) ~\mathcal{R}_{\Gamma}(P_r, H) 
  \leq \mathcal{R}_{\Gamma}(P_r,\widehat H) 
  \leq (1+\delta) ~\mathcal{R}_{\Gamma}(P_r,H).
 \end{equation}
\end{proposition}

Assumption \eqref{eq:SubgaussianNablaLogF} means that $\nabla \log f(X)$ is a sub-Gaussian vector. Intuitively it says that the tails of the distribution of $\nabla \log f(X)$ are at most Gaussian. Together with the independence of $X_1,\hdots,X_K$, the sub-Gaussian property is an essential ingredient in the proof of Proposition \ref{prop:ControlReconstructionError}. We give now examples for which \eqref{eq:SubgaussianNablaLogF} is satisfied. 

\begin{example}[Bounded gradients]
 Assume $x\mapsto\nabla \log f(x)$ is uniformly bounded and assume $H$ is full rank. 
 Let $L<\infty$ be such that $\|\nabla \log f(x)\|_{H^{-1}} \leq L$ for all $x\in\mathbb{R}^d$.
 We have
 $$
  \sup_{u\in \mathbb{R}^d \backslash\{0\}}
  \frac{| u^T  \nabla \log f(x) |}{\sqrt{u^T  H u}}  
  =\|\nabla \log f(x)\|_{H^{-1}}
  \leq L
 $$
 for any $x\in\mathbb{R}^d$ so that $| u^T  \nabla \log f(x) |\leq L \sqrt{u^T  H u}$ for all $x,u\in\mathbb{R}^d$. This means that $| u^T  \nabla \log f(X) |$ is almost surely bounded by $L \sqrt{u^T  H u}$. Since the $\psi_2$-norm is bounded by the $L^{\infty}$-norm, we have $\|u^T  \nabla \log f(X)\|_{\psi_2} \leq L \sqrt{u^T  H u} $.  This shows that relation \eqref{eq:SubgaussianNablaLogF} holds true for all $u\in \mathbb{R}^d$.
\end{example}

\begin{example}[Unbounded gradients]
 We now consider an example where $\nabla \log f(X)$ satisfies \eqref{eq:SubgaussianNablaLogF} without being bounded. As in Section \ref{sec:SimpleExample}, let $\mu=\mathcal{N}(0,I_d)$ be the standard normal prior and let $f:x\mapsto \exp(-\frac{1}{2}x^T Ax)$ for some symmetric matrix $A\succeq0$. Notice that $\nabla \log f(x) = -Ax$ can be arbitrarily large when $\|x\|\rightarrow \infty$. We have $X\sim\nu=\mathcal{N}(0,(I_d+A)^{-1})$ and $H=A (I_d+A)^{-1} A$. For any $u\in\mathbb{R}^d$, the random variable $Z=u^T  \nabla \log f(X) = - u^T  AX$ is Gaussian with zero mean and variance $\sigma_Z^2 = u^T  A (I_d+A)^{-1} A u = u^T  H u$. Then $Z$ is sub-Gaussian and $\|Z\|_{\psi_2}\leq L \sigma_Z$ holds for some absolute constant $L$; see, for instance, Example 5.8 in \cite{vershynin_2012}. This means that \eqref{eq:SubgaussianNablaLogF} holds for all $u\in \mathbb{R}^d$ with a constant $L$ which does not depend on $A$. 
\end{example}

Proposition \ref{prop:ControlReconstructionError} gives a sufficient condition for \eqref{eq:ControlReconstructionError} to hold with high probability. This relation yields the following quasi-optimality result:
\begin{equation}\label{eq:06ejfow}
 \mathcal{R}_{\Gamma}(\widehat P_r,H) 
 \overset{\eqref{eq:ControlReconstructionError}}{\leq} \frac{1}{1-\delta} \mathcal{R}_{\Gamma}(\widehat P_r,\widehat H) 
 \overset{\eqref{eq:ApproxPr}}{\leq} \frac{1}{1-\delta} \mathcal{R}_{\Gamma}(P_r^*,\widehat H) 
 \overset{\eqref{eq:ControlReconstructionError}}{\leq} \frac{1+\delta}{1-\delta} \mathcal{R}_{\Gamma}(P_r^*, H) .
\end{equation}
Then, even though $\widehat P_r$ is not a minimizer of $P_r\mapsto \mathcal{R}_{\Gamma}(P_r, H)$, the reconstruction error  $\mathcal{R}_{\Gamma}(\widehat P_r,H)$ is no greater than an arbitrary multiplicative constant $\frac{1+\delta}{1-\delta} \geq 1$ times the minimum of the reconstruction error $\mathcal{R}_{\Gamma}(P_r^*,H)$. In particular $\mathcal{R}_{\Gamma}(P_r^*,H)=0$ implies $\mathcal{R}_{\Gamma}(\widehat P_r,H)=0$. 
Together with Corollary \ref{cor:KL_bound} and Proposition \ref{prop:bestPr}, the above quasi-optimality result allows us to deduce that, with probability at least $1 - \eta$, we have
\begin{equation}
 \Dkl( \nu || \nu_r^* )
 \overset{\eqref{eq:KL_bound}}{\leq} \frac{\kappa}{2} \mathcal{R}_{\Gamma}(\widehat P_r,H) 
 \overset{\eqref{eq:06ejfow}}{\leq} \frac{\kappa(1+\delta)}{2(1-\delta)} \mathcal{R}_{\Gamma}(P_r^*,H)
 \overset{\eqref{eq:bestErrorPr}}{=} \frac{\kappa(1+\delta)}{2(1-\delta)}\sum_{i=r+1}^d \lambda_i ,
\label{eq:boundKLoverall}
\end{equation}
where $\frac{\mathrm{d}\nu_r^*}{\mathrm{d}\mu}\propto \mathbb{E}_\mu(f|\sigma(P_r))$ with $P_r=\widehat P_r$. 
This relation ensures that, up to a multiplicative constant $\frac{1+\delta}{1-\delta}$, the bound on the error $\Dkl( \nu || \nu_r^* )$ decays with $r$ at the same rate as if we had used the projector $P_r=P_r^*$ that minimizes the true reconstruction error.

We note that condition \eqref{eq:KlowerBound} requires $K$ to be at least proportional to the rank of $H$. If $H$ is full rank and if $d\gg 1$, then $K$ ought to be large in order to satisfy \eqref{eq:KlowerBound}. In practice, though, we observe that suitable projectors can be obtained by choosing $K$ proportional to the rank $r$ of $\widehat P_r$ (which is PCA's ``rule of thumb''). Here there is a challenge in finding weaker conditions on $K$ that nonetheless allow one to obtain a quasi-optimality result. As in \cite{reiss2020nonasymptotic}, one should try to exploit some properties of $H$ such as a rapid spectral decay or a large spectral gap.

\subsection{Approximation of the conditional expectation}\label{sec:ApproximationConditionalExpectation}

In this section we assume that a rank-$r$ projector $P_r\in\mathbb{R}^{d\times d}$ is given. We consider the problem of approximating the conditional expectation $\mathbb{E}_\mu(f|\sigma(P_r))$. Recall that, using the notation of Proposition \ref{prop:ExplicitCondExp}, we have
\begin{equation}\label{eq:CondExpExplicit_2}
 \mathbb{E}_\mu(f|\sigma(P_r)) : x\mapsto \int_{\mathbb{R}^{d-r}} f(P_r x + U_\perp \xi_\perp) \,p_\perp(\xi_\perp| P_r x ) \,\mathrm{d}\xi_\perp .
\end{equation}
For any $x\in\mathbb{R}^d$, an approximation of $\mathbb{E}_\mu(f|\sigma(P_r))(x)$ can be obtained via a Monte Carlo estimate of the form $\frac{1}{M}\sum_{i=1}^M f\big(P_r x + U_\perp \Xi_i \big)$, where $\Xi_1,\ldots,\Xi_M \in\mathbb{R}^{d-r}$ are independent copies of the random vector $\Xi \sim p_\perp( \, \cdot  \, | P_r x )$. In general, the law of $\Xi$ depends on $x$ and we should redraw the samples $\Xi_1,\hdots,\Xi_M$ for each different value of $x$, which can be computationally demanding. 

Instead, we will consider an approximation of $\mathbb{E}_\mu(f|\sigma(P_r))$ of the form
\begin{equation}\label{eq:MonteCarloCond}
 \widehat F_r: x \mapsto \frac{1}{M}\sum_{i=1}^{M} f(P_r x + (I_d-P_r) Y_i ), 
\end{equation}
where $Y_1,\hdots,Y_M$ are independent samples from the prior $\mu$. The form of $\widehat{F}_r$ is motivated by situations where $X_r$ and $X_\perp$ are in fact independent under the prior, i.e., where $p_\perp (\cdot \vert P_r x) = p_\perp (\cdot)$. This occurs, for example, in the case of a Gaussian prior with covariance $\Gamma^{-1}$, since the projector $P_r$ is by construction $\Gamma$-orthogonal. In this case, $\widehat{F}_r$ is an unbiased estimator of $\mathbb{E}_\mu(f|\sigma(P_r))$.

From a computational perspective, the samples $Y_i$ will be drawn \emph{once} and reused for each evaluation of $x \mapsto \widehat F_r(x)$. We define $\widehat\nu_r$ to be the random probability distribution defined by
\begin{equation}\label{eq:defHatNuR}
 \frac{\mathrm{d}\widehat\nu_r}{\mathrm{d}\mu}\propto \widehat F_r.
\end{equation}
Notwithstanding the motivation above, the laws of $(I_d-P_r) Y_i$ and of $U_\perp \Xi_i$ are in general different, so that $\widehat F_r(x)$ is \emph{in general} a biased estimate of $\mathbb{E}_\mu(f|\sigma(P_r))(x)$. In this case we cannot hope for $\Dkl(\nu_r^* || \widehat\nu_r )$ to go to zero with $M$. In order to analyze this bias, let us introduce the probability distribution $\mu'$ such that
\begin{equation}\label{eq:defMuPrime}
 \int h ~\mathrm{d}\mu' = \mathbb{E}\big(h( P_r X + (I_d-P_r)Y )  \big),
\end{equation}
for any Borel function $h$, where $X$ and $Y$ are independent random variables with distribution $\mu$. Equation \eqref{eq:defMuPrime} is equivalent to saying that $\mu'$ is the probability distribution of the random vector $P_r X + (I_d-P_r)Y$. The following proposition is proven in Section \ref{proof:ExplicitCondExp_muPrime}.

\begin{proposition}\label{prop:ExplicitCondExp_muPrime}
 The conditional expectation $\mathbb{E}_{\mu'}(f|\sigma(P_r))$ is such that
 $$
  \mathbb{E}_{\mu'}(f|\sigma(P_r)) : x\mapsto \mathbb{E}\big( f(P_r x + (I_d-P_r)Y )  \big) .
 $$
\end{proposition}

By Proposition \ref{prop:ExplicitCondExp_muPrime} we can write $\mathbb{E}(\widehat F_r) = \mathbb{E}_{\mu'}(f|\sigma(P_r))$ which allows us to interpret $\widehat F_r$ as a Monte Carlo estimate of $\mathbb{E}_{\mu'}(f|\sigma(P_r))$. The following proposition gives a bound for the expectation of the error $\Dkl(\nu_r^*|| \widehat\nu_r)$. The proof is given in Section \ref{proof:DecompExpectKL}.

\begin{proposition}\label{prop:DecompExpectKL}
 Let $\nu_r'$ be the distribution such that $\frac{\mathrm{d}\nu_r'}{\mathrm{d}\mu} \propto \mathbb{E}_{\mu'}(f|\sigma(P_r))$ and assume that $f(x)>0$ for $\mu$-a.e.~$x$.
 We have
 \begin{align}\label{eq:DecompExpectKL}
  \mathbb{E}\big( \Dkl(\nu_r^*|| \widehat\nu_r) \big) 
  &\leq \Dkl(\nu_r^*|| \nu_r') 
  + \frac{1}{2} \, \mathbb{E} \int \Big(\frac{\widehat F_r - \mathbb{E}_{\mu'}(f|\sigma(P_r))}{\mathbb{E}_{\mu'}(f|\sigma(P_r))} \Big)^2 \mathrm{d}\nu_r^*  \nonumber \\ 
  & ~~~ + \mathbb{E}  \int \mathcal{O} \left( \Big( \frac{\widehat F_r(x) - \mathbb{E}_{\mu'}(f|\sigma(P_r))(x)}{\mathbb{E}_{\mu'}(f|\sigma(P_r))(x)} \Big)^3  \right)  \, \nu_r^* (\mathrm{d}x) .
 \end{align}
 
\end{proposition}

Proposition \ref{prop:DecompExpectKL} shows that, up to a third order term, the expectation of the error $\Dkl(\nu_r^*|| \widehat\nu_r)$ is bounded by the sum of two terms. The first term is the Kullback--Leibler divergence from $\nu_r'$ to $\nu_r^*$ and corresponds to the bias $\mathbb{E}(\widehat F_r) \neq \mathbb{E}_{\mu}(f|\sigma(P_r))$. The second term can be interpreted as a measure of the variance of $\widehat F_r$. Under some assumptions on the distribution $\mu$, the following proposition provides an upper bound for those two terms.

\begin{proposition}\label{prop:controlMC}
 
In addition to the assumptions of Proposition \ref{prop:DecompExpectKL}, assume that $\mu$ admits a Lebesgue density $\rho \propto \exp(-V -\Psi )$, where $\Psi$ is a bounded function such that $ \exp( \sup \Psi - \inf \Psi ) \leq \kappa$ and where $V:x \mapsto \frac{1}{2}\|x-m\|_{\Gamma}^2$ for some $m\in\mathbb{R}^d$ and some symmetric positive-definite matrix $\Gamma\in\mathbb{R}^{d\times d}$. Then, for any $\|\cdot\|_{\Gamma}$-orthogonal projector $P_r$, we have
 \begin{align}
  \Dkl(\nu_r^*||\nu_r') & \leq \kappa^7 (\kappa^2-1) \mathcal{E}_\Gamma(P_r,f), \label{eq:controlMC_Bais}
 \end{align}
and
 \begin{align}
  \frac{1}{2} \, \mathbb{E} \int \Big(\frac{\widehat F_r - \mathbb{E}_{\mu'}(f|\sigma(P_r))}{\mathbb{E}_{\mu'}(f|\sigma(P_r))} \Big)^2 \mathrm{d}\nu_r^* &  \leq \frac{\kappa^7}{2M} \mathcal{E}_\Gamma(P_r,f),  \label{eq:controlMC_MeanSquaredError}
 \end{align}
 where
 \begin{equation}\label{eq:defE}
  \mathcal{E}_\Gamma(P_r,f) = \int \|(I_d-P_r^T )\nabla \log f  \|_{\Gamma^{-1}}^2 \frac{f}{\mathbb{E}_{\mu}(f|\sigma(P_r))} \mathrm{d}\nu.
 \end{equation}
 
\end{proposition} 

The proof is given is Section \ref{proof:controlMC}. Proposition \ref{prop:controlMC} requires $\mu$ to be a bounded perturbation of a Gaussian distribution $\mathcal{N}(m,\Gamma^{-1})$. This is a stronger assumption than Assumption \ref{assu:convexityOfV}. 
Neglecting the third order term in \eqref{eq:DecompExpectKL}, Proposition \ref{prop:controlMC} allows one to bound $\mathbb{E}( \Dkl(\nu_r^*|| \widehat\nu_r) )$ by $(C_1 + \frac{C_2}{M} ) \mathcal{E}_\Gamma(P_r,f) $ where $C_1$ and $C_2$ are two constants which depend only on $\kappa$. If $\mu$ is Gaussian ($\kappa=1$) then $C_1 = 0$ and $C_2 = 1/2$. In that case, $\mathbb{E}( \Dkl(\nu_r^*|| \widehat\nu_r) )$ goes to zero with $M$. This is not surprising because $\mu'=\mu$ holds whenever $\mu$ is Gaussian,\footnote{Recall that $P_r$ is an orthogonal projector with respect to the norm induced by the precision matrix of $\mu$.} so that $\nu_r^* = \nu_r'$ and hence $\Dkl(\nu_r^*||\nu_r')=0$. In the general case $\kappa\neq 1$, \eqref{eq:controlMC_Bais} and \eqref{eq:controlMC_MeanSquaredError} show that both the variance and the bias of $\widehat F_r$ are no greater than a constant independent of $P_r$ times $\mathcal{E}_\Gamma(P_r,f)$. Note that the quantity $\mathcal{E}_\Gamma(P_r,f)$ differs from the reconstruction error $\mathcal{R}_{\Gamma}(P_r,H)$ only by the term $f \, \mathbb{E}_{\mu}(f|\sigma(P_r))^{-1}$; compare \eqref{eq:reconstructionError} and \eqref{eq:defE}. Provided that this term is uniformly bounded from above, we can write
$$
 \mathcal{E}_\Gamma(P_r,f) \leq \Big(\sup \frac{f}{\mathbb{E}_{\mu}(f|\sigma(P_r))} \Big) \mathcal{R}_{\Gamma}(P_r,H).
$$
The above relation shows that the error $\mathbb{E}( \Dkl(\nu_r^*|| \widehat\nu_r) )$ can be controlled by the reconstruction error $\mathcal{R}_{\Gamma}(P_r,H)$, provided the supremum of $f\,\mathbb{E}_{\mu}(f|\sigma(P_r))^{-1}$ is finite. Then a small number of samples $M$ will be enough to guarantee that $\Dkl(\nu_r^*|| \widehat\nu_r)$ is, in expectation, of the same order of magnitude as the first term $\Dkl(\nu|| \nu_r^*)$ in the decomposition \eqref{eq:DecompKL}, which was our initial motivation. We now give an example for which the term $\sup \left ( f / \mathbb{E}_{\mu}(f|\sigma(P_r)) \right )$ decreases to one with the rank of the projector.

\begin{example}\label{rem:SupF/EF}
 As in Section \ref{sec:SimpleExample}, we again consider $\mu=\mathcal{N}(0,I_d)$ and $f:x\mapsto \exp(-\frac{1}{2}x^T Ax)$ for some symmetric matrix $A\succeq 0$. By applying Proposition \ref{prop:bestPr}, the projector that minimizes $P_r\mapsto \mathcal{R}_{\Gamma}(P_r,H)$ is $P_r=\sum_{i=1}^r v_i v_i^T $ where $v_i$ is the $i$-th eigenvector of $A$ and let $\Sigma=(I_d+A)^{-1}$ and $\Sigma_r=(I_d+P_rAP_r)^{-1}$. Using the closed-form expression \eqref{eq:CondExpGaussianCase} for $\mathbb{E}_{\mu}(f|\sigma(P_r))$ and since $A-P_r A P_r = (I_d-P_r)A(I_d-P_r)\succeq0$, we have
 \begin{align*}
  \sup \frac{f}{\mathbb{E}_{\mu}(f|\sigma(P_r))} &= \sup_{x\in\mathbb{R}^d} \frac{\exp(-\frac{1}{2}x^T (A-P_rAP_r)x )}{ \mathrm{det}(\Sigma_r^{-1}\Sigma)^{1/2} }  \\
  &= \frac{1}{ \mathrm{det}(\Sigma_r^{-1}\Sigma)^{1/2}} 
  = \prod_{i=r+1}^d (1+\alpha_i)^{1/2} ,
 \end{align*}
 where $\alpha_i\geq0$ is the $i$-th eigenvalue of $A$. This shows that, in this example, the supremum of $f ~ \mathbb{E}_{\mu}(f|\sigma(P_r))^{-1}$ goes monotonically to one with the rank $r$ of $P_r$. 
 
\end{example}

\begin{remark}
In the way they are presented, the error analyses of the projector and of the conditional expectation (i.e., the two terms on the right-hand side of \eqref{eq:DecompKL}) are not unified: the first term is bounded in high probability \eqref{eq:boundKLoverall}, while the second is controlled in expectation (via Proposition~\ref{prop:DecompExpectKL} and subsequent discussion). One could unify these results, for instance by using a Markov inequality to control the probability that $\Dkl(\nu_r^*||\widehat\nu_r)$ exceeds a certain value. However, we believe this step is not essential to the purpose of this section, which is to show that the associated sample approximations are feasible and sound.
\end{remark}

\subsection{Algorithms}\label{sec:alg}

\subsubsection{Ideal algorithm}

Algorithm \ref{algo:ideal} can be used to construct an approximation $\widehat\nu_r$ of the posterior distribution $\nu$. It assumes that we can draw samples from the posterior distribution $\nu$.
Since this is typically not possible in practice, this algorithm is called ``ideal.''

\begin{algorithm}[h]
\caption{Ideal algorithm}
  \begin{algorithmic}[1]
\Require{Likelihood function $f$, prior distribution $\mu$, error threshold $\varepsilon$, sample sizes $K$ and $M$.}
\State Draw $K$ independent samples $X_1,\ldots,X_K$ from $\mathrm{d}\nu\propto f \mathrm{d}\mu$
\State Compute $\nabla\log f(X_k)$ for $k = 1, \ldots, K$
\State Assemble the matrix $\widehat H=\frac{1}{K} \sum_{k=1}^K \big( \nabla \log f(X_k) \big) \big( \nabla \log f( X_k ) \big)^T$ 
\State Compute a rank-$r$ projector $P_r$ with the smallest rank such that $\mathcal{R}_{\Gamma}(P_r, \widehat H) \leq \varepsilon$
\State Draw $M$ samples $Y_1,\ldots,Y_M$ from $\mu$
\State Return the approximate distribution $\widehat\nu_r$ defined by
\[
\frac{\text{d}\widehat\nu_r}{\text{d}\mu}\propto\widehat F_r, \textrm{\quad where \;\;} \widehat F_r = \frac{1}{M} \sum^{M}_{i = 1} \left(P_r x + (I_d-P_r)Y_i \right)
\]
  \end{algorithmic}
\label{algo:ideal}
\end{algorithm}

\subsubsection{Construction using approximate measures}

Obtaining samples from the posterior distribution $\nu$ at step 1 of Algorithm \ref{algo:ideal} can be difficult in practice; indeed, this challenge is one of the motivations for the dimension reduction approach presented here. To alleviate this difficulty, we can construct $\widehat H$ using samples from another distribution $\widetilde\nu$ that can be directly simulated.
By the definition \eqref{eq:defH} of $H$, we have
\begin{align*}
 \mathcal{R}_{\Gamma}(P_r,H) & = \int \| (I_d - P_r^T) \nabla \log f \|_{\Gamma^{-1}}^2   \frac{\text{d}\nu}{\text{d}\widetilde\nu} \text{d}\widetilde\nu \\
 & \leq \big(  \sup \frac{\text{d}\nu}{\text{d}\widetilde\nu} \big) \int \| (I_d - P_r^T) \nabla \log f \|_{\Gamma^{-1}}^2  \text{d}\widetilde\nu ,
\end{align*}
so that, thanks to Corollary \ref{cor:KL_bound}, we can write
\begin{equation}\label{eq:BoundHTilde}
 \Dkl( \nu || \nu_r^* ) \leq \big(  \sup \frac{\text{d}\nu}{\text{d}\widetilde\nu} \big) \frac{\kappa }{2} \mathcal{R}_{\Gamma}(P_r,\widetilde H)
 ,\quad\text{where}\quad
 \widetilde H = \int(\nabla\log f)(\nabla\log f)^T \text{d}\widetilde\nu,
\end{equation}
for any projector $P_r$. Assuming $\widetilde\nu$ satisfies $\sup \frac{\text{d}\nu}{\text{d}\widetilde\nu} <+\infty$ (a common assumption in importance sampling \cite{mcbook,robert2013monte}), the above relation suggests that minimizing the approximate reconstruction error $P_r\mapsto\mathcal{R}_{\Gamma}(P_r,\widetilde H)$ can yield suitable projectors. In particular if $\widetilde H$ is rank deficient, then $\mathcal{R}_{\Gamma}(P_r,\widetilde H)=0$ and thus $\Dkl( \nu || \nu_r^* )=0$ for a suitable projector $P_r$ with $r=\text{rank}(\widetilde H)$. In general, however, the constant $\sup \frac{\text{d}\nu}{\text{d}\widetilde\nu}$ will be unknown in practice, and we will not be able to use \eqref{eq:BoundHTilde} as a bound for the error $\Dkl( \nu || \nu_r^* )$.

By drawing samples from $\widetilde \nu$ rather than from $\nu$ at step 1 of Algorithm~\ref{algo:ideal}, the matrix $\widehat H$ computed at step 2 is a Monte Carlo approximation of $\widetilde H$. Proposition \ref{prop:ControlReconstructionError} still applies when replacing $\nu$ by $\widetilde\nu$. Then for $K$ sufficiently large ($K=\mathcal{O}(\text{rank}(\widetilde H))$, for instance), it holds with high probability that any rank-$r$ projector that minimizes $P_r\mapsto\mathcal{R}_{\Gamma}(P_r,\widehat H)$ will be a quasi-optimal solution to the minimization problem of $P_r\mapsto\mathcal{R}_{\Gamma}(P_r,\widetilde H)$; see the discussion of Section \ref{sec:ConstructionOfTheProjector}.
Here we list two common choices of $\widetilde\nu$. 
\begin{enumerate}
  \item {\bf Laplace approximation}. The Laplace approximation \cite{schillings2020convergence} constructs a Gaussian distribution $\widetilde\nu = \mathcal{N}(\widetilde m,\widetilde \Sigma)$, where $\widetilde m$ is the mode of $\nu$ and the matrix $\widetilde\Sigma^{-1}$ is obtained from Hessian of the negative log density of $\nu$ evaluated at $\widetilde m$. Note that if $\nu$ is Gaussian then $\widetilde\nu$ is exactly $\nu$. Even though Laplace's method is a simple way to obtain a Gaussian approximation of $\nu$, there is no guarantee that $\sup \frac{\text{d}\nu}{\text{d}\widetilde\nu}$ is finite in general.
  
  \item {\bf Prior distribution}. Sampling from the prior distribution $\mu$ is usually tractable. With the choice $\widetilde\nu=\mu$, we have $\frac{\text{d}\nu}{\text{d}\widetilde\nu} = \frac{\text{d}\nu}{\text{d}\mu} =\frac{f}{\int f\text{d}\mu}$. For most applications, the likelihood function $f$ is bounded so that $\sup \frac{\text{d}\nu}{\text{d}\widetilde\nu} <\infty$. Note that this choice has been considered in \cite{constantine2016accelerating}.
  
\end{enumerate}

\subsubsection{Iterative construction}\label{sec:iterativeAlgo}

As suggested by \cite{cui2016scalable}, rather than limiting ourselves to a fixed approximation $\widetilde H$, we can approximate the true $H$ \eqref{eq:defH} using a sequential importance sampling framework. Let us consider a sequence of posterior approximations $\widehat\nu_r^{(0)},\hdots,\widehat\nu_r^{(L)}$ where $\widehat\nu_r^{(0)}=\mu$ and, for any $1\leq l \leq L$, the distribution $\widehat\nu_r^{(l)}$ is associated with a projector $P_r^{(l)}$ by
$$
\frac{\text{d}\widehat\nu_r^{(l)}}{\text{d}\mu \;\;\;} \propto \widehat F_r^{(l)}, 
\quad\text{where}\quad
\widehat F_r^{(l)}: x\mapsto \frac{1}{M} \sum^{M}_{i = 1} f\left(P_r^{(l)} x + (I_d-P_r^{(l)})Y_i\right).
$$
Notice that the same samples $Y_i$ are used for every $1\leq l \leq L$. The idea is to use $\widehat\nu_r^{(l)}$ as a biasing  distribution for the estimation of $H$. Let $X_1^{(l)}, \ldots, X_K^{(l)} $ be $K$ independent samples from $\widehat\nu_r^{(l)}$. We can write
\[
H = \int \big( \nabla \log f(x) \big) \big( \nabla \log f( x ) \big)^T \frac{\text{d}\nu}{\text{d}\widehat\nu_r^{(l)}} \text{d}\widehat\nu_r^{(l)},
\]
so that
\[
\widehat H^{(l)}= \frac{1}{\sum_{k =1 }^{K} w_k^{(l)}} \sum_{k=1}^K w_k^{(l)} \big( \nabla \log f(X_k^{(l)}) \big) \big( \nabla \log f( X_k^{(l)} ) \big)^T,
\quad
w_k^{(l)} = \frac{f(X_k^{(l)})}{\widehat F_r^{(l)} (X_k^{(l)})}
\]
is a self-normalized importance sampling estimator of $H$. Having computed $\widehat H^{(l)}$, the $(l+1)$-th projector $P_r^{(l+1)}$ is defined as a projector with minimal rank such that the approximate reconstruction error $\mathcal{R}_{\Gamma}(P_r^{(l+1)},\widehat H^{(l)})$ is below some prescribed tolerance. 

This iterative construction is detailed in Algorithm \ref{algo:iter_algo}.
Notice that at the first iteration, the importance weights $w_k^{(0)}$ are set to one. We make this choice in order to avoid the potential degeneracy (i.e., large variance of the weights) that might occur when $\mu$ is a poor approximation of $\nu$.
Also, our implementation includes a constraint on the rank of the projector, so that it cannot exceed a user-defined maximum rank $r_{\text{max}}$.
By doing so, we avoid any explosion of the rank in the earlier stages of the algorithm, i.e., when a poor posterior approximation might yield a crude approximation of $H$.
We emphasize that this algorithm only involves sampling from the low-dimensional posterior approximations $\{ \widehat{\nu}_r^{(l)} \}$.

\begin{algorithm}[h]
\caption{Iterative algorithm}
  \begin{algorithmic}[1]
\Require{Likelihood function $f$, prior distribution $\mu$, threshold $\varepsilon$, sample sizes $K$ and $M$, maximum number of iterations $L$, maximum rank $r_{\text{max}}$}
\State Draw $M$ samples $Y_1,\ldots,Y_M$ from $\mu$
\For{$l = 0, \ldots, L$}
\If{$l = 0$}
\State Draw $K$ samples $X_1^{(l)},\ldots,X_K^{(l)}$ from $\mu$
\State Compute $\nabla\log f(X_k^{(l)})$ and set the weights $w_k^{(l)} = 1$ for $k = 1 ,\ldots, K$
\Else
\State Draw $K$ samples $X_1,\hdots,X_K$ from $\widehat\nu_r^{(l)}$ using any MCMC algorithm
\State Compute $\nabla\log f(X_k^{(l)})$ and $w_k^{(l)} = \frac{f(X_k^{(l)})}{\widehat F_r^{(l)} (X_k^{(l)})}$ for $k = 1 ,\ldots, K$
\EndIf
\State Assemble the matrix 
\[
\widehat H^{(l)}= \frac{1}{\sum_{k =1 }^{K} w_k^{(l)}} \sum_{k=1}^K w_k^{(l)} \big( \nabla \log f(X_k^{(l)}) \big) \big( \nabla \log f( X_k^{(l)} ) \big)^T
\]
\State Compute the lowest rank $r_\varepsilon$ such that $\mathcal{R}_{\Gamma}(P, \widehat H_{\,}^{(l)}) \leq \varepsilon$ for some rank-$r_\varepsilon$ projector $P$
\State Put $r = \min(r_{\text{max}}, r_\varepsilon)$ and form the rank-$r$ projector $P_r^{(l+1)}$ which minimizes $P\mapsto\mathcal{R}_{\Gamma}(P, \widehat H_{\,}^{(l)})$
\State Define the approximate distribution $\widehat\nu_r^{(l+1)}$ as
\[
\frac{\text{d}\widehat\nu_r^{(l+1)}}{\text{d}\mu}\propto\widehat F_r^{(l+1)}, 
\textrm{\quad where \;\;}
\widehat F_r^{(l+1)} :x\mapsto \frac{1}{M} \sum^{M}_{i = 1} f\left(P_r^{(l+1)} x + (I_d-P_r^{(l+1)})Y_i\right)
\]
\EndFor
\State Return the approximate distribution $\widehat\nu_r^{(L+1)}$.
  \end{algorithmic}
\label{algo:iter_algo}
\end{algorithm}

\section{Alternative approaches to dimension reduction}\label{sec:4}

\subsection{Karhunen--Loève-based dimension reduction}\label{sec:KarhunenLoeve}

The Karhunen--Loève decomposition is a simple and powerful tool for reducing the dimension of a given random vector $X\in\mathbb{R}^d$. Letting $m=\mathbb{E}(X)$, this method exploits the fact that $X-m$ may take values mostly on a low-dimensional subspace of $\mathbb{R}^d$, so that $X-m$ can be well approximated by $P_r(X-m)$ for some low-rank projector $P_r \in\mathbb{R}^{d\times d}$. 
The standard approach is to seek $P_r$ such that the mean squared error $\mathbb{E}( \| (X-m) - P_r (X-m) \|_2^2 )$ is below some prescribed tolerance. We can write
$$
 \mathbb{E}( \| (X-m) - P_r (X-m) \|_2^2 ) = \text{trace}((I_d-P_r)\Sigma(I_d-P_r^T)) ,
$$
where $\Sigma = \mathbb{E}((X-m)(X-m)^T)$ is the covariance matrix of $X$. Using Proposition \ref{prop:bestPr}, we have that the orthogonal projector onto the $r$-dimensional leading eigenspace of $\Sigma$ is a minimizer of the mean squared error over the set of rank-$r$ projectors. Furthermore we have
$$
 \min_{\substack{ P_r\in\mathbb{R}^{d\times d}, \\ \text{rank-$r$ projector}}} \mathbb{E}( \|  (X-m) - P_r (X-m)   \|_2^2 ) = \sum_{i=r+1}^d \sigma_i^2,
$$
where $\sigma_i^2$ is the $i$-th eigenvalue of $\Sigma$. This relation shows that a strong decay in the spectrum of $\Sigma$ ensures the mean squared error can be arbitrarily small for some $r\ll d$. The eigenvectors of $\Sigma$ are called the Karhunen--Loève modes of $X$ and $m + P_r (X-m)$ corresponds to the truncated Karhunen--Loève decomposition of $X$.

\smallskip

This methodology can be used to approximate the prior distribution $\mu$. Assuming $X\sim\mu$ has mean $m=\mathbb{E}(X)$ and covariance matrix $\Sigma$, let $\mu_r$ be the distribution of $m+P_r (X-m)$, where $P_r$ is the rank-$r$ orthogonal projector onto the dominant eigenspace of the prior covariance matrix $\Sigma$. As proposed in \cite{li2006efficient,marzouk2009dimensionality,li2015note}, we can then introduce the approximate posterior $\tilde\nu_r$ such that 
$$
 \frac{\text{d}\tilde\nu_r}{\text{d}\mu_r} \propto f.
$$
This Karhunen--Loève-based dimension reduction and the dimension reduction method primarily considered in this paper yield \emph{different approximation formats} for the posterior measure: the latter considers approximate measures that are fully supported on $\mathrm{supp}(\mu)$ (which is convex by Assumption \ref{assu:convexityOfV}), while the former seeks an approximate measure that is supported on $\mathrm{supp}(\mu_r) = \mathrm{supp}(\mu) \cap \{m + \text{Im}(P_r)\}$. The difference in support makes the two approximations hard to compare since, by construction, the divergence $\Dkl( \nu || \nut_r)$ is infinite. This is hardly surprising, as Karhunen--Loève dimension reduction does not try to minimize $\Dkl( \nu \vert\vert \nut_r)$, but rather focuses on the mean squared error $\Ex(\| (X-m) - P_r(X-m) \|^2)$, with $X \sim \mu$.

From a computational perspective, Karhunen--Loève dimension reduction relies only on the prior measure $\mu$: there is no need to compute expectations over the posterior measure $\nu$ or to evaluate gradients of the likelihood function, and hence the resulting approximation is easy to compute. Yet these computational advantages highlight some fundamental limitations of the method. The efficiency of the reduction strategy is limited to cases where there is a sharp decay in the eigenvalues of the prior covariance matrix $\Sigma$; the method does not exploit any low-dimensional structure that the likelihood function might have.

\begin{remark}
 It is also possible to apply the Karhunen--Loève dimension reduction method to the posterior measure $\nu$. In that case, we seek an approximation of $\nu$ defined as the distribution of $m+P_r(X-m)$ for some projector $P_r$, where $X\sim\nu$ follows the posterior distribution and $m=\mathbb{E}(X)$ is the posterior mean. As before, the projector can be defined as the minimizer of $P_r\mapsto\mathbb{E}(\|(X-m)-P_r(X-m)\|_2^2)$ for $X \sim \nu$, which turns out to be an orthogonal projector onto the leading eigenspace of the posterior covariance. This approach is no longer an \textit{a priori} method since it requires computing an integral over the posterior distribution.
\end{remark}

\subsection{Likelihood-informed subspace}

The likelihood-informed subspace (LIS) reduction method \cite{cui2014likelihood} leverages optimality results available for the Bayesian linear--Gaussian model \cite{spantini2015optimal} to propose a structure-exploiting approximation of the posterior distribution. The method assumes the prior to be Gaussian $\mu = \mathcal{N}(m, \Sigma )$ and the likelihood function $f$, up to a multiplicative constant, to have the following form:
$$
  f:x\mapsto \exp\Big( -\frac{1}{2} \| G(x) - y \|_{\Sigma_\text{obs}^{-1}}^2 \Big).
$$
Here $G$ is a suitably regular forward operator, $y$ is the observed data vector, and $\Sigma_\text{obs}$ is the covariance matrix of the observational noise which is assumed to be additive and Gaussian.
The LIS reduction approximates the posterior $\nu$ by $\nulis_r$, defined as 
$$
  \frac{\mathrm{d}\nulis_r}{ \mathrm{d} \mu}(x) \propto f( P_r x + (I_d-P_r)m ).
$$
Here $P_r$ is a projector that is self--adjoint with respect to the inner product induced by $\Sigma^{-1}$ and whose range is spanned by the leading generalized eigenvectors of the matrix pencil $(\Hlis , \Sigma^{-1})$, where
\begin{equation} \label{eq:pencil_lis}
    \Hlis = \int \nabla G^T \, \Sigma_\text{obs}^{-1} \, \nabla G \,\mathrm{d}\nu.
\end{equation}
The matrix $\Hlis$ is the expectation over the posterior $\nu$ of the Gauss--Newton Hessian of the log-likelihood $\nabla G^T \, \Sigma_\text{obs}^{-1} \, \nabla G$. 
The work of \cite{cui2016scalable} generalizes the construction of the matrix $\Hlis$ to reference measures other than the posterior---for example, replacing $\nu$ with the prior measure or the Laplace approximation to the posterior.

The LIS methodology differs from the methodology considered in this paper in the following ways. First, the likelihood function is approximated by $x\mapsto f( P_r x + (I_d-P_r)m )$, rather than by the conditional expectation $\Ex_\mu(f | \sigma(P_r))$. This choice leads to a suboptimal approximation with respect to the Kullback--Leibler divergence; see Section \ref{sec:2.1}. Notice however that this approximation choice is quite similar to the Monte Carlo approximation of the conditional expectation presented in Section \ref{sec:ApproximationConditionalExpectation}. If the sample size satisfies $M=1$, the approximation of the conditional expectation defined by \eqref{eq:MonteCarloCond} is $\widehat F_r : x\mapsto f( P_r x + (I_d-P_r)Y)$, where $Y$ is a sample drawn from the prior distribution $\mu$.

The other important difference is in the definition of the projector. Recall that in Section \ref{sec:ControlledApprox}, $P_r$ is defined as the projector onto the leading eigenspace of the pencil $(H,\Sigma^{-1})$, where $H$ is defined in \eqref{eq:defH}.
The matrices $H$ and $\Hlis$ are in general different, and so are the resulting projectors. The projector $P_r$ in the present paper is defined as the minimizer of an upper bound on the Kullback--Leibler divergence between $\nu$ and its approximation (see Propositon \ref{prop:bestPr} in Section \ref{sec:ControlledApprox}). In contrast, the projector introduced in \cite{cui2014likelihood} is only justified by analogy with optimality results developed in the linear--Gaussian case \cite{spantini2015optimal}.
As a consequence, the LIS projector does not come with an error analysis on the resulting posterior approximation, while  the strategy presented in this paper does. Yet both methodologies can perform remarkably well in applications, and can be even comparable; cf.\ Section \ref{sec:numerics}.

\subsection{Active subspace for Bayesian inverse problems}

The active subspace (AS) method \cite{constantine2014active,russi2010uncertainty} is a dimension reduction technique which addresses the approximation (in the $L^2$ sense) of a real-valued function by means of a ridge function. These directions span the so-called \textit{active subspace} \cite{russi2010uncertainty}. In \cite{constantine2016accelerating}, this methodology is used to approximate the log-likelihood function in a Bayesian inverse problem. Denoting the prior covariance by $\Sigma$, the posterior distribution $\nu$ is approximated by
\begin{equation} \label{eq:appx_as}
    \frac{\mathrm{d} \nuas_r}{ \mathrm{d} \mu} \propto \exp \Ex_{\mu} ( \log f \vert \sigma(P_r) ),
\end{equation}
where $P_r$ is defined as the $\Sigma^{-1}$-orthogonal projector onto the leading generalized eigenspace of the matrix pencil $(\Has , \Sigma^{-1})$, with
\begin{equation} \label{eq:pencil_as}
    \Has = \int (\nabla \log f) (\nabla \log f)^T \,\mathrm{d} \mu.
\end{equation}

Active subspace reduction differs from the methodology introduced in this paper in the following aspects. First, the conditional expectation in \eqref{eq:appx_as} applies to the log-likelihood function rather than to the likelihood function itself. This choice is motivated by the fact that $\Ex_{\mu} ( \log f \vert \sigma(P_r) )$ yields an optimal approximation of  $\log f$ in the set $\{g\circ P_r , \,  g:\mathbb{R}^d \rightarrow \mathbb{R}\}$ with respect to the $L^2_\mu$-norm. However, the function $\exp \Ex_{\mu} ( \log f \vert \sigma(P_r) )$ is \emph{not} optimal with respect to the Kullback-Leibler divergence; see Section \ref{sec:2.1}.
 
Another difference is in the choice of projector. Comparing \eqref{eq:pencil_as} with \eqref{eq:defH}, we see that the integral in \eqref{eq:pencil_as} is taken over the prior $\mu$ rather than the posterior $\nu$, and thus the matrices $\Has$ and $H$ are in general different. Notice that $\Has$ corresponds to the matrix $\widetilde H$ defined in \eqref{eq:BoundHTilde} with the choice $\widetilde\nu=\mu$.

Finally, we mention that the active subspace method comes with an upper bound on the Hellinger distance between $\nu$ and its approximation; see Theorem 3.1 in \cite{constantine2016accelerating}. This analysis relies on a Poincaré inequality rather than on a logarithmic Sobolev inequality. Moreover, the Hellinger bound in \cite{constantine2016accelerating} contains unknown (and uncontrolled) constants. In this paper, we provide sufficient conditions on the prior $\mu$ to control the constants associated with the logarithmic Sobolev inequality and hence with our Kullback--Leibler error bound. From a theoretical point of view, it is challenging to relate the two bounds, and we cannot make a definitive statement about their relative merit: comparing upper bounds of two different metrics is not really informative. Instead we will compare the two methodologies by means of numerical experiments in Section \ref{sec:numerics}.

\section{Numerical illustration} \label{sec:numerics}

\newcommand{\meter}{{\rm m}}
\newcommand{\unittime}{{\rm day}}
\newcommand{\param}{x}
\newcommand{\data}{y}
\newcommand{\prmean}{m}
\newcommand{\prcov}{\Sigma}
\newcommand{\obscov}{\Sigma_{\rm obs}}
\newcommand{\normal}{\mathcal{N}}
\newcommand{\forward}{G}
\newcommand{\error}{\varepsilon_\text{obs}}
\newcommand{\real}{\mathbb{R}}
\renewcommand{\hat}{\widehat}

We use two Bayesian inverse problems to numerically demonstrate various theoretical aspects of the proposed dimension reduction method. 
In both examples we assume that the prior distribution is Gaussian, $\mu = \normal (\prmean, \prcov)$. 
This choice of prior distribution satisfies Assumption \ref{assu:convexityOfV} with $\Gamma=\Sigma^{-1}$ and $\kappa=1$. 
In both examples, we suppose that the data $\data$ correspond to predictions of a nonlinear forward model $x\mapsto \forward(\param)$ that are corrupted by measurement noise, where the latter is additive and normally distributed with zero mean and covariance $\obscov$. Thus we have a likelihood function, up to a multiplicative constant, of the form
\begin{equation}\label{eq:formOfLikelihood}
 f:x\mapsto  \exp\left( - \frac12 \|\forward(\param) - \data \|^2_{\obscov^{-1}}\right).
\end{equation}
Computer codes used to produce these numerical results, along with instructions, are available at \url{https://github.com/tcui001/certified}.

\subsection{Example 1: Atmospheric remote sensing}

Our first example is a realistic atmospheric remote sensing problem, where satellite observations from the {\it Global Ozone MOnitoring System} (GOMOS) are used to estimate the concentration profiles of various gases in the atmosphere.

\subsubsection{Problem setup}

The GOMOS instrument repeatedly measures light intensities at different wavelengths $\lambda$ and different altitudes ``$\mathrm{alt}$.'' 
The light transmissions $T_{\lambda,\mathrm{alt}}$ are modeled using Beer's law:
\begin{equation}
  T_{\lambda,\mathrm{alt}}
  = T_{\lambda,\mathrm{alt}} \Big( \kappa^{1},\hdots,\kappa^{N_{\mathrm{gas}}} \Big)
  = \exp \left( - \int_{\text{path}(\mathrm{alt})} \sum_{\mathrm{gas}=1}^{N_{\mathrm{gas}}} a_\lambda^{\mathrm{gas}}(z(\zeta)) \kappa^{\mathrm{gas}}(z(\zeta)) \, \mathrm{d} \zeta \right) ,
  \label{mod}
\end{equation}
where the integral is taken along the ray path, and $z(\zeta)$ is the height in the atmosphere of a point $\zeta$ on the path. The curvature of the earth is taken into account; see the illustration in Figure \ref{fig:gomos_setting}. 
The quantity $a_\lambda^{\mathrm{gas}}(z)$, known from laboratory measurements, is called the cross-section. It is a measure of how much a gas absorbs light of a given wavelength $\lambda$ at a given height $z$.
In this model there are $N_{\rm gas}$ gases, and for each we would like to infer the density profile $\kappa^{\mathrm{gas}}:z\mapsto\kappa^{\mathrm{gas}}(z)$.
Each of the $N_{\rm gas}$ profiles are modeled as independent random processes with log-normal prior distributions.
That is, $\log \kappa^{\mathrm{gas}}(\cdot)$ follows a Gaussian distribution $ \mathcal{N}(\prmean^{\mathrm{gas}},\prcov^{\mathrm{gas}})$ where $\prcov^{\mathrm{gas}}$ is the covariance operator associated with the squared exponential kernel
\begin{equation}
 (z, z^\prime) \mapsto \sigma^2_{\mathrm{gas}}\exp \left(-\frac{\|z - z^\prime\|^2}{2z_0^2} \right).
\end{equation}
These priors are chosen to promote smooth gas density profiles with large variations.

We consider a discretization of the vertical axis $z$ into $N_{\rm alt}$ layers with piecewise constant densities. We denote by $x\in\mathbb{R}^{N_{\rm alt} N_{\rm gas}}$ the vector containing the logarithms of the (unknown) $N_{\rm gas}$ gas densities in the $N_{\rm alt}$ layers, i.e., 
$$
 x = \text{vec}  \begin{pmatrix}
      \log( \kappa^{1}(z_1) ) &\hdots& \log( \kappa^{N_{\rm gas}}(z_1) )\\
      \vdots & \ddots & \vdots \\
      \log(\kappa^{1}(z_{N_{\rm alt}})) &\hdots& \log( \kappa^{N_{\rm gas}}(z_{N_{\rm alt}})  )
     \end{pmatrix} ,
$$
where $z_1,\hdots,z_{N_{\rm alt}}$ denote the $N_{\rm alt}$ altitudes and $\text{vec}(\cdot)$ denotes the vectorization operator.
Because $\kappa^{\rm gas}$ is log-normally distributed, the prior distribution $\mu=\mathcal{N}(m,\Sigma)$ of the parameter $x$ is Gaussian, and its prior mean $m$ and its prior covariance $\Sigma$ are derived from the discretization of $m^{\rm gas}$ and $\Sigma^{\rm gas}$.

The data $y$ is a noisy measurement of the light transmissions $T_{\lambda,\mathrm{alt}}$ at $N_{\lambda}$ different wavelengths and at $N_{\rm alt}$ different altitudes. It is a vector $y\in\mathbb{R}^{N_{\lambda}N_{\rm alt}}$ given by
$$
 y = G(x) + \varepsilon_{\rm obs}
$$
where the forward model $G(x)$ results from the discretization of Beer's law \eqref{mod}, and where $\varepsilon_{\rm obs} \sim \mathcal{N}(0,\Sigma_{\rm obs})$ is centered Gaussian noise with known observation covariance $\Sigma_{\rm obs}$.
With this model, the likelihood function $f$ takes the form of \eqref{eq:formOfLikelihood}.

Here we adopt the same model setup and synthetic data set used in \cite{cui2014likelihood,cui2016scalable}. The atmosphere is discretized into $N_{\rm alt} = 50$ layers and with $N_{\rm gas}=4$ profiles to infer so the total parameter dimension is $d = N_{\rm alt}N_{\rm gas}= 200$. 
We have observations at $N_\lambda = 1416$ wavelengths, and thus the dimension of the data is $N_{\rm alt}N_{\lambda} = 70800$.
We refer the readers to \cite{tamminen2004adaptive,haario2004Markov} for a further description of the model setup, the data set, and the Bayesian treatment of this inverse problem.

\begin{figure}[h]
  \includegraphics[width=0.8\textwidth]{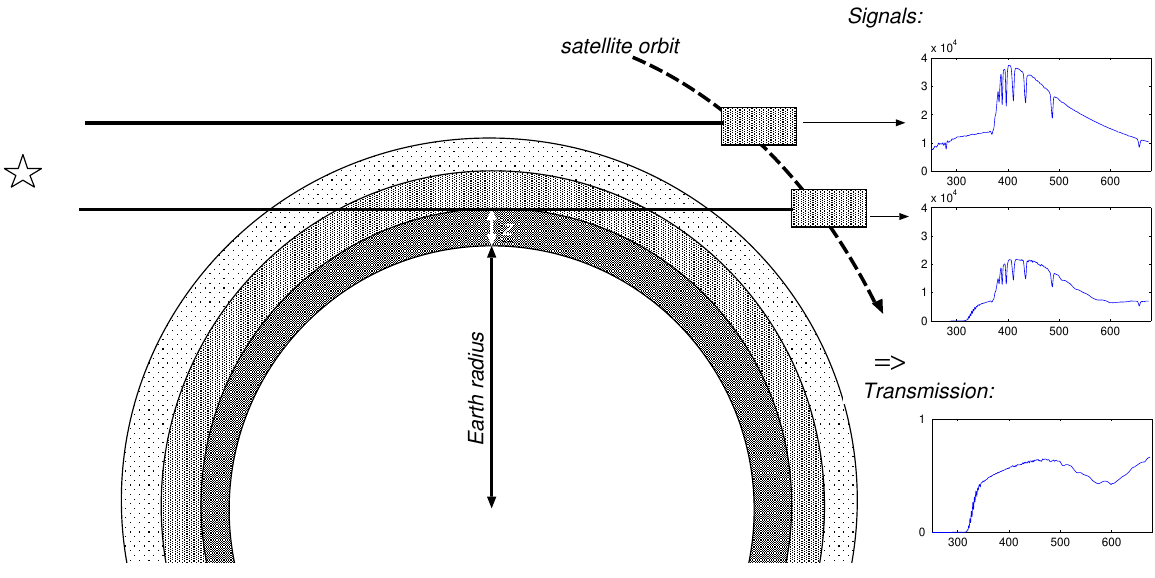}
  \caption{The principle of the GOMOS measurement. The atmosphere is represented as spherical layers around the Earth. Note that the thickness of the layers is much larger relative to the Earth in this figure than in reality. The figure is adopted from \cite{haario2004Markov}, with the permission of the authors.}
  \label{fig:gomos_setting}
\end{figure}

\subsubsection{Comparison of approximations}
\label{sec:gomoscompare}

\begin{figure}[h!]
  \centerline{\includegraphics[width=\textwidth]{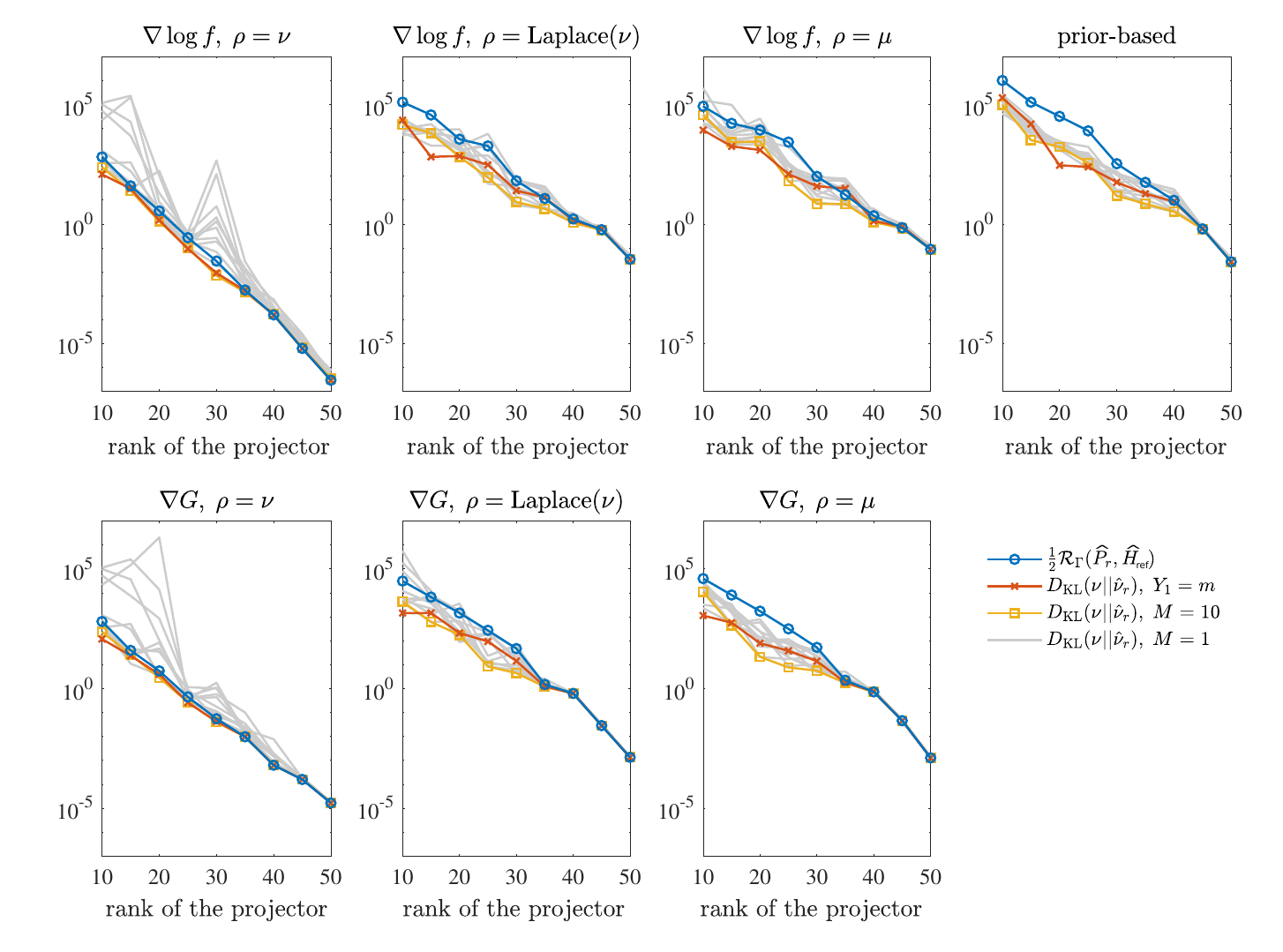}}
  \caption{GOMOS example: posterior approximation errors
   versus the rank of the projector $P_r$. The title of each plot (except the top right one) summarizes the combination of integrand and reference measure $\rho$ used to build the projector; see \eqref{eq:the differents Hs}. For each projector, we consider three different approximations of the conditional expectation and plot the resulting KL divergences $\Dkl ( \nu ||  \hat\nu_r)$: the prior mean approximation (crosses, $Y_1 = m$), the sample average approximation (squares, $M=10$), and the one-sample approximation (grey lines, $M=1$). For each choice of projector, the reconstruction error $\frac12 \mathcal{R}_{\Gamma}(\widehat{P}_r, \widehat{H}_\text{ref})$ is also illustrated by the blue circles, where the reference matrix $\widehat{H}_\text{ref}$ is a Monte Carlo estimate of $H=\int \nabla (\log f) (\nabla\log f)^T \text{d} \nu $ computed using $10^6$ posterior samples generated by the DILI algorithm.
  }
  \label{fig:gomos_bounds}
\end{figure}

We compare the posterior approximations obtained by the proposed method (see Sections \ref{sec:2} and \ref{sec:3}) to those obtained by the existing methods described in Section \ref{sec:4}. The comparison is summarized in Figures \ref{fig:gomos_bounds} and \ref{fig:gomos_dkl}.
\smallskip

Except for the prior-based method described in Remark \ref{sec:PriorBasedDimRed}, which builds a projector only from the prior covariance matrix $\Sigma$, the projector $P_r$ is built with a generalized eigendecomposition of 
\begin{equation}\label{eq:the differents Hs}
 H^{(\nabla\log f)}_{\rho} = \int (\nabla\log f)(\nabla\log f)^T \,\mathrm{d}\rho
\quad\text{ or }\quad
 H^{(\nabla G)}_{\rho} = \int (\nabla G)^T\obscov^{-1}(\nabla G) \,\mathrm{d}\rho,
\end{equation}
where $\rho$ is either the prior $\mu$, the posterior $\nu $, or the Laplace approximation of the posterior $\nu$. 
Using this notation, we have $H = H^{(\nabla\log f)}_{\nu}$ with $\rho=\nu$.
Each of the six possible combinations ($\{\nabla \log f , \nabla G\} \times \{\mu, \nu, \text{Laplace}(\nu)\}$) corresponds to a subplot in Figure \ref{fig:gomos_bounds}. 
In all cases, we compute Monte Carlo estimates of the matrices $H^{(\nabla\log f)}_{\rho}$ or $H^{(\nabla G)}_{\rho}$ above using $K=10^6$ samples, to minimize the impact of Monte Carlo error in this first comparison.
For $\rho = \mu$ or $\rho = \text{Laplace}(\nu)$, these samples are readily obtained. For $\rho=\nu$, the DILI algorithm with the `MGLI-Langevin' proposal \cite[Proposal  3.4]{cui2016dimension} is used as a state-of-the-art method for obtaining posterior samples. 
We first run the DILI algorithm for $10^5$ iterations to adaptively estimate the matrix $H^{(\nabla G)}_{\nu}$ and the associated reduced dimensional subspace required internally within the algorithm. Then, we continue to run the DILI algorithm with the previously computed subspace for another $10^6$ iterations to generate posterior samples.
This choice is not essential; we could equivalently have used many other algorithms to generate a reference posterior sample set---for instance, an adaptive Metropolis method \cite{chen2016accelerated,haario2001adaptive},
a delayed-acceptance method \cite{christen2005Markov,cui2011Bayesian,cui2019Aposteriori}, Hamiltonian Monte Carlo \cite{duane1987hybrid,neal2011MCMC} and improved versions thereof \cite{girolami2011riemann,hoffman2014no}, or sequential Monte Carlo samplers \cite{del2006sequential}.

Once we have estimated the matrix $H^{(\nabla\log f)}_{\rho}$ or $H^{(\nabla G)}_{\rho}$ for $\rho\in\{\mu, \nu, \text{Laplace}(\nu)\}$, we compute the generalized eigendecomposition of the estimate (i.e., $(\widehat{H}^{(\nabla\log f)}_{\rho}, \Sigma^{-1})$ or $(\widehat{H}^{(\nabla G)}_{\rho}, \Sigma^{-1})$) and assemble the rank-$r$ projector $\widehat{P}_r$ onto the corresponding leading eigenspace for $r\in\{10, 15, \hdots, 50\}$. 
Then we approximate the conditional expectation $\mathbb{E}_\mu(f|\sigma(P_r))$ by
$$
 x \mapsto \frac{1}{M}\sum_{i=1}^{M} f(\widehat{P}_r x + (I_d-\widehat{P}_r) Y_i ), 
 \quad\text{ where } Y_1,\hdots,Y_M \overset{\text{iid}}{\sim} \mu ,
$$
with either $M=10$ prior samples (yellow line with squares in Figure~\ref{fig:gomos_bounds}) or with $M=1$ prior sample (solid grey lines).
We also consider the deterministic approximation $x \mapsto f(\widehat{P}_r x + (I_d-\widehat{P}_r) m )$, which corresponds to using $M=1$ sample that is fixed to the prior mean, $Y_1=m$ (red line with crosses).
Finally, for each of the projectors $\widehat{P}_r$ obtained above, we also calculate and plot (blue line with circles) the reconstruction error $\frac12 \mathcal{R}_{\Gamma} (\widehat{P}_r,\widehat{H}_\text{ref})$ \eqref{eq:reconstructionError}, where $\widehat{H}_\text{ref}$ is an independent Monte Carlo estimate of $H = \int (\nabla \log f) (\nabla \log f)^T \text{d} \nu $, again computed using $10^6$ posterior samples generated by the DILI algorithm. As described in Corollary~\ref{cor:KL_bound}, the reconstruction error is an upper bound for the KL divergence $\Dkl( \nu || \nu_r^* )$ achieved by any given projector.

Once both the projector and the approximation of the conditional expectation are defined, the KL divergence $\Dkl (\nu || \hat\nu_r)$ is approximated with an independent set of $10^6$ samples drawn from the posterior using the DILI algorithm \cite{cui2016dimension}. This approximation is assumed to be sufficiently accurate for the purpose of our experiments.

Among the different ways of approximating the conditional expectation, the sample-based approach with $M = 10$ outperforms both the one-sample approach ($M=1$) and the deterministic approach ($Y_1 = m$). 
As shown in the theoretical analysis of Section \ref{sec:ApproximationConditionalExpectation}, the error due to approximating the conditional expectation is, in expectation, proportional to the reconstruction error; this is why we observe that the error decays quickly with the rank of the projector even when $M$ is small.
Note that the cost for evaluating the approximation of the conditional expectation is proportional to $M$: it is exactly $M$ times the cost of evaluating the exact likelihood function $f$, which might be expensive when $M$ is large.
Therefore we would prefer using the deterministic prior mean approach ($Y_1 = m$) which, in addition to having a low evaluation cost, has an accuracy comparable to that of the sample-based approach with $M = 10$ and seems better than the one-sample approach ($M=1$).

Figure \ref{fig:gomos_bounds} also shows that the reconstruction error provides a good estimate of the KL divergence, a quantity which is rarely available in practice.  In theory, the reconstruction error gives a bound on the KL divergence for any projector, provided that the conditional expectation is computed exactly---which is not what we do here.  In practice, however, we observe that the reconstruction error is in general a fairly good error indicator even when using a sample approximation of the conditional expectation. This might not be true when sample sizes are very small; see, for instance, the cases $\rho=\nu$ and $M=1$ in Figure \ref{fig:gomos_bounds}.

\begin{figure}[h]
  \centerline{\includegraphics[width=\textwidth]{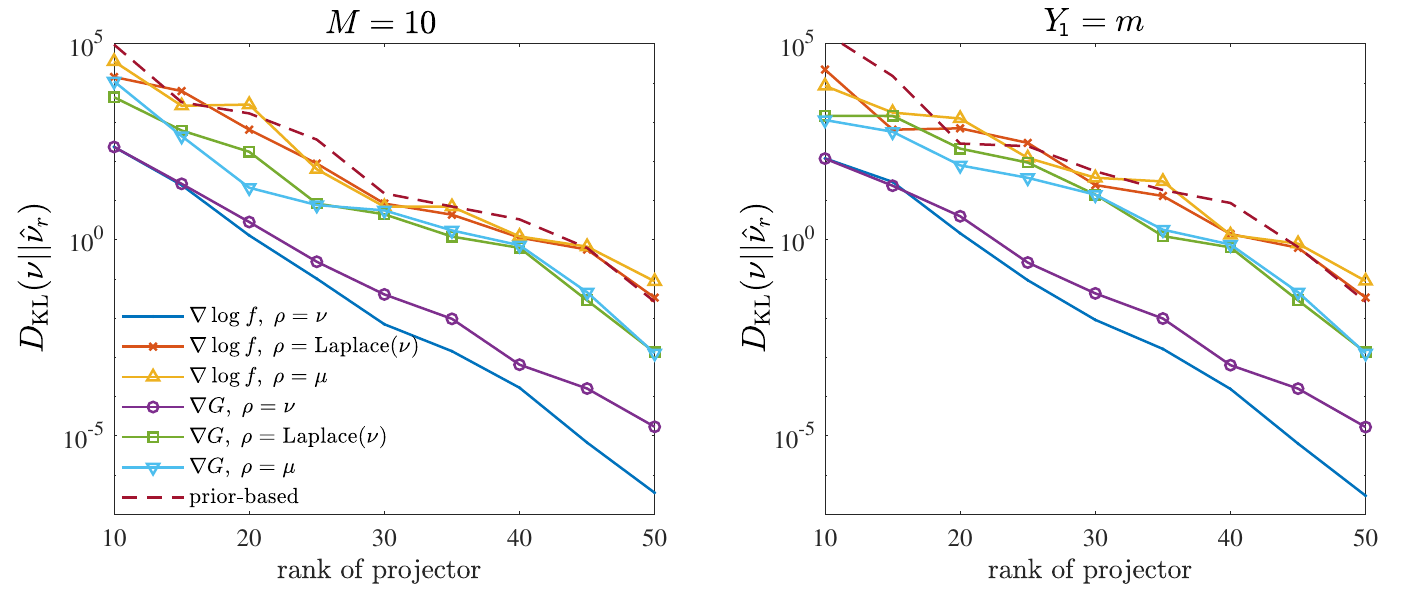}}
  \caption{GOMOS example. KL divergence $\Dkl ( \nu ||  \hat\nu_r)$ for posterior approximations $\hat\nu_r$ obtained using various projectors $P_r$. The left plot shows results with a sample average ($M = 10$) used to approximate the conditional expectation, while the right plot shows results using the prior mean approximation ($Y_1 = m$).}
  \label{fig:gomos_dkl}
\end{figure}

In Figure \ref{fig:gomos_dkl} we compare the performance of the different definitions of the projector. Here we only consider the KL divergence of posterior approximations obtained using the sample-based approximation of the conditional expectation with $M = 10$ (left) and the prior mean approximation $Y_1 = m$ (right).  
The projector obtained from $H_\nu^{(\nabla\log f)}$ achieves the best accuracy overall. The LIS method, obtained from $H_\nu^{(\nabla\log G)}$ also performs reasonably well. 
We observe a performance gap between these two projectors and the other five: the latter lead to significantly less accurate posterior approximations, at any given rank.

\subsubsection{Demonstration of Algorithm \ref{algo:iter_algo}}

Now we demonstrate the iterative procedure given in Algorithm \ref{algo:iter_algo} for constructing projectors and posterior approximations.
In this exercise, we approximate the conditional expectation using the prior mean option ($Y_1 = m$).
For each iteration $l$, the rank of the projector, the reconstruction error $\mathcal{R}_{\Gamma}(P_r^{(l)},\widehat{H}_\text{ref})$, and the KL divergence from the posterior measure $\nu$ to the resulting approximation $\hat\nu_r^{(l)}$, are shown in Figure \ref{fig:gomos_iter}.
The left column of Figure \ref{fig:gomos_iter} illustrates the iterative procedure where the error threshold $\varepsilon = 10^{-2}$ is used to determine the rank of the projector in each iteration.
Here we also set the maximum rank of the projector to be $r_{\rm max} = 40$. 
The right column of Figure \ref{fig:gomos_iter} shows a fixed-rank variant of the iterative procedure.
In this variant, the rank of the projector is held constant at $r=30$. In both cases, we use the adaptive MALA algorithm \cite{atchade2006adaptive} to draw samples from the $r$-dimensional distribution $\widehat{\nu}_r^{(l)}$ at each iteration. This choice is incidental; any other standard MCMC algorithm could have been used instead.

The iterative procedure appears to be effective in approximating the posterior.
When the rank of the projector is dynamically adjusted (left column of Figure \ref{fig:gomos_iter}), the iterative procedure achieves the desired level of accuracy within the first two iterations; then the rank of the projector, the reconstruction error, and the KL divergence  are stabilized in later iterations. 
When the rank of the projector is fixed (right column of Figure \ref{fig:gomos_iter}), the KL divergence decays significantly in the first two iterations, and then is stabilized in later iterations.

\begin{figure}[h]
  \centerline{\includegraphics[width=0.8\textwidth]{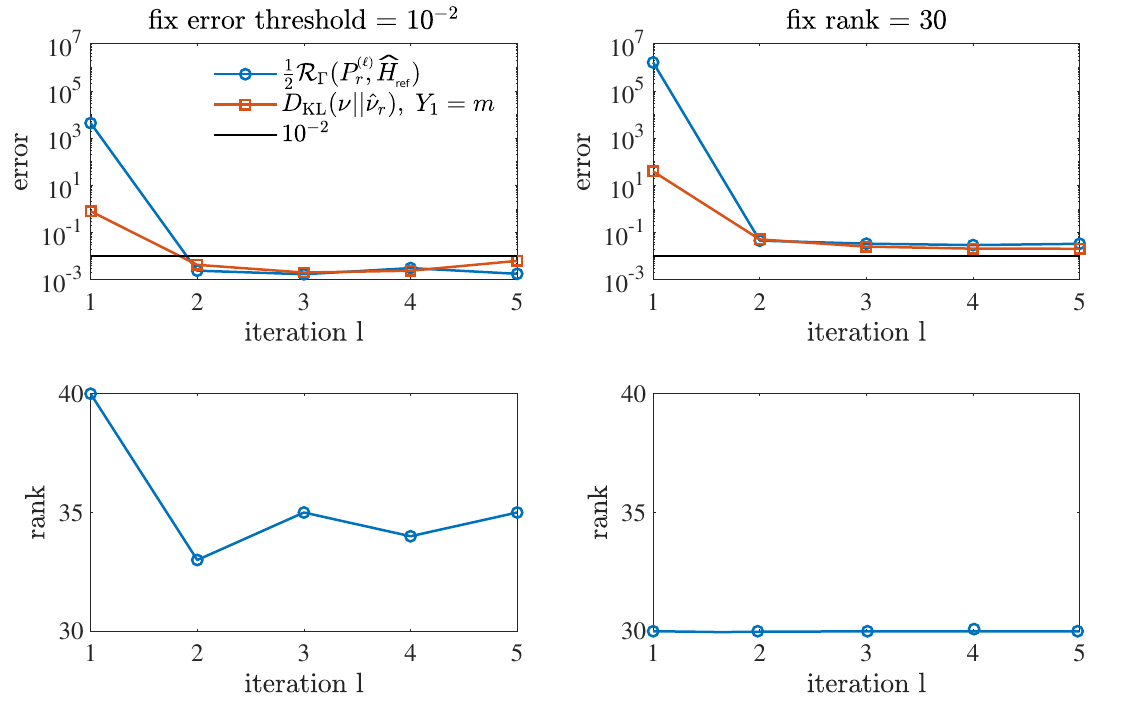}}
  \caption{GOMOS example. Rank of the projector, error bound  $\frac12\mathcal{R}_{\Gamma}(P_r^{(l)},\widehat{H}_\text{ref})$, and KL divergence $\Dkl(\nu|| \hat\nu_r^{(l)})$, where $\hat\nu_r^{(l)}$ is produced by Algorithm \ref{algo:iter_algo} and where $\widehat{H}_\text{ref}$ is a Monte Carlo estimate of $H=\int(\nabla\log f)(\nabla\log f)^T \mathrm{d}\nu$ computed using $10^6$ posterior samples generated by the DILI algorithm.
  The left column shows results where a target error tolerance $\varepsilon =  10^{-2}$ and $r_{\max}=40$ are used to determine the rank of the projector. The right column shows a fixed-rank variant of the iterative procedure, where the rank is held at $r=30$.
  }
  \label{fig:gomos_iter}
\end{figure}

\subsubsection{Impact of sample size}

We now analyze the impact of the sample size $K$ in the Monte Carlo approximation $\widehat H$ of $H$; see \eqref{eq:defHK}. As discussed in Section \ref{sec:ConstructionOfTheProjector}, Proposition \ref{prop:ControlReconstructionError} ensures that for sufficiently large $K$, the reconstruction error is sufficiently well approximated so that the resulting $\widehat P_r$ in \eqref{eq:ApproxPr} is quasi-optimal with high probability.
We now illustrate numerically the impact of the sample size $K$ via the following criteria:
\begin{enumerate}
 \item Is $\mathcal{R}_{\Gamma}(\widehat P_r, H)$ close to the minimum of the true reconstruction error $\mathcal{R}_{\Gamma}(P_r^*, H)$? Since our primary goal is to minimize a bound on the KL divergence, it makes sense to measure the quality of $\widehat P_r$ in terms of its ability to minimize the reconstruction error.  Thus we do not look at distances between $\widehat P_r$ and the minimizer $P_r^*$ (such as the operator norm $\|\widehat P_r - P_r^*\|$ or the Frobenius norm $\|\widehat P_r - P_r^*\|_F$) because these  distances are not directly related to the problem of posterior approximation.
 
 \item Can we use $\mathcal{R}_{\Gamma}(\widehat P_r, \widehat H)$ as an error indicator for the KL divergence? Since the quantity $\mathcal{R}_{\Gamma}(\widehat P_r, \widehat H)$ is the only one that is accessible in practice, we would like to know whether it is safe to use as an error estimator.
\end{enumerate}
Numerical results are summarized in the left plot of Figure \ref{fig:gomos_ss} for different sample sizes $K \in \{50 , 100 , 200 , 500 , 1000 , 2000, 10^6\}$ generated by the DILI algorithm. To assess the variability of these results due to the randomness in  $\widehat H$, we also repeat the previous experiment $10$ times for $K \in \{50 , 200 , 1000, 10^6\}$ (right plot of Figure \ref{fig:gomos_ss}).

\begin{figure}[h!]
  \centerline{\includegraphics[width=\textwidth]{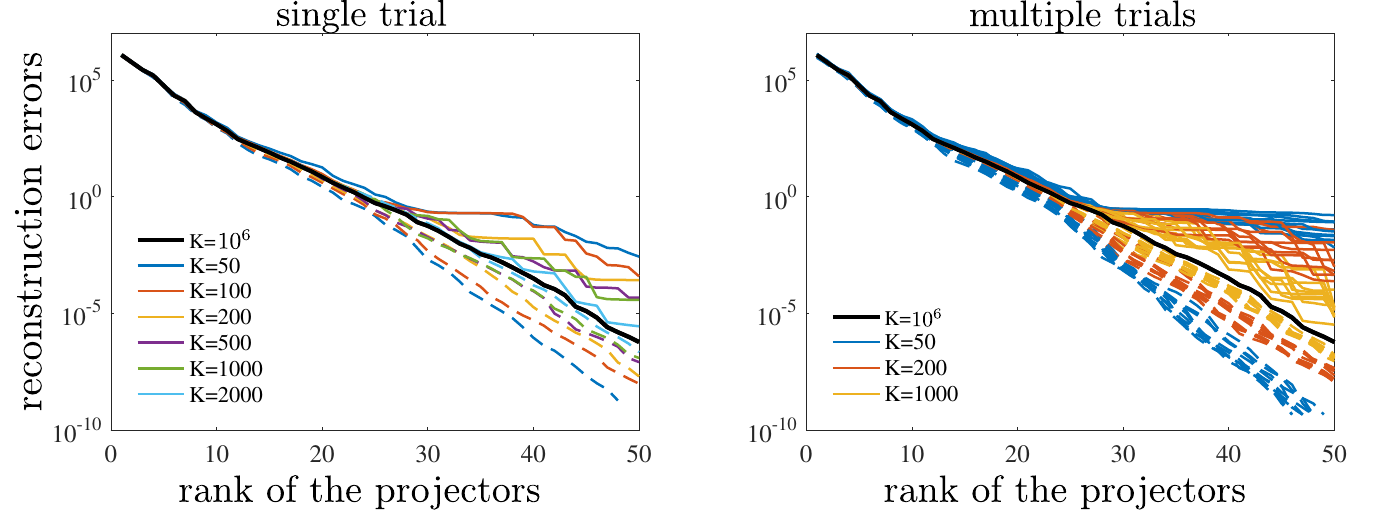}}
  \caption{GOMOS example. The approximate reconstruction error $\mathcal{R}_{\Gamma}(\widehat P_r, \hat H)$ (dashed lines) and reference reconstruction error $\mathcal{R}_{\Gamma}(\widehat P_r, \widehat{H}_\text{ref})$ (solid lines) for projectors $\widehat P_r$ obtained using different sample sizes $K$. The reference matrix $\widehat{H}_\text{ref}$ is a Monte Carlo estimate of $H=\int (\nabla \log f) (\nabla\log f)^T \text{d} \nu $ computed using $10^6$ posterior samples generated by the DILI algorithm.
  Left: errors obtained in single-trial experiments. Right: errors obtained in 10 repeated experiments. }
  \label{fig:gomos_ss}
\end{figure}

Concerning the first criterion, we observe that the quality of $\widehat P_r$ depends not only on the sample size $K$, but also on the rank $r$ of the projector. Indeed, $\mathcal{R}_{\Gamma}(\widehat P_r, \widehat{H}_\text{ref})$ (solid colored lines in Figure \ref{fig:gomos_ss}) is closest to the minimum of the reconstruction error (black line) for small $r$ and large $K$. This suggests that the sample size should be chosen larger when $r$ is chosen to be large. Note that this fact is not revealed by the theoretical analysis of Section \ref{sec:ConstructionOfTheProjector}. 

Concerning the second criterion, we observe that the approximate reconstruction error $\mathcal{R}_{\Gamma}(\widehat P_r, \hat H)$ (dashed colored lines) is always smaller than the reference reconstruction error $\mathcal{R}_{\Gamma}(\widehat P_r, \widehat{H}_\text{ref})$ (solid colored lines). This relationship is particularly apparent when the sample size $K$ is small and the rank $r$ is large.
This means that one should be careful when using $\mathcal{R}_{\Gamma}(\widehat P_r, \hat H)$ in place of the exact reconstruction error, as it tends to underestimate the error.

Overall, in this example, both criteria are validated for sufficiently large sample sizes. With $K \geq 1000$, the resulting projector $\widehat P_r$ has a reconstruction error comparable to that of the projector $P_r^*$, and the approximate reconstruction error $\mathcal{R}_{\Gamma}(\widehat P_r, \widehat H)$ provides a fairly good estimate for  $\mathcal{R}_{\Gamma}(\widehat P_r, H)$.

\subsection{Example 2: Elliptic PDE}
\label{sec:ellpitic}

Our second example is an inverse problem aiming at estimating the spatially inhomogeneous coefficient of an elliptic PDE, adopted from \cite{cui2016scalable}. In physical terms, our problem setup corresponds to inferring the
transmissivity field of a two-dimensional groundwater aquifer from
partial observations of the stationary drawdown field of the
water table, measured from well bores.

\subsubsection{Problem setup}
Consider a three kilometer by one kilometer problem domain
$\Omega = [0\,\meter, 3000\,\meter]\times [0\,\meter, 1000\,\meter]$,
with boundary $\partial \Omega$. We denote the spatial coordinate by
$\zeta \in \Omega$.
Consider the transmissivity field $T(\zeta)$ (units
[$\meter^2/\unittime$]), the drawdown field $p(\zeta)$ (units
[$\meter$]), and sink/source terms $q(\zeta)$ (units
[$\meter/\unittime$]).
The drawdown field for a given transitivity and source/sink
configuration is governed by the elliptic equation:
\begin{equation}
  -\nabla \cdot \left( T(\zeta) \nabla p(\zeta) \right) = q(\zeta), \quad \zeta \in \Omega .
  \label{eq:forward_e}
\end{equation}
We prescribe the drawdown value to be zero on the boundary (i.e., a
zero Dirichlet boundary condition), and define the source/sink term
$q(\zeta)$ as the superposition of four weighted Gaussian plumes with
standard width $50$ meters. The plumes are centered near the four
corners of the domain (at $[20\,\meter, 20\,\meter]$,
$[2980\,\meter, 20\,\meter]$, $[2980\,\meter, 980\,\meter]$ and
$[20\,\meter, 980\,\meter]$) with magnitudes of $-3000$, 2000, 4000,
and $-3000$ [$\meter/\unittime$], respectively.
We solve \eqref{eq:forward_e} by the finite element method.

\begin{figure}[h!]
  \centerline{\includegraphics[width=0.8\textwidth]{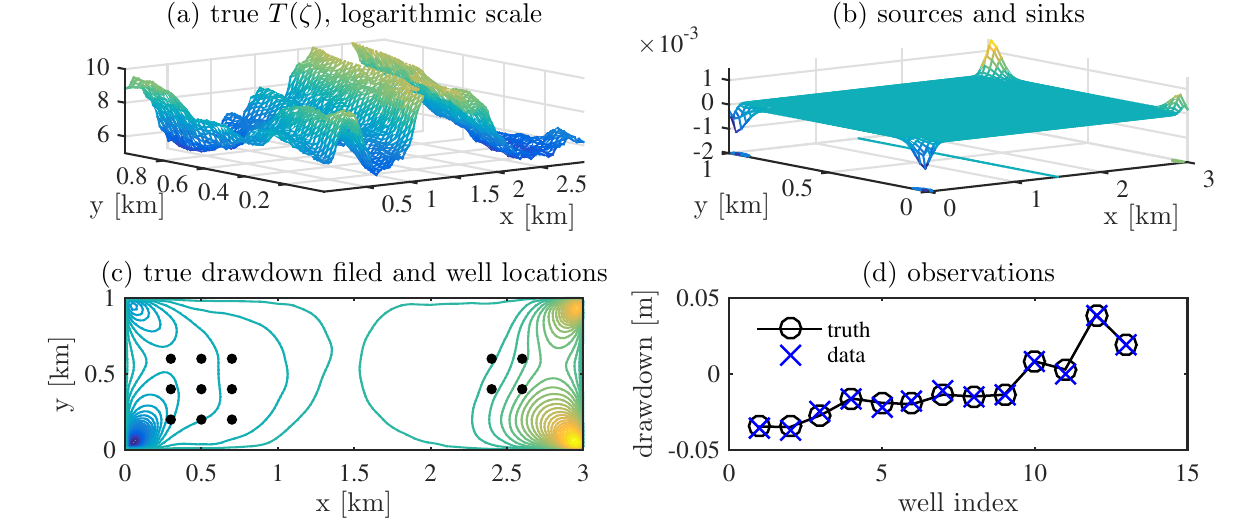}}
  \caption{Setup of the groundwater inversion example. (a) ``True''
    transmissivity field. (b) Sources and sinks. (c) Drawdown field
    resulting from the true transmissivity field, with observation
    wells indicated by black dots. (d) Data $\data$; circles represent
    the noise-free drawdowns at each well, while crosses represent the
    observed drawdowns with measurement noise.}
  \label{fig:setup_e}
\end{figure}

The discretized transitivity field $T(\zeta)$ is endowed with a log-normal prior distribution, i.e.,
$$
  T(\zeta) = \exp(\param(\zeta)), \; \text{ where } \; \param \sim \normal\left(\prmean, \prcov \right),
$$
where the prior mean is set to $\log( 1000\,[\meter/\unittime] )$ and
the inverse of the covariance matrix $\prcov^{-1}$ is defined through
the discretization of an Laplace-like stochastic partial differential
equation \cite{lindgren2011explicit},
$$
  ( -\nabla \cdot K \nabla + \kappa^2 ) \param(\zeta) = \mathcal{W}(\zeta),
$$
where $\mathcal{W}(\zeta)$ is white noise.
In this example, we set the stationary, anisotropic correlation tensor
$K$ to
\[
K = \left[\begin{array}{rr} 0.55 & -0.45\\ -0.45 &
    0.55 \end{array}\right],
\]
and put $\kappa = 50$.
The ``true'' transmissivity field is a realization from the prior
distribution. The true transmissivity field, sources/sinks, 
simulated drawdown field, and synthetic data are shown in Figure
\ref{fig:setup_e}.
Partial observations of the pressure field are collected at 13
sensors whose locations are depicted by black dots in Figure
\ref{fig:setup_e}(c), with additive error $e \sim \normal(0, \sigma^2 I_{13})$.
The standard deviation $\sigma$ of the measurement noise is prescribed
so that the observations have signal-to-noise ratio 20.
The noisy data are shown in Figure \ref{fig:setup_e}(d).

In this example, the finite element discretization of
\eqref{eq:forward_e} uses $120\times 40$ bilinear elements to
represent the drawndown field $p(\zeta)$, while the
transmissivity field $T(\zeta)$ is modeled as piecewise constant for
each element. This yields the discretized forward model 
$\forward : \real^{4800} \rightarrow \real^{13}$, where the parameter $\param$ is of dimension
$d = 4800$.

\subsubsection{Numerical results}

We proceed with the same comparison as in the GOMOS example. We compare the KL divergence $\Dkl ( \nu ||  \hat\nu_r)$ for different posterior approximations $\hat\nu_r$ obtained using the various projectors given in Sections \ref{sec:3} and \ref{sec:4}. 
The projectors used in this comparison are built in the same way as in the GOMOS example.
Also we consider the same three options ($M=1$, $M=10$, and $Y_1=m$) for approximating the conditional expectation within $\hat\nu_r$.
Again, to compute the KL divergence $\Dkl ( \nu ||  \hat\nu_r)$, we draw $10^6$ samples from the posterior using the DILI algorithm, which follows the same setup as that of the previous example. 
The comparison is summarized in Figures \ref{fig:ellptic_bounds} and \ref{fig:ellptic_dkl}.

\begin{figure}[h!]
  \centerline{\includegraphics[width=\textwidth]{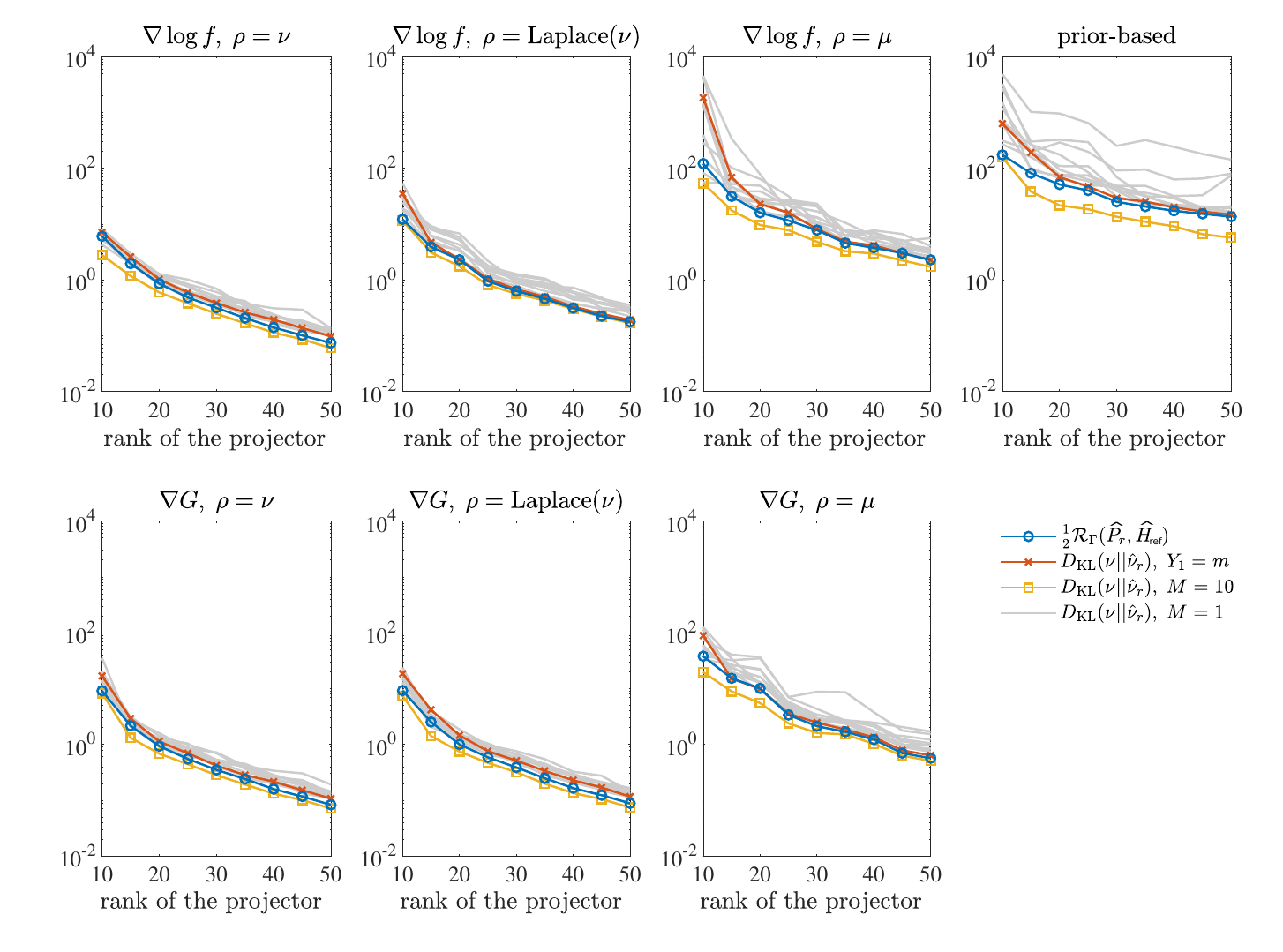}}
  \caption{Same as Figure \ref{fig:gomos_bounds}, but for the elliptic inverse problem: posterior approximation errors versus the rank of the projector $P_r$. The title of each plot (except the top right one) summarizes the combination of integrand and reference measure $\rho$ used to build the projector; see \eqref{eq:the differents Hs}. For each projector, we consider three different approximations of the conditional expectation and plot the resulting KL divergences $\Dkl ( \nu ||  \hat\nu_r)$: the prior mean approximation (crosses, $Y_1 = m$), the sample average approximation (squares, $M=10$), and the one-sample approximation (grey lines, $M=1$). For each choice of projector, the reconstruction error $\frac12 \mathcal{R}_{\Gamma}(\widehat{P}_r, \widehat{H}_\text{ref})$ is also illustrated by the blue circles, where the reference matrix $\widehat{H}_\text{ref}$ is a Monte Carlo estimate of $H=\int (\nabla\log f)(\nabla\log f)^T \text{d} \nu $ computed using $10^6$ posterior samples generated by the DILI algorithm.}
  \label{fig:ellptic_bounds}
\end{figure}

In Figure \ref{fig:ellptic_bounds} we observe that the approximation of the conditional expectation with $M = 10$ samples outperforms the two other options, $M=1$ and $Y_1=m$. This holds true regardless of how the projector is constructed.
We also notice that the prior-mean approximation $Y_1 = m$ in general performs better than the one-sample approximation $M=1$.
Similar to the GOMOS example, the reconstruction error $\mathcal{R}_{\Gamma}(P_r,\widehat{H}_\text{ref})$ provides an effective error indicator for posterior approximations 
that employ the estimate of the conditional expectation with $M=10$ for any choice of projector $P_r$.
Also, the KL divergence decays with $r$ at the same rate as the reconstruction error, independently of the choice of approximation scheme for the conditional expectation.

\begin{figure}[h!]
  \centerline{\includegraphics[width=\textwidth]{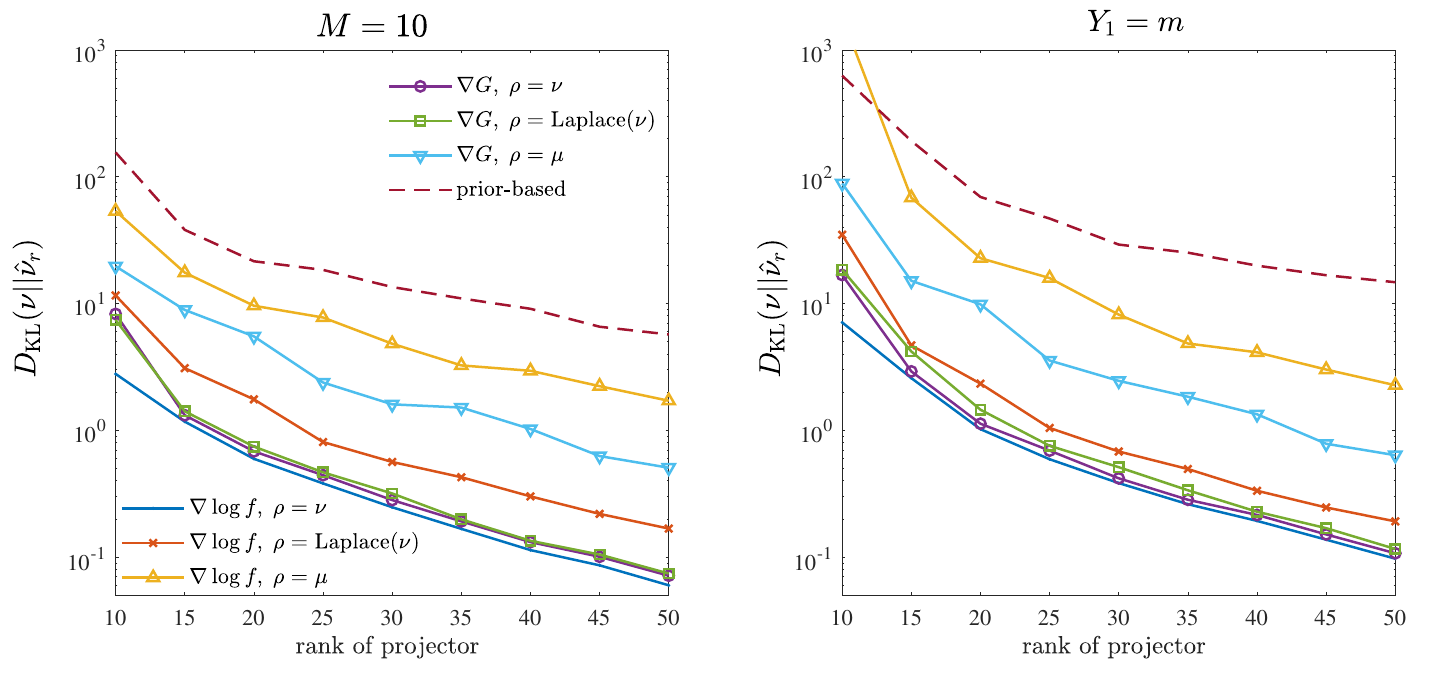}}
  \caption{Same as Figure \ref{fig:gomos_dkl}, but for the elliptic inverse problem: KL divergence $\Dkl ( \nu ||  \hat\nu_r)$ for posterior approximations $\hat\nu_r$ obtained using various projectors $P_r$. The left plot shows results with a sample average ($M = 10$) used to approximate the conditional expectation, while the right plot shows results using the prior mean approximation ($Y_1 = m$).}
  \label{fig:ellptic_dkl}
\end{figure}

A direct comparison of the posterior approximations defined by different projectors is shown in Figure \ref{fig:ellptic_dkl}, where approximations using the sample estimate ($M = 10$) and prior mean approximation ($Y_1 = m$) of the conditional expectation are collected in the left and right plots, respectively.   
The projector obtained from $H^{(\nabla \log f)}_\rho$ with $\rho=\nu$ outperforms all the other methods. In other words---and as observed in the previous example---the method proposed in Section \ref{sec:2} provides more effective dimension reduction than the alternatives. Note also that the projectors obtained from $H^{(\nabla G)}_\rho$ with $\rho=\nu$ and with $\rho=\textrm{Laplace}(\nu)$ perform better than the projector obtained from $H^{(\nabla \log f)}_\rho$ with $\rho=\textrm{Laplace}(\nu)$, indicating that efficient dimension reduction can also be obtained from gradient of the forward model itself.
All other projectors have a rather large accuracy gap relative to the abovementioned projectors. In particular, choosing $\rho = \mu$ in this example seems to be much less effective than averaging over the posterior distribution or its Laplace  approximation.

\section{Conclusion}

We have addressed the problem of reducing the dimension of a Bayesian inverse problem, in the nonlinear\slash non-Gaussian setting.  A Bayesian inverse problem has a low intrinsic dimension when the update from the prior distribution to the posterior distribution is essentially low-dimensional, meaning that the data only inform a few directions in the parameter space. We proposed a methodology that reveals and exploits such structure by seeking an approximation of the likelihood as a ridge function, i.e., a function that varies only on a low-dimensional subspace of its input space. To obtain this approximation, we first identified the optimal profile function of the ridge approximation, meaning the profile that minimizes the Kullback--Leibler divergence from the posterior approximation to the exact posterior.  Then, using logarithmic Sobolev inequalities, we derived an upper bound for the remaining error. The bound admits a simple form and can be easily minimized: this is how the informed directions are discovered. Moreover, our methodology provides a computable upper bound for the approximation error in the posterior distribution, measured in terms of Kullback--Leibler divergence, a quantity that is otherwise difficult to compute in practice.

Our method is fundamentally gradient-based and requires the second moment matrix of the gradient of log-likelihood function. Computing this matrix can be challenging in practice because it entails integrating over the posterior distribution. We thus propose several sample-based approximation schemes, including an iterative algorithm which employs only low-dimensional posterior approximations.
One direction which remains to be explored is the optimal balance between the computational effort used to reduce the dimension of the problem and the computational effort of exploring the final dimension-reduced posterior distribution.

The use of logarithmic Sobolev inequalities requires some (rather strong) assumptions on the prior distribution, e.g., being Gaussian, being a Gaussian mixture, or being a bounded perturbation of a Gaussian. 
One open question left for future work is how to weaken the assumption on the prior so that the resulting theory applies to priors with tails that are heavier than Gaussian tails. Another natural improvement of the methodology is the extension to infinite-dimensional parameter spaces.

Analytical and numerical examples demonstrate good performance of the proposed method. In particular, we show that it outperforms other state-of-the-art gradient-based dimension reduction schemes.
We also note that effective MCMC algorithms for large-scale Bayesian inverse problems, such as the DILI sampler of \cite{cui2016dimension}, fundamentally rely on the low-dimensional structure of the posterior distribution. 
In its original form, \cite{cui2016dimension} uses the LIS method to discover this low-dimensional structure, a method which is less efficient than the one we propose here. Incorporating our new developments in the DILI algorithm (or in any other MCMC algorithm that exploits dimension reduction) could thus yield better sampling performance.
Making this precise---e.g., understanding how the quality of the subspace $\text{Im}(P_r)$ affects MCMC mixing rates---is another important topic for future work.

\section{Appendices}

\subsection{Proof of Proposition \ref{prop:ExplicitCondExp}}\label{proof:ExplicitCondExp}

 We denote by $U_r\in\mathbb{R}^{d\times r}$ a matrix whose columns form a basis of $\text{Im}(P_r)$ and we let $U\in\mathbb{R}^{d\times d}$ be defined as the horizontal concatenation of $U_r$ and $U_\perp$. 
 We have $U\xi = U_r\xi_r + U_\perp \xi_\perp $ for any $\xi\in\mathbb{R}^d$, where $\xi_r\in\mathbb{R}^r$ and $\xi_\perp\in\mathbb{R}^{d-r}$ are the vectors containing respectively the $r$ first and the $d-r$ last components of $\xi$.
 Let $F:x\mapsto \int_{\mathbb{R}^{d-r}} f(P_r x + U_\perp \xi_\perp) ~p_\perp(\xi_\perp| P_r x ) \mathrm{d}\xi_\perp$. By definition of $p_\perp(\xi_\perp| P_r x )$ and since $P_r U_r=U_r$, this function satisfies
 $$
  F(U_r\xi_r) 
  = \frac{\int_{\mathbb{R}^{d-r}} f(U_r\xi_r + U_\perp \xi_\perp) ~ \rho( U_r\xi_r + U_\perp \xi_\perp  ) \mathrm{d}\xi_\perp}{  \int_{\mathbb{R}^{d-r}} \rho( U_r\xi_r + U_\perp \xi_\perp'  ) \mathrm{d}\xi_\perp'},
 $$
 for any $\xi_r\in\mathbb{R}^r$. Let $h$ be $\sigma(P_r)$-measurable function. By Lemma \ref{cor:Doob}, $h$ can be written as the composition of a function with $P_r$, so that $h( U_r\xi_r + U_\perp \xi_\perp ) = h (U_r\xi_r)$ holds for any $\xi_r\in\mathbb{R}^r$ and $\xi_\perp\in\mathbb{R}^{d-r}$. Thus we have
 \begin{align*}
  \int f~h ~\mathrm{d}\mu 
  &= \int_{\mathbb{R}^d} f(x) h(x) \rho(x) \mathrm{d}x 
  = \int_{\mathbb{R}^d} f(U\xi) h(U\xi) \rho(U\xi) |U| \mathrm{d}\xi \\
  &= \int_{\mathbb{R}^r}\int_{\mathbb{R}^{d-r}} f(U_r\xi_r + U_\perp \xi_\perp) h(U_r\xi_r ) \rho(U_r\xi_r + U_\perp \xi_\perp) |U|\mathrm{d}\xi_\perp \mathrm{d} \xi_r \\
  &= \int_{\mathbb{R}^r} \left( \int_{\mathbb{R}^{d-r}} f(U_r\xi_r + U_\perp \xi_\perp) \rho(U_r\xi_r + U_\perp \xi_\perp) \mathrm{d}\xi_\perp \right) h(U_r\xi_r ) |U| \mathrm{d} \xi_r \\
  &= \int_{\mathbb{R}^r} \left( F(U_r\xi_r)  \int_{\mathbb{R}^{d-r}} \rho( U_r\xi_r + U_\perp \xi_\perp') \mathrm{d}\xi_\perp' \right) h(U_r\xi_r ) |U| \mathrm{d} \xi_r \\
  &= \int_{\mathbb{R}^r} \int_{\mathbb{R}^{d-r}} F(U_r\xi_r)   h(U_r\xi_r ) \rho( U_r\xi_r + U_\perp \xi_\perp' ) |U| \mathrm{d}\xi_\perp' \mathrm{d} \xi_r \\
  &= \int_{\mathbb{R}^d} F(x) h(x) \rho(x)  \mathrm{d}x = \int F~h~\mathrm{d}\mu,
 \end{align*}
 where $|U|$ denotes the determinant of $U$. This shows that $F$ satisfies \eqref{eq:CondExp_VarForm}. Since it is a $\sigma(P_r)$-measurable function (as the composition of a function with $P_r$), then it is the conditional expectation $\mathbb{E}_\mu(f|\sigma(P_r))$.

\subsection{Proof of Theorem \ref{prop:subspaceLogSob}}\label{proof:subspaceLogSob}

 In this proof, let $K = \mathrm{supp}(\mu)$.
 Assumption \ref{assu:convexityOfV} implies $\frac{\mathrm{d}\mu}{\mathrm{d}x} =\rho \propto e^{-(V+\Psi)}$ where $V:K\to\mathbb{R}$ and $\Psi:K\to\mathbb{R}$ satisfy respectively \eqref{eq:convexityOfV} and \eqref{eq:kappa}.
 Let $P_r\in\mathbb{R}^{d\times d}$ be a rank-$r$ projector and let $U_\perp\in\mathbb{R}^{d\times (d-r)}$ be a matrix whose columns form a basis of $\text{Ker}(P_r)$. 
 For any $x\in K$, the conditional density $\xi_\perp\mapsto p_\perp(\xi_\perp|P_r x)$ defined by \eqref{eq:defCondDensity} is such that $p_\perp(\cdot | P_rx) \propto e^{-(V_\perp+\Psi_\perp)} $ where $V_\perp$ and $\Psi_\perp$ are functions on $K_\perp(x) = \{\xi_\perp\in\mathbb{R}^{d-r}: U_\perp\xi_\perp + P_rx \in K\}\subset\mathbb{R}^{d-r}$ defined by $V_\perp(\xi_\perp ) = V(P_rx + U_\perp\xi_\perp)$ and $\Psi_\perp(\xi_\perp ) = \Psi(P_rx + U_\perp\xi_\perp)$ for any $\xi_\perp\in K_\perp(x)$. 
 Because $K$ is convex, the set $K_\perp(x)$ is also convex.
 The function $V_\perp$ inherits the convexity property of $V$: using the chain rule, we have $\nabla^2V_\perp(\xi_\perp) = U_\perp^T \nabla^2 V(P_rx + U_\perp\xi_\perp) U_\perp $ so that
 $$
  \nabla^2V_\perp(\xi_\perp) 
  ~\succeq~ 
  U_\perp^T \Gamma U_\perp \succ 0,
 $$
 for any $\xi_\perp \in K_\perp(x)$.
 Also, $\Psi_\perp$ is a bounded function which satisfies
 $$
  e^{\sup \Psi_\perp - \inf \Psi_\perp} \leq e^{\sup \Psi - \inf \Psi} \leq \kappa.
 $$
 We conclude that the distribution on the convex set $K_\perp(x)$ with Lebesgue density $p_\perp(\cdot|P_rx)$ satisfies Assumption \ref{assu:convexityOfV}.
 Then it satisfies the logarithmic Sobolev inequality \eqref{eq:logSob}, meaning that
 \begin{align}
  \int g(\xi_\perp)^2 &\log \frac{g(\xi_\perp)^2}{ \int g(\xi_\perp')^2 p_\perp(\xi_\perp'|P_rx)\mathrm{d}\xi_\perp' } ~ p_\perp(\xi_\perp|P_rx)\mathrm{d}\xi_\perp \nonumber\\
  &\leq 2\kappa \int \| \nabla g(\xi_\perp) \|_{*}^2 ~p_\perp(\xi_\perp|P_rx)\mathrm{d}\xi_\perp , \label{eq:tmp0185}
 \end{align}
 holds for any function $g$ with sufficient regularity, where the norm $\| \cdot \|_{*}$ is defined by $\| v \|_*^2 = v^T (U_\perp^T\Gamma U_\perp)^{-1}v$ for any $v\in\mathbb{R}^{d-r}$. Inequality \eqref{eq:tmp0185} holds in particular when the function $g$ is defined by $g(\xi_\perp) = h(P_rx + U_\perp \xi_\perp)$ for any $\xi_\perp\in K_\perp(x)$, where $h$ is any function on $\mathbb{R}^d$ with sufficient regularity. By the chain rule we have $\nabla g (\xi_\perp) = U_\perp^T \nabla h(P_rx + U_\perp \xi_\perp)$ so that \eqref{eq:tmp0185} can be written as
 \begin{align*}
  \int h(P_rx + U_\perp \xi_\perp)^2 &\log \Big(\frac{h(P_rx + U_\perp \xi_\perp)^2}{ \int h(P_rx + U_\perp \xi_\perp')^2 p_\perp(\xi_\perp'|P_rx)\mathrm{d}\xi_\perp' } \Big) p_\perp(\xi_\perp|P_rx)\mathrm{d}\xi_\perp \\
  &\leq 2 \kappa  \int \| U_\perp^T \nabla h(P_rx + U_\perp \xi_\perp) \|_{*}^2 ~p_\perp(\xi_\perp|P_rx)\mathrm{d}\xi_\perp. 
 \end{align*}
 Note that the denominator in the logarithm is the conditional expectation $\mathbb{E}_\mu(h^2|\sigma(P_r))$ evaluated at $x$, see Proposition \ref{prop:ExplicitCondExp}. Let $U_r\in\mathbb{R}^{d\times r}$ be a matrix whose columns form a basis of $\text{Im}(P_r)$, and let $\xi_r\in K_r(x)=\{\xi_r\in\mathbb{R}^r: U_r\xi_r + (I_d-P_r)x \in K\} \subset\mathbb{R}^r$. Replacing $x$ by $U_r\xi_r$ in the above inequality and using the definition of the conditional density \eqref{eq:defCondDensity}, we obtain
 \begin{align*}
  \int h(U_r\xi_r + & U_\perp \xi_\perp)^2 \log \Big(\frac{h(U_r\xi_r + U_\perp \xi_\perp)^2}{ \mathbb{E}_\mu(h^2|\sigma(P_r))(U_r\xi_r) } \Big) \rho(U_r\xi_r + U_\perp \xi_\perp)\mathrm{d}\xi_\perp\\
  &\leq 2 \kappa  \int \| U_\perp^T \nabla h(U_r\xi_r + U_\perp \xi_\perp) \|_{*}^2 ~\rho(U_r\xi_r + U_\perp \xi_\perp)\mathrm{d}\xi_\perp.
 \end{align*}
 By integrating over $\xi_r$ we obtain
 \begin{equation*}\label{eq:tmp9576}
  \int h^2  \log \frac{h^2}{ \mathbb{E}_\mu(h^2|\sigma(P_r)) } ~\mathrm{d}\mu 
  \leq 2 \kappa \int_{\mathbb{R}^d} \| U_\perp^T \nabla h \|_{*}^2 ~\mathrm{d}\mu.
 \end{equation*}
 To conclude the proof, it remains to show that $\| U_\perp^T \nabla h \|_{*}^2 \leq \| (I_d - P_r)^T \nabla h \|_{\Gamma^{-1}}^2$ holds. By definition of $U_\perp$ we have $(I_d-P_r)U_\perp = U_\perp$ so that $\| U_\perp^T \nabla h \|_{*} = \| U_\perp^T (I_d-P_r)^T\nabla h \|_{*} $. 
 Also, note that $\|\cdot\|_*$ is the dual norm of the norm $\| \cdot \|_{U_\perp^T\Gamma U_\perp}$ defined by $\| \cdot \|_{U_\perp^T\Gamma U_\perp}^2 = (\cdot)^T U_\perp^T\Gamma U_\perp (\cdot)$, meaning that $\|v\|_* = \sup\{ |\xi_\perp^T v| : \xi_\perp\in\mathbb{R}^{d-r} , \|\xi_\perp \|_{U_\perp^T\Gamma U_\perp} \leq 1  \}$ for any $v\in\mathbb{R}^{d-r}$. Thus we have
 \begin{align*}
  \| U_\perp^T \nabla h \|_{*} 
  &= \| U_\perp^T (I_d-P_r)^T\nabla h \|_{*}
  = \sup_{0\neq \xi_\perp \in\mathbb{R}^{d-r} } \frac{| ( U_\perp\xi_\perp) ^T (I_d-P_r)^T\nabla h |}{ \| U_\perp \xi_\perp \|_{\Gamma} } \\
  & \leq \sup_{0\neq x \in \mathbb{R}^d } \frac{| x ^T (I_d-P_r)^T \nabla h |}{ \|x\|_{\Gamma}  } 
  \overset{~y=\Gamma^{1/2}x~}{=} \sup_{0\neq y \in \mathbb{R}^d } \frac{| y ^T \Gamma^{-1/2} (I_d-P_r)^T\nabla h |}{ \|y\| } \\
  &= \|\Gamma^{-1/2}(I_d-P_r)^T\nabla h\| = \|(I_d-P_r^T)\nabla h\|_{\Gamma^{-1}},
 \end{align*}
 where $\|\cdot\|=\sqrt{(\cdot)^T(\cdot)}$ denotes the Euclidean norm of $\mathbb{R}^d$ and $\Gamma^{1/2}$ denotes a symmetric positive-definite square root of $\Gamma$. This concludes the proof.

\subsection{Proof of Proposition \ref{prop:ControlReconstructionError}}\label{proof:ControlReconstructionError}

The proof of Proposition \ref{prop:ControlReconstructionError} requires concentration properties of sub-Gaussian random vectors. We will need the following lemma, which is essentially Theorem 5.39 in \cite{vershynin_2012}.
\begin{lemma}\label{lem:Concentration}
 Let $Z$ be an isotropic sub-Gaussian random vector in $\mathbb{R}^n$, meaning that $\mathbb{E}(Z Z^T) = I_n$, and let $L$ be a constant such that
 \begin{equation}\label{eq:SubGaussianZ}
  \| w^T Z \|_{\psi_2} \coloneqq \sup_{p\geq 1} p^{-1/2} \mathbb{E}(|w^T Z|^p)^{1/p} \leq L \| w \| ,
 \end{equation}
 for any $w\in\mathbb{R}^n$, where $\|\cdot\|$ denotes the Euclidean norm of $\mathbb{R}^n$. Let $\hat\Sigma = \frac{1}{K}\sum_{k=1}^K Z_k Z_k^T$, where $Z_1,\hdots,Z_K$ are $K$ independent copies of Z. Then for every $t\geq 0$, with probability at least $1-2\exp(-ct^2)$, we have
 \begin{equation}\label{eq:Concentration}
  \| \hat\Sigma - I_d \| \leq \max\{ \tau,\tau^2 \} 
  \quad\text{where}\quad
  \tau = C \sqrt{\frac{n}{K}} + \frac{t}{\sqrt{K}},
 \end{equation}
 where $\|\hat\Sigma - I_d\|$ is the spectral norm of $\hat\Sigma - I_d$. Here, $c=L^{-4}c_1 $ and $C=L^2\sqrt{\log (9) / c_1}$, where $c_1$ is an absolute (numerical) constant.
\end{lemma}

\begin{proof}[Proof of Lemma \ref{lem:Concentration}]
 The proof is exactly the one of Theorem 5.39 in \cite{vershynin_2012} with $A\in\mathbb{R}^{K\times d}$ the random matrix whose rows contain the vectors $Z_1,\hdots,Z_K$. By \eqref{eq:SubGaussianZ} the rows of $A$ are independent sub-Gaussian isotropic vectors with sub-Gaussian norm $\|Z\|_{\psi_2}$ smaller than $L$. Then, following \cite{vershynin_2012}, we have that \eqref{eq:Concentration}, which is nothing but Equation (5.23) in \cite{vershynin_2012} with $\hat\Sigma = \frac{1}{K}AA^T$, holds with probability greater than $1-2\exp(- c_1 t^2 \|Z\|_{\psi_2}^{-4})$. Here $c_1$ is an absolute constant coming from Corollary 5.17 in \cite{vershynin_2012}. Since $\|Z\|_{\psi_2}\leq L$, Relation \eqref{eq:Concentration} holds with probability greater than $1-2\exp(- ct^2 )$, where $c=L^{-4}c_1$. Finally, as mentioned at Step 3 of the proof in \cite{vershynin_2012}, any constant $C$ such that $C \geq \|Z\|_{\psi_2}^2\sqrt{\log (9) / c_1}$ gives the result, so we can choose $C = L^2 \sqrt{\log (9) / c_1}$. 
\end{proof}

We now give the proof of Proposition \ref{prop:ControlReconstructionError}.

\begin{proof}
 Assume we are given an approximation $\hat H$ of $H$ such that
 \begin{equation}\label{eq:control_Hhat}
  (1-\delta) H \preceq  \hat H \preceq (1+\delta) H,
 \end{equation}
 for some $\delta<1$. For any $P_r$ we can write $\mathcal{R}_{\Gamma}(P_r,H) = \text{trace}(\Gamma^{-1}(I_d-P_r^T)H (I_d-P_r)) = \text{trace}(HB)$ with $B = (I_d-P_r)\Gamma^{-1}(I_d-P_r^T)$. Since $B$ is a symmetric matrix with $B \succeq 0$, the relation $\text{trace}(H_1B) \leq \text{trace}(H_2B)$ holds for any symmetric matrices $H_1$ and $H_2$ such that $H_1 \preceq H_2$. Then Relation \eqref{eq:control_Hhat} yields \eqref{eq:ControlReconstructionError}. It remains to show that, under the assumptions \eqref{eq:SubgaussianNablaLogF} and \eqref{eq:KlowerBound}, the $K$-sample Monte Carlo estimate $\hat H$ of $H$ satisfies \eqref{eq:control_Hhat} with probability greater than $1-\eta$. 
 \\

 Let $n=\text{rank}(H)$ and let $G\in\mathbb{R}^{d\times n}$ be a full column rank matrix\footnote{For instance $G = U D^{1/2}$ where $D\in\mathbb{R}^{n\times n}$ is the diagonal matrix containing the non-zeros eigenvalues of the SPD matrix $H$ and where $U\in\mathbb{R}^{d\times n}$ is the matrix whose columns are the corresponding eigenvectors.} such that $H=GG^T$. Since $\text{Im}(H) = \text{Im}(G)$ and since $\nabla \log f(X) \in \text{Im}(H)$ almost surely, where $X\sim \nu$, there exists a random vector $Z$ in $\mathbb{R}^n$ such that $GZ = \nabla\log f(X)$. We can write
 $$
  G^TG\mathbb{E}( ZZ^T )G^TG = G^T\mathbb{E}( (GZ)(GZ)^T ) G = G^THG = (G^TG)^2,
 $$
 so that, since $G^TG$ is invertible ($G$ is full column rank), we have $\mathbb{E}( ZZ^T ) = I_n$. Then $Z$ is isotropic. For any $w\in\mathbb{R}^n$, Relation \eqref{eq:SubgaussianNablaLogF} allows writing
 $$
  \| w^T Z \|_{\psi_2}
  = \| u^T \nabla \log f(X) \|_{\psi_2} 
  \leq  L \sqrt{ u^T H u }
  = L \|w\|,
 $$
 where $u\in\mathbb{R}^d$ is any vector such that $w=G^Tu$. Then $Z$ is sub-Gaussian and satisfies \eqref{eq:SubGaussianZ} with the same $L$ as in \eqref{eq:SubgaussianNablaLogF}. Thus Lemma \ref{lem:Concentration} ensures that, for any $t\geq0$ and with a probability greater than $1-2\exp(-ct^2)$, the matrix $\hat\Sigma = \frac{1}{K}\sum_{k=1}^K Z_k Z_k^T$ satisfies \eqref{eq:Concentration}. Let $\delta\in(0,1)$ and assume
 \begin{equation}\label{eq:tmp052785}
  K\geq \delta^{-2} (C \sqrt{n} + t)^{2}.
 \end{equation}
 This is equivalent to $\tau \leq \delta $ where $\tau=C \sqrt{nK^{-1}} + t\sqrt{K^{-1}}$. Since $\delta<1$ we have $\max\{\tau,\tau^2\} = \tau\leq \delta$ so that \eqref{eq:Concentration} yields $(1-\delta)I_n \preceq\hat\Sigma \preceq (1+\delta ) I_n$. By multiplying by $G$ from the left and by $G^T$ from the right, we obtain \eqref{eq:control_Hhat}. Then for any $t\geq0$, the condition \eqref{eq:tmp052785} ensures that \eqref{eq:control_Hhat} holds with probability at least $1-2\exp(-ct^2)$.
 Let $\eta\in(0,1)$ and chose $t$ such that $\eta = 2\exp(-ct^2)$, meaning $t=\sqrt{c^{-1}\log(2\eta^{-1})}$. Since $c = L^{-4}c_1$ and $C=L^2\sqrt{\log (9) / c_1}$ for some absolute constant $c_1$, we can write
 \begin{align*}
  \delta^{-2} (C \sqrt{n} + t)^{2}
  &= \delta^{-2} (C \sqrt{n} + \sqrt{c^{-1}\log(2\eta^{-1})})^{2} \\
  &\leq \Omega \delta^{-2} L^4 (  \sqrt{\text{rank}(H)} + \sqrt{\log(2\eta^{-1})})^{2}
 \end{align*}
 where $\Omega = \log(9)/c_1$. Then \eqref{eq:KlowerBound} implies \eqref{eq:tmp052785} so that \eqref{eq:control_Hhat} holds with probability greater than $1-\eta$, which concludes the proof of Proposition \ref{prop:ControlReconstructionError}.
\end{proof}

\subsection{Proof of Proposition \ref{prop:ExplicitCondExp_muPrime}}\label{proof:ExplicitCondExp_muPrime}

 Let $F : x\mapsto \mathbb{E} ( f(P_r x + (I_d-P_r)Y ) )$ where $Y\sim\mu$ and let $h$ be any $\sigma(P_r)$-measurable function. By Doob-Dynkin's lemma, see Lemma \ref{cor:Doob}, $h$ can be written as a composition of a function with $P_r$, so that $h(P_r x + (I_d-P_r)y) = h(x)$ holds for all $x,y\in\mathbb{R}^d$. $F$ also satisfies the same property. Using these observations and using the definition of $\mu'$, we can write
 \begin{align*}
  \int &  F\,h \,\mathrm{d}\mu' = \int \int F(P_r x + (I_d-P_r)y )  \, h(P_r x + (I_d-P_r)y) \,\mu(\mathrm{d}x)\mu(\mathrm{d}y) \\
  &= \int F( x ) h( x ) \,\mu(\mathrm{d}x) 
  = \int \Big(\int f(P_r x + (I_d-P_r)y ) \mu(\mathrm{d}y) \Big) h(x ) \, \mu(\mathrm{d}x) \\
  &= \int \int f(P_r x + (I_d-P_r)y ) \, h(P_r x +(I_d-P_r)y ) \,\mu(\mathrm{d}x)\mu(\mathrm{d}y) = \int f \,h \,\mathrm{d}\mu' ,
 \end{align*}
 which shows that $F$ is the conditional expectation $\mathbb{E}_{\mu'}(f|\sigma(P_r))$.

\subsection{Proof of Proposition \ref{prop:DecompExpectKL}}\label{proof:DecompExpectKL}
~
Because $f(x)>0$ for $\mu$-a.e.~$x$, we have that $\mathbb{E}_{\mu}(f|\sigma(P_r))(x)>0$ for $\mu$-a.e.~$x$. Also, any realization of $\hat F_r$ defined as in \eqref{eq:MonteCarloCond} satisfies $\hat F_r(x)>0$ for $\mu$-a.e.~$x$.
 Thus, we can decompose $\Dkl(\nu_r^*|| \hat\nu_r) $ as follows
 $$
  \Dkl(\nu_r^*|| \hat\nu_r) 
  = \Dkl(\nu_r^*|| \nu_r')  + \int \log \Big(\frac{\mathbb{E}_{\mu'}(f|\sigma(P_r))}{\hat F_r} \Big)\mathrm{d}\nu_r^* + \log \Big(\frac{\hat Z}{Z'} \Big),
 $$
 where $Z$, $Z'$ and $\hat Z$ are the normalizing constants associated with $\mathbb{E}_{\mu}(f|\sigma(P_r))$, $\mathbb{E}_{\mu'}(f|\sigma(P_r))$ and $\hat F_r$ respectively. Because $\mathbb{E}(\hat Z) = \mathbb{E}(\int \hat F_r \mathrm{d}\mu) = \int \mathbb{E}(\hat F_r )\mathrm{d}\mu = \int \mathbb{E}_{\mu'}(f|\sigma(P_r)) \mathrm{d}\mu = Z'$, Jensen's inequality allows writing $\mathbb{E} (\log (\hat Z)) \leq \log ( \mathbb{E} (\hat Z) ) =  \log (Z') $. Then, taking the expectation in the previous relation yields
 \begin{equation}\label{eq:tmp2986532}
  \mathbb{E}\big( \Dkl(\nu_r^*|| \hat\nu_r) \big) \leq \Dkl(\nu_r^*|| \nu_r') + \mathbb{E} \int \log \Big(\frac{\mathbb{E}_{\mu'}(f|\sigma(P_r))}{\hat F_r} \Big) \, \mathrm{d}\nu_r^* .
 \end{equation}
 Using the second order Taylor expansion $\log(1+x) = x - \frac{1}{2}x^2 + \mathcal{O}(x^3)$ and since $\mathbb{E}(\hat F_r - \mathbb{E}_{\mu'}(f|\sigma(P_r))) = 0$, we can write
 \begin{align*}
  \mathbb{E} \int \log \Big(  &  \frac{\mathbb{E}_{\mu'}(f|\sigma(P_r))}{\hat F_r} \Big)  \, \mathrm{d}\nu_r^* 
  = - \mathbb{E}  \int \log \Big(1+\frac{\hat F_r - \mathbb{E}_{\mu'}(f|\sigma(P_r))}{\mathbb{E}_{\mu'}(f|\sigma(P_r))} \Big) \, \mathrm{d}\nu_r^*  \\
  &= 
  \mathbb{E}  \int \frac{1}{2}  \Big( \frac{\hat F_r - \mathbb{E}_{\mu'}(f|\sigma(P_r))}{\mathbb{E}_{\mu'}(f|\sigma(P_r))} \Big)^2  + \mathcal{O}\left( \Big( \frac{\hat F_r - \mathbb{E}_{\mu'}(f|\sigma(P_r))}{\mathbb{E}_{\mu'}(f|\sigma(P_r))} \Big)^3  \right)  \, \mathrm{d}\nu_r^* ,
 \end{align*}
 which together with \eqref{eq:tmp2986532} conclude the proof.

\subsection{Proof of Proposition \ref{prop:controlMC}}\label{proof:controlMC}

To prove this proposition, we will need the following lemmas and corollaries.

\begin{lemma}\label{lem:boundKLbyChi2}
 Let $\nu_1$ and $\nu_2$ be two probability distributions such that $\frac{\mathrm{d}\nu_1}{\mathrm{d}\mu} \propto f_1$ and $\frac{\mathrm{d}\nu_2}{\mathrm{d}\mu} \propto f_2$. Then we have
 $$
  \Dkl( \nu_1 || \nu_2 ) \leq \frac{Z_2}{Z_1^2}\int \frac{(f_1 - f_2 )^2}{f_2} \mathrm{d}\mu
 $$
 where $Z_1 = \int f_1~\mathrm{d}\mu$ and $Z_2 = \int f_2~\mathrm{d}\mu$ are normalizing constants.
\end{lemma}

\begin{proof}
 Let $\chi^2(\nu_1||\nu_2) = \int \frac{(f_1/Z_1 - f_2/Z_2)^2}{f_2/Z_2} \mathrm{d}\mu$ be the $\chi^2$-divergence from $\nu_2$ to $\nu_1$. Using Jensen's inequality, we have $\Dkl( \nu_1 || \nu_2 ) \leq \chi^2(\nu_1||\nu_2) $.
 For any $\alpha\in\mathbb{R}$ we can write
 $$
  \int \frac{(f_1/Z_1 - \alpha f_2/Z_2)^2}{f_2/Z_2} \mathrm{d}\mu  
  = \int \frac{(f_1/Z_1)^2 }{f_2/Z_2} \mathrm{d}\mu -2\alpha +\alpha^2 
  = \chi^2(\nu_1||\nu_2) + (1-\alpha)^2 .
 $$
 The choice $\alpha=Z_2/Z_1$ yields
 $
  \chi^2(\nu_1||\nu_2)
  \leq \int \frac{(f_1/Z_1 - f_2/Z_1)^2}{f_2/Z_2} \mathrm{d}\mu
 $
 which gives the result.
\end{proof}

\begin{lemma}\label{lem:bound_Expectation_Chi2}
 Let $\mu$ and $\mu'$ be two probability distributions with Lebesgue densities $\rho$ and $\rho'$. Then for any measurable function $h$ we have
 $$
  \Big( \int h ~\mathrm{d}\mu - \int h ~\mathrm{d}\mu' \Big)^2 
  \leq \chi^2(\mu||\mu') \int \Big(h-\int h \mathrm{d}\mu'\Big)^2\mathrm{d}\mu'
 $$
 where $\chi^2(\mu||\mu') = \int \frac{( \rho - \rho'  )^2}{\rho'}\mathrm{d}x $ denotes the $\chi^2$-divergence from $\mu'$ to $\mu$.
\end{lemma}

\begin{proof}
 Using Cauchy-Schwarz inequality we can write
 \begin{align*}
  \Big( \int h ~\mathrm{d}\mu - \int h ~\mathrm{d}\mu' \Big)^2 
  &= \Big( \int (h ~\sqrt{\rho'}) \Big(\frac{\rho-\rho'}{\sqrt{\rho'}}\Big) \mathrm{d}x \Big)^2  \\
  &\leq \Big(\int h^2 \rho'~ \mathrm{d}x \Big) \Big( \int \frac{( \rho - \rho'  )^2}{\rho'}\mathrm{d}x \Big) 
 \end{align*}
 Replacing $h$ by $h-\int h\mathrm{d}\mu'$ gives the result.
\end{proof}

\begin{lemma}\label{lem:Chi2control}
 Assume $\rho(x) = \exp( -\frac{1}{2}\|x-m\|^2_{\Gamma} - \Psi(x) )$ where $\Psi$ is a bounded function such that
 $\exp( \sup \Psi - \inf \Psi ) \leq \kappa$ and $m\in\mathbb{R}^d$. Given an $\|\cdot\|_\Gamma$-orthogonal projector $P_r$, let $U_r\in\mathbb{R}^{d\times r}$ and $U_\perp\in\mathbb{R}^{d\times (d-r)}$ be two matrices whose columns form a basis of $\text{Im}(P_r)$ and $\text{Ker}(P_r)$ respectively. Let
 \begin{align*}
  p_\perp(\xi_\perp) &= \frac{1}{C} \int_{\mathbb{R}^{r}} \rho( U_r\xi_r' + U_\perp\xi_\perp ) \mathrm{d}\xi_r' \\
  p_r(\xi_r) &= \frac{1}{C} \int_{\mathbb{R}^{d-r}} \rho( U_r\xi_r + U_\perp\xi_\perp' ) \mathrm{d}\xi_\perp' ,
 \end{align*}
 where $C = \int_{\mathbb{R}^{r}\times\mathbb{R}^{d-r}} \rho( U_r\xi_r' + U_\perp\xi_\perp' ) \mathrm{d}\xi_r'\mathrm{d}\xi_\perp'$.
 Then we have
 \begin{equation}\label{eq:tmp0000}
  \kappa^{-2} p_r(\xi_r)p_\perp(\xi_\perp)
  \leq
   \frac{1}{C}~\rho(U_r\xi_r  + U_\perp \xi_\perp  )
  \leq
  \kappa^{2} p_r(\xi_r)p_\perp(\xi_\perp).
 \end{equation}
 Furthermore the relations $$\kappa^{-2} p_\perp(\xi_\perp)\leq p_\perp(\xi_\perp|P_rx) \leq \kappa^{2} p_\perp(\xi_\perp) ,$$ and $$\chi^2( p_\perp(\cdot|P_rx) || p_\perp ) \leq \kappa^2 -1 ,$$ hold for any $x\in\mathbb{R}^d$, where
 $$
  p_\perp(\xi_\perp|P_r x) = \frac{\rho( P_r x + U_\perp \xi_\perp  )}{\int_{\mathbb{R}^{d-r}} \rho( P_r x + U_\perp \xi_\perp'  ) \mathrm{d}\xi_\perp'}.
 $$
\end{lemma}
\begin{proof}
 Assume for simplicity that $m=0$. (Replacing $x$ with $x-m$ at the end of the proof will cover the case $m\neq0$.)
 Let $\alpha = \exp(-\sup\Psi)$ and $\beta = \exp(-\inf\Psi)$ so that $\beta/\alpha\leq\kappa$. For any $\xi_r\in\mathbb{R}^r$, $\xi_\perp\in\mathbb{R}^{d-r}$ we have
 \begin{equation}\label{eq:tmp0001}
  \alpha~ \exp( -\frac{1}{2}\|U_r\xi_r + U_\perp \xi_\perp \|^2_{\Gamma}  )
   \leq 
   \rho(U_r\xi_r  + U_\perp \xi_\perp  ) 
   \leq 
   \beta~ \exp( -\frac{1}{2}\|U_r\xi_r + U_\perp \xi_\perp \|^2_{\Gamma} ).
 \end{equation}
 Using the fact that $\|U_r\xi_r + U_\perp \xi_\perp \|^2_{\Gamma} = \|U_r\xi_r\|^2_{\Gamma} + \|U_\perp \xi_\perp \|^2_{\Gamma}$, and integrating with respect to either $\xi_r$ or $\xi_\perp$, Equation \eqref{eq:tmp0001} yields
 \begin{align}
  \alpha C_r \exp( -\frac{1}{2}\|U_\perp\xi_\perp\|^2_{\Gamma}  )
   \leq 
   \underbrace{\int_{\mathbb{R}^{r}} \rho(U_r\xi_r' + U_\perp \xi_\perp ) \mathrm{d}\xi_r'}_{ p_\perp(\xi_\perp) C}
   \leq 
   \beta C_r \exp( -\frac{1}{2}\|U_\perp\xi_\perp\|^2_{\Gamma}  )   \label{eq:tmp0002}
   \\
   \alpha C_\perp \exp( -\frac{1}{2}\|U_r\xi_r\|^2_{\Gamma}  )
   \leq 
   \underbrace{\int_{\mathbb{R}^{d-r}} \rho(U_r\xi_r + U_\perp \xi_\perp' ) \mathrm{d}\xi_\perp'}_{ p_r(\xi_r) C}
   \leq 
   \beta C_\perp \exp( -\frac{1}{2}\|U_r\xi_r \|^2_{\Gamma}  ) ,   \label{eq:tmp0003}
 \end{align}
 where $C_r=\int_{\mathbb{R}^{d-r}} \exp( -\frac{1}{2}\|U_r \xi_r' \|^2_{\Gamma} ) \mathrm{d}\xi_r'$ and $C_\perp = \int_{\mathbb{R}^{d-r}}\exp(- \frac{1}{2}\|U_\perp \xi_\perp' \|^2_{\Gamma} )  \mathrm{d}\xi_\perp' $. By integrating \eqref{eq:tmp0003} over $\xi_r$ we get
 \begin{equation}\label{eq:tmp0004}
   \alpha C_r C_\perp 
   \leq 
   C = \int_{\mathbb{R}^{r}\times\mathbb{R}^{d-r}} \rho(U_r\xi_r' + U_\perp \xi_\perp' ) \mathrm{d}\xi_r'\mathrm{d}\xi_\perp'
   \leq 
   \beta C_r C_\perp .
 \end{equation}
 Combining \eqref{eq:tmp0001}, \eqref{eq:tmp0002}, \eqref{eq:tmp0003} we obtain
  \begin{align*}
 \frac{1}{C}\rho(U_r\xi_r  + U_\perp \xi_\perp  ) 
   &\makebox[5em]{$\overset{\eqref{eq:tmp0001}}{\leq}$} \frac{\beta}{C}~ \exp( -\frac{1}{2}\|U_r\xi_r  \|^2_{\Gamma} ) \exp( -\frac{1}{2}\| U_\perp \xi_\perp \|^2_{\Gamma} ) \\
   &\makebox[5em]{$\overset{\eqref{eq:tmp0002}\&\eqref{eq:tmp0003}}{\leq}$} \frac{\beta}{C}  \left( \frac{C }{\alpha C_\perp}  p_r(\xi_r) \right)  \left(\frac{C }{\alpha C_r} p_\perp(\xi_\perp) \right) \\
   &\makebox[5em]{$\overset{\eqref{eq:tmp0004}}{\leq}$} \left(\frac{\beta }{\alpha }\right)^2   p_\perp(\xi_\perp)p_r(\xi_r) \leq \kappa^2   p_\perp(\xi_\perp)p_r(\xi_r),
 \end{align*}
 which gives the right inequality of \eqref{eq:tmp0000}. The left inequality of \eqref{eq:tmp0000} is obtained in a similar way.
 Finally, dividing \eqref{eq:tmp0000} by $p_r(\xi_r)$ and letting $\xi_r$ such that $U_r\xi_r = P_rx$ for some $x\in\mathbb{R}^d$, we can write $\kappa^{-2} p_\perp(\xi_\perp)\leq p_\perp(\xi_\perp|P_rx) \leq \kappa^{2} p_\perp(\xi_\perp) $ and then
 \begin{align*}
  \chi^2( p_\perp(\cdot|P_rx) || p_\perp ) 
  &= \int \frac{( p_\perp(\xi_\perp|P_rx) - p_\perp(\xi_\perp))^2}{p_\perp(\xi_\perp)} \mathrm{d}\xi_\perp 
  =  \int \frac{p_\perp(\xi_\perp|P_rx)^2}{p_\perp(\xi_\perp)} \mathrm{d}\xi_\perp -1  \\
  &= \int \frac{p_\perp(\xi_\perp|P_rx)}{p_\perp(\xi_\perp)} p_\perp(\xi_\perp|P_rx)\mathrm{d}\xi_\perp -1 \leq \kappa^2-1 .
 \end{align*}
 This concludes the proof.

\end{proof}

As first observed in \cite{rothaus1978lower}, a distribution which satisfies the logarithmic Sobolev inequality also satisfies the Poincaré inequality; see for instance Proposition 5.1.3 in \cite{bakry2013analysis}. The next proposition shows that the subspace logarithmic Sobolev inequality \eqref{eq:subspaceLogSob} implies the subspace Poincaré inequality.

\begin{corollary}[Subspace Poincaré inequality]\label{cor:subspacePoincare}
 Let $\mu$ be a probability distribution which satisfies the subspace logarithmic Sobolev inequality \eqref{eq:subspaceLogSob} for any continuously differentiable function $h$. Then $\mu$ satisfies the following \emph{subspace Poincaré inequality}:
 \begin{equation}\label{eq:subspacePoincare}
  \int \big( h - \mathbb{E}_{\mu}(h|\sigma(P_r)) \big)^2 \mathrm{d}\mu\leq \kappa \int \| (I_d - P_r^T)\nabla h \|_{\Gamma^{-1}}^2 \mathrm{d}\mu,
 \end{equation}
 for any continuously differentiable function $h$.
\end{corollary}
\begin{proof}
 Replacing $h$ with $1+\varepsilon h $ in the subspace logarithmic Sobolev inequality \eqref{eq:subspaceLogSob} we obtain
 \begin{equation}\label{eq:tmp7324941}
  \int (1+\varepsilon h)^2 \log \frac{(1+\varepsilon h)^2}{\mathbb{E}_\mu( (1+\varepsilon h)^2 |\sigma(P_r))} ~\mathrm{d}\mu 
 \leq 2 \varepsilon^2 \kappa \int \| (I_d - P_r^T)\nabla h \|_{\Gamma^{-1}}^2 ~\mathrm{d}\mu.
 \end{equation}
 A Taylor expansion of the logarithm as $\varepsilon$ goes to zero permits us to write
 $$
  \log(1+\varepsilon h)^2 = 2 \varepsilon h - \varepsilon^2 h^2 + \mathcal{O}(\varepsilon^3),
 $$
 and
 \begin{align*}
  \log\mathbb{E}_\mu( (1+\varepsilon h)^2 |\sigma(P_r)) 
  &= 2\varepsilon\mathbb{E}_\mu(h |\sigma(P_r)) +\varepsilon^2\mathbb{E}_\mu(h^2 |\sigma(P_r)) 
  - 2\varepsilon^2\mathbb{E}_\mu(h |\sigma(P_r))^2 + \mathcal{O}(\varepsilon^3) .
 \end{align*}
 Thus, the left-hand side of \eqref{eq:tmp7324941} becomes
 \begin{align*}
  \int  & (1+\varepsilon h)^2 \log\frac{(1+\varepsilon h)^2}{\mathbb{E}_\mu( (1+\varepsilon h)^2 |\sigma(P_r))} ~\mathrm{d}\mu\\
  &=  \int \big(1+2\varepsilon h + \mathcal{O}(\varepsilon^2)\big) \Bigg[2\varepsilon\Big(h-\mathbb{E}_\mu(h |\sigma(P_r))\Big) \\
  &\qquad\qquad\qquad+\varepsilon^2\Big(-h^2-\mathbb{E}_\mu(h^2 |\sigma(P_r)) + 2\mathbb{E}_\mu(h |\sigma(P_r))^2\Big) + \mathcal{O}(\varepsilon^3) \Bigg] ~\mathrm{d}\mu \\  
  &=  \varepsilon \int  2 \Big(h-\mathbb{E}_\mu(h |\sigma(P_r))\Big) ~\mathrm{d}\mu \\
  &+ \varepsilon^2 \int 4\Big(h^2 - h \mathbb{E}_\mu(h |\sigma(P_r))\Big)+\Big(-h^2-\mathbb{E}_\mu(h^2 |\sigma(P_r)) + 2\mathbb{E}_\mu(h |\sigma(P_r))^2\Big)~\mathrm{d}\mu \\
  &+\mathcal{O}(\varepsilon^3).
 \end{align*}
 Using the identities $\int \mathbb{E}_\mu(  h |\sigma(P_r)) \mathrm{d}\mu = \int h \mathrm{d}\mu$,  $\int \mathbb{E}_\mu(  h^2 |\sigma(P_r)) \mathrm{d}\mu = \int h^2 \mathrm{d}\mu$ and $\int h \mathbb{E}_\mu(  h |\sigma(P_r))  \mathrm{d}\mu=\int \mathbb{E}_\mu(  h |\sigma(P_r))^2 \mathrm{d}\mu$, the above equality simplifies to
 \begin{align*}
  \int  (1+\varepsilon h)^2 \log\frac{(1+\varepsilon h)^2}{\mathbb{E}_\mu( (1+\varepsilon h)^2 |\sigma(P_r))} ~\mathrm{d}\mu
  &= \varepsilon^2 \int 2h^2 - 2 \mathbb{E}_\mu(h |\sigma(P_r))~\mathrm{d}\mu +\mathcal{O}(\varepsilon^3) \\
  &= 2 \varepsilon^2 \int \big(h- \mathbb{E}_\mu(h |\sigma(P_r))\big)^2~\mathrm{d}\mu +\mathcal{O}(\varepsilon^3).
 \end{align*}
 Substituting this relation into \eqref{eq:tmp7324941} yields \eqref{eq:subspacePoincare}. This concludes the proof.
\end{proof}

We now have all the material to prove Proposition \ref{prop:controlMC}.

\begin{proof}

 We first show that \eqref{eq:controlMC_Bais} holds. By Lemma \ref{lem:boundKLbyChi2} we have
\begin{equation}\label{eq:tmp05721}
 \Dkl(\nu_r^*||\nu_r') \leq \frac{Z'}{Z^2} \int \frac{( \mathbb{E}_{\mu}(f|\sigma(P_r)) - \mathbb{E}_{\mu'}(f|\sigma(P_r)) )^2}{\mathbb{E}_{\mu'}(f|\sigma(P_r))} \mathrm{d}\mu,
\end{equation}
where $Z$ and $Z'$ are normalizing constants associated with $\mathbb{E}_{\mu}(f|\sigma(P_r))$ and $\mathbb{E}_{\mu'}(f|\sigma(P_r))$ respectively. Fix $x\in\mathbb{R}^d$. Recall that Proposition \ref{prop:ExplicitCondExp} yields
\begin{equation}
 \mathbb{E}_{\mu}(f|\sigma(P_r))(x) = \int_{\mathbb{R}^{d-r}} f(P_rx + U_\perp \xi_\perp ) p_\perp(\xi_\perp|P_rx)\mathrm{d}\xi_\perp \label{eq:tmp602561} 
\end{equation}
and using $p_\perp(\xi_\perp)$ defined in Lemma \ref{lem:Chi2control}, we have
\begin{equation}
 \mathbb{E}_{\mu'}(f|\sigma(P_r))(x) = \int_{\mathbb{R}^{d-r}} f(P_rx + U_\perp \xi_\perp ) p_\perp(\xi_\perp)\mathrm{d}\xi_\perp \label{eq:tmp602562}
\end{equation}
by applying the definition of conditional expectation in Proposition \ref{prop:ExplicitCondExp_muPrime}. 
Then, by Lemma \ref{lem:Chi2control}, the relation
\begin{equation}\label{eq:tmpCondExpNutoNuPrime}
 \kappa^{-2} ~\mathbb{E}_{\mu'}(f|\sigma(P_r))(x) \leq \mathbb{E}_{\mu}(f|\sigma(P_r))(x) \leq \kappa^2 ~\mathbb{E}_{\mu'}(f|\sigma(P_r))(x),
\end{equation}
holds for any $x$. Also, notice that both $\mathbb{E}_{\mu}(f|\sigma(P_r))(x)$ and $\mathbb{E}_{\mu'}(f|\sigma(P_r))(x)$ can be written as an expectation of $\xi_\perp\mapsto f(P_rx+U_\perp\xi_\perp)$ with respect to the probability densities $p_\perp(\cdot|P_rx)$ and $p_\perp$ respectively. Then, using Lemma \ref{lem:bound_Expectation_Chi2} and Lemma \ref{lem:Chi2control}, we have
\begin{align*}
 \big( \mathbb{E}_{\mu}( &  f|\sigma(P_r))(x) - \mathbb{E}_{\mu'}(f|\sigma(P_r))(x) \big)^2 \\
 &\leq (\kappa^2-1) \int_{\mathbb{R}^{d-r}} \big( f(P_rx+U_\perp\xi_\perp)- \mathbb{E}_{\mu'}(f|\sigma(P_r))(x) \big)^2 p_\perp(\xi_\perp)\mathrm{d}\xi_\perp \\
 &= (\kappa^2-1) \int \big( f(P_rx+(I_d-P_r)y)- \mathbb{E}_{\mu'}(f|\sigma(P_r))(x) \big)^2 \mu(\mathrm{d}y).
\end{align*}
Because $f(x)>0$ for $\mu$-a.e.\ $x$, we can divide by $\mathbb{E}_{\mu'}(f|\sigma(P_r))(x)>0$ and integrate with respect to $\mathrm{d}\mu(x)$ in order to obtain
\begin{align}
 \int & \frac{ ( \mathbb{E}_{\mu}(f|\sigma(P_r))(x)  - \mathbb{E}_{\mu'}(f|\sigma(P_r))(x) )^2}{\mathbb{E}_{\mu'}(f|\sigma(P_r))(x)} \mathrm{d}\mu(\mathrm{d}x) \nonumber\\
 &\leq (\kappa^2-1) \int \int \frac{(f(P_rx+(I_d-P_r)y)-\mathbb{E}_{\mu'}(f|\sigma(P_r))(x) )^2}{\mathbb{E}_{\mu'}(f|\sigma(P_r))(x)}  \mu(\mathrm{d}y)\mu(\mathrm{d}x) \label{eq:tmp9672}\\
 &= (\kappa^2-1) \int \int \frac{(f(P_rx+(I_d-P_r)y)-\mathbb{E}_{\mu'}(f|\sigma(P_r))(P_rx+(I_d-P_r)y) )^2}{\mathbb{E}_{\mu'}(f|\sigma(P_r))(P_rx+(I_d-P_r)y)}  \mu(\mathrm{d}y)\mu(\mathrm{d}x) \label{eq:tmp6927}\\
 &= (\kappa^2-1) \int \frac{(f-\mathbb{E}_{\mu'}(f|\sigma(P_r)) )^2}{\mathbb{E}_{\mu'}(f|\sigma(P_r))}  \mathrm{d}\mu' \label{eq:tmp096572}
\end{align}

Going from \eqref{eq:tmp9672} to \eqref{eq:tmp6927}, we used the fact that, since $\mathbb{E}_{\mu'}(f|\sigma(P_r))$ is a $\sigma(P_r)$-measurable function, it can be writen as the composition of some function with $P_r$, so that the relation $\mathbb{E}_{\mu'}(f|\sigma(P_r))(x) = \mathbb{E}_{\mu'}(f|\sigma(P_r))(P_rx+(I_d-P_r)y)$ holds for any $x,y\in\mathbb{R}^d$. To obtain \eqref{eq:tmp096572}, we used the definition of $\mu'$, see Equation \eqref{eq:defMuPrime}. The last term in \eqref{eq:tmp096572} satisfies
\begin{align}
 \int \frac{(f-\mathbb{E}_{\mu'}(f|\sigma(P_r)) )^2}{\mathbb{E}_{\mu'}(f|\sigma(P_r))}  & \mathrm{d}\mu'
 = \int \Big( \frac{f}{\mathbb{E}_{\mu'}(f|\sigma(P_r))^{1/2}}-\mathbb{E}_{\mu'}(f|\sigma(P_r))^{1/2} \Big)^2  \mathrm{d}\mu' \label{eq:tmp123_1}\\
 &\leq \int \Big( \frac{f}{\mathbb{E}_{\mu'}(f|\sigma(P_r))^{1/2}}-\frac{\mathbb{E}_{\mu}(f|\sigma(P_r))}{\mathbb{E}_{\mu'}(f|\sigma(P_r))^{1/2}} \Big)^2  \mathrm{d}\mu' \label{eq:tmp123_2}\\
 &\leq \kappa^2 \int \Big( \frac{f}{\mathbb{E}_{\mu'}(f|\sigma(P_r))^{1/2}}-\frac{\mathbb{E}_{\mu}(f|\sigma(P_r))}{\mathbb{E}_{\mu'}(f|\sigma(P_r))^{1/2}} \Big)^2  \mathrm{d}\mu \label{eq:tmp123_31}\\
 &\leq \kappa^4 \int \Big( \frac{f}{\mathbb{E}_{\mu}(f|\sigma(P_r))^{1/2}}-\frac{\mathbb{E}_{\mu}(f|\sigma(P_r))}{\mathbb{E}_{\mu}(f|\sigma(P_r))^{1/2}} \Big)^2  \mathrm{d}\mu \label{eq:tmp123_3}\\
 &= \kappa^4 \int \Big( \frac{f}{\mathbb{E}_{\mu}(f|\sigma(P_r))^{1/2}}- \mathbb{E}_{\mu}(f|\sigma(P_r))^{1/2} \Big)^2  \mathrm{d}\mu \label{eq:tmp123_4}\\
 &\leq \kappa^5 \int \|(I_d-P_r^T)\nabla \Big( \frac{f}{\mathbb{E}_{\mu}(f|\sigma(P_r))^{1/2}} \Big) \|_{\Gamma^{-1}}^2 \mathrm{d}\mu \label{eq:tmp123_5} \\
 &= \kappa^5 \int \|(I_d-P_r^T)\nabla \log f \|_{\Gamma^{-1}}^2 \frac{f^2}{\mathbb{E}_{\mu}(f|\sigma(P_r))} \mathrm{d}\mu  \label{eq:tmp123_6}
\end{align}
Let us explain the previous steps. 
\begin{itemize}
 
 \item To go from \eqref{eq:tmp123_1} to \eqref{eq:tmp123_2} notice that $\mathbb{E}_{\mu'}(f|\sigma(P_r))^{1/2}$ is a $\sigma(P_r)$-measurable function which satisfies $\mathbb{E}_{\mu'}(f|\sigma(P_r))^{1/2} = \mathbb{E}_{\mu'}( f\mathbb{E}_{\mu'}(f|\sigma(P_r))^{-1/2} |\sigma(P_r)) $. In particular, $\mathbb{E}_{\mu'}(f|\sigma(P_r))^{1/2}$ is a best approximation of $f\mathbb{E}_{\mu'}(f|\sigma(P_r))^{-1/2}$ among all $\sigma(P_r)$-measurable functions with respect to the $L^2_{\mu'}$-norm. Then \eqref{eq:tmp123_2} follows from the fact that $\frac{\mathbb{E}_{\mu}(f|\sigma(P_r))}{\mathbb{E}_{\mu'}(f|\sigma(P_r))^{1/2}}$ is $\sigma(P_r)$-measurable.
 
 \item To go from \eqref{eq:tmp123_2} to \eqref{eq:tmp123_3} we apply Lemma \ref{lem:Chi2control} twice as follows. First, since the integrand is positive, \eqref{eq:tmp0000} permits us to go from \eqref{eq:tmp123_2} to \eqref{eq:tmp123_31}. Second, with the same argument, \eqref{eq:tmp0000} yields $\mathbb{E}_{\mu}(f|\sigma(P_r)) \leq \kappa^{2} \mathbb{E}_{\mu'}(f|\sigma(P_r))$, which permits us to go from \eqref{eq:tmp123_31} to \eqref{eq:tmp123_3}.
 
 \item From \eqref{eq:tmp123_4} to \eqref{eq:tmp123_5}, we applied the subspace Poincaré inequality \eqref{eq:subspacePoincare} with $h= f\mathbb{E}_{\mu}(f|\sigma(P_r)^{-1/2}$. Indeed, since $\mathbb{E}_{\mu}( h |\sigma(P_r)) = \mathbb{E}_{\mu}(f|\sigma(P_r))^{1/2} $, the right-hand side of \eqref{eq:tmp123_4} can be written as $\kappa^{4}\int(h-\mathbb{E}_{\mu}( h |\sigma(P_r)))^2\mathrm{d}\mu$, which is $\kappa^{4}$ times the left hand side of \eqref{eq:subspacePoincare}.

\end{itemize}
Together with \eqref{eq:tmp05721} and \eqref{eq:tmp096572}, \eqref{eq:tmp123_6} yields 
\begin{equation}\label{eq:tmp5914}
 \Dkl(\nu_r^*||\hat\nu_r) \leq  \kappa^5 (\kappa^2-1) \frac{Z'}{Z}\int \|(I_d-P_r^T)\nabla \log f \|_{\Gamma^{-1}}^2 \frac{f}{\mathbb{E}_{\mu}(f|\sigma(P_r))} \frac{f}{Z}\mathrm{d}\mu .
\end{equation}
By integrating \eqref{eq:tmpCondExpNutoNuPrime} with respect to $\mathrm{d}\mu(x)$ we obtain $Z'\leq \kappa^2 Z$, so that \eqref{eq:tmp5914} yields \eqref{eq:controlMC_Bais}.
\\

We now show that \eqref{eq:controlMC_MeanSquaredError} holds. First note that, for any $x\in\mathbb{R}^d$, the independence of the $Y_i$'s in the definition \eqref{eq:MonteCarloCond} of $\hat F_r$ ensures that
\begin{align*}
 \mathbb{E} \Big(\Big(  & \frac{\hat F_r(x) - \mathbb{E}_{\mu'}(f|\sigma(P_r))(x)}{\mathbb{E}_{\mu'}(f|\sigma(P_r))(x)}   \Big)^2\Big) \\
 &= \frac{1}{M} \int \Big(  \frac{f(P_rx+(I_d-P_r)y) - \mathbb{E}_{\mu'}(f|\sigma(P_r))(x)}{\mathbb{E}_{\mu'}(f|\sigma(P_r))(x)}   \Big)^2  \mu(\mathrm{d}y).
\end{align*}
Multiplying by $\mathbb{E}_{\mu}(f|\sigma(P_r))(x)/Z$ and taking the expectation with respect to $x\sim \mu$ we have
\begin{align}
 \int  & \frac{\mathbb{E}_{\mu}(f|\sigma(P_r))}{Z}   \mathbb{E}  \Big(  \frac{\hat F_r - \mathbb{E}_{\mu'}(f|\sigma(P_r))}{\mathbb{E}_{\mu'}(f|\sigma(P_r))}   \Big)^2 \mathrm{d}\mu \nonumber\\
 &= \frac{1}{M}
 \int \frac{\mathbb{E}_{\mu}(f|\sigma(P_r))(x)}{Z} \int \Big(  \frac{f(P_rx+(I_d-P_r)y) - \mathbb{E}_{\mu'}(f|\sigma(P_r))(x)}{\mathbb{E}_{\mu'}(f|\sigma(P_r))(x)}   \Big)^2  \mu(\mathrm{d}y)  \mu(\mathrm{d}x) \nonumber\\
 &= \frac{1}{M}
 \int \frac{\mathbb{E}_{\mu}(f|\sigma(P_r))}{Z} \Big(  \frac{f  - \mathbb{E}_{\mu'}(f|\sigma(P_r))}{\mathbb{E}_{\mu'}(f|\sigma(P_r))}   \Big)^2  \mathrm{d} \mu' \label{eq:tmp00000001} \\
 &\leq \frac{\kappa^2}{MZ} \int \frac{( f  - \mathbb{E}_{\mu'}(f|\sigma(P_r)) )^2}{\mathbb{E}_{\mu'}(f|\sigma(P_r))}  \mathrm{d} \mu' \label{eq:tmp00000002} \\
 &\leq \frac{\kappa^7}{MZ} \int \|(I_d-P_r^T)\nabla \log f \|_{\Gamma^{-1}}^2 \frac{f^2}{\mathbb{E}_{\mu}(f|\sigma(P_r))} \mathrm{d}\mu  \label{eq:tmp00000003}
\end{align}
Let us detail the previous steps. To get \eqref{eq:tmp00000001} we can use similar arguments as when going from \eqref{eq:tmp9672} to \eqref{eq:tmp096572}, meaning exploiting the $\sigma(P_r)$-measurablility of $\mathbb{E}_{\mu'}(f|\sigma(P_r))$ and of $\mathbb{E}_{\mu}(f|\sigma(P_r))$ and the property \eqref{eq:defMuPrime} of $\mu'$. Going from \eqref{eq:tmp00000001} to \eqref{eq:tmp00000002}, we used relation \eqref{eq:tmpCondExpNutoNuPrime}. Using \eqref{eq:tmp123_6}, we get \eqref{eq:tmp00000003}, which yields \eqref{eq:controlMC_MeanSquaredError} and concludes the proof of Proposition \ref{prop:controlMC}.
\end{proof}

\bibliographystyle{amsplain}

\providecommand{\bysame}{\leavevmode\hbox to3em{\hrulefill}\thinspace}
\providecommand{\MR}{\relax\ifhmode\unskip\space\fi MR }
\providecommand{\MRhref}[2]{%
  \href{http://www.ams.org/mathscinet-getitem?mr=#1}{#2}
}
\providecommand{\href}[2]{#2}
\begin{thebibliography}{}

\end{thebibliography}


\begin{thebibliography}{10}

\bibitem{agapiou2017importance}
Sergios Agapiou, Omiros Papaspiliopoulos, Daniel Sanz-Alonso, and AM~Stuart,
  \emph{Importance sampling: Intrinsic dimension and computational cost},
  Statistical Science (2017), 405--431.

\bibitem{atchade2006adaptive}
Yves~F Atchad{\'e}, \emph{An adaptive version for the {M}etropolis adjusted
  {L}angevin algorithm with a truncated drift}, Methodology and Computing in
  applied Probability {8} (2006), no.~2, 235--254.

\bibitem{bakry1985diffusions}
Dominique Bakry and Michel {\'E}mery, \emph{Diffusions hypercontractives},
  Seminaire de probabilit{\'e}s XIX 1983/84, Springer, 1985, pp.~177--206.

\bibitem{bakry2013analysis}
Dominique Bakry, Ivan Gentil, and Michel Ledoux, \emph{Analysis and geometry of
  Markov diffusion operators}, vol. 348, Springer Science \& Business Media,
  2013.

\bibitem{banerjee2005optimality}
Arindam Banerjee, Xin Guo, and Hui Wang, \emph{On the optimality of conditional
  expectation as a Bregman predictor}, IEEE Transactions on Information Theory
  {51} (2005), no.~7, 2664--2669.

\bibitem{beskos2018multilevel}
Alexandros Beskos, Ajay Jasra, Kody Law, Youssef Marzouk, and Yan Zhou,
  \emph{Multilevel sequential Monte Carlo with dimension-independent
  likelihood-informed proposals}, SIAM/ASA Journal on Uncertainty
  Quantification {6} (2018), no.~2, 762--786.

\bibitem{blanchard2007statistical}
Gilles Blanchard, Olivier Bousquet, and Laurent Zwald, \emph{Statistical
  properties of kernel principal component analysis}, Machine Learning
  {66} (2007), no.~2-3, 259--294.

\bibitem{bobkov2000brunn}
Sergey~G Bobkov and Michel Ledoux, \emph{From Brunn-Minkowski to Brascamp-Lieb
  and to logarithmic Sobolev inequalities}, Geometric \& Functional Analysis
  GAFA {10} (2000), no.~5, 1028--1052.

\bibitem{boucheron2013concentration}
St{\'e}phane Boucheron, G{\'a}bor Lugosi, and Pascal Massart,
  \emph{Concentration inequalities: A nonasymptotic theory of independence},
  Oxford University Press, 2013.

\bibitem{brennan2020greedy}
Michael Brennan, Daniele Bigoni, Olivier Zahm, Alessio Spantini, and Youssef
  Marzouk, \emph{Greedy inference with structure-exploiting lazy maps},
  Advances in Neural Information Processing Systems {33} (2020).

\bibitem{chen2016accelerated}
Yuxin Chen, David Keyes, Kody~JH Law, and Hatem Ltaief, \emph{Accelerated
  dimension-independent adaptive Metropolis}, SIAM Journal on Scientific
  Computing {38} (2016), no.~5, S539--S565.

\bibitem{christen2005Markov}
J~Andr{\'e}s Christen and Colin Fox, \emph{Markov chain Monte Carlo using an
  approximation}, Journal of Computational and Graphical statistics {14}
  (2005), no.~4, 795--810.

\bibitem{cohen2012capturing}
Albert Cohen, Ingrid Daubechies, Ronald DeVore, Gerard Kerkyacharian, and
  Dominique Picard, \emph{Capturing ridge functions in high dimensions from
  point queries}, Constructive Approximation {35} (2012), no.~2,
  225--243.

\bibitem{conrad2016accelerating}
Patrick Conrad, Youssef Marzouk, Natesh Pillai, and Aaron Smith,
  \emph{Accelerating asymptotically exact MCMC for computationally intensive
  models via local approximations}, Journal of the American Statistical
  Association {111} (2016), no.~516, 1591--1607.

\bibitem{constantine2014active}
Paul~G Constantine, Eric Dow, and Qiqi Wang, \emph{Active subspace methods in
  theory and practice: applications to kriging surfaces}, SIAM Journal on
  Scientific Computing {36} (2014), no.~4, A1500--A1524.

\bibitem{constantine2016accelerating}
Paul~G Constantine, Carson Kent, and Tan Bui-Thanh, \emph{Accelerating Markov
  chain Monte Carlo with active subspaces}, SIAM Journal on Scientific
  Computing {38} (2016), no.~5, A2779--A2805.

\bibitem{cui2011Bayesian}
Tiangang~Cui, Colin~Fox, and Michael~O'Sullivan, \emph{Bayesian calibration of a large-scale
  geothermal reservoir model by a new adaptive delayed acceptance Metropolis
  Hastings algorithm}, Water Resources Research {47} (2011), no.~10, W10521.

\bibitem{cui2019Aposteriori}
Tiangang~Cui, Colin~Fox, and Michael~O'Sullivan, \emph{A posteriori stochastic correction of reduced models in delayed‐acceptance MCMC, with application to multiphase subsurface inverse problems}, International Journal for Numerical Methods in Engineering {118} (2019), no.~10, 578--605.

\bibitem{cui2021conditional}
Tiangang Cui, Sergey Dolgov, and Olivier Zahm, \emph{Conditional deep inverse
  Rosenblatt transports}, arXiv preprint arXiv:2106.04170 (2021).

\bibitem{cui2016dimension}
Tiangang Cui, Kody~JH Law, and Youssef Marzouk, \emph{Dimension-independent
  likelihood-informed MCMC}, Journal of Computational Physics {304}
  (2016), 109--137.

\bibitem{cui2014likelihood}
Tiangang Cui, James Martin, Youssef Marzouk, Antti Solonen, and Alessio
  Spantini, \emph{Likelihood-informed dimension reduction for nonlinear inverse
  problems}, Inverse Problems {30} (2014), no.~11, 114015.

\bibitem{cui2016scalable}
Tiangang Cui, Youssef Marzouk, and Karen Willcox, \emph{Scalable posterior
  approximations for large-scale Bayesian inverse problems via
  likelihood-informed parameter and state reduction}, Journal of Computational
  Physics {315} (2016), 363--387.

\bibitem{cui2015data}
Tiangang Cui, Youssef Marzouk, and Karen Willcox, \emph{Data-driven model
  reduction for the Bayesian solution of inverse problems}, International
  Journal for Numerical Methods in Engineering {102} (2015), no.~5,
  966--990.

\bibitem{cui2021data}
Tiangang Cui and Olivier Zahm, \emph{Data-free likelihood-informed dimension
  reduction of Bayesian inverse problems}, Inverse Problems {37} (2021),
  no.~4, 045009.

\bibitem{del2006sequential}
Pierre Del~Moral, Arnaud Doucet, and Ajay Jasra, \emph{Sequential Monte Carlo
  samplers}, Journal of the Royal Statistical Society: Series B (Statistical
  Methodology) {68} (2006), no.~3, 411--436.

\bibitem{duane1987hybrid}
Simon Duane, Anthony~D Kennedy, Brian~J Pendleton, and Duncan Roweth,
  \emph{Hybrid Monte Carlo}, Physics Letters B {195} (1987), no.~2,
  216--222.

\bibitem{flath2011fast}
H~Pearl Flath, Lucas~C Wilcox, Volkan Ak{\c{c}}elik, Judith Hill, Bart van
  Bloemen~Waanders, and Omar Ghattas, \emph{Fast algorithms for Bayesian
  uncertainty quantification in large-scale linear inverse problems based on
  low-rank partial Hessian approximations}, SIAM Journal on Scientific
  Computing {33} (2011), no.~1, 407--432.

\bibitem{fornasier2012learning}
Massimo Fornasier, Karin Schnass, and Jan Vybiral, \emph{Learning functions of
  few arbitrary linear parameters in high dimensions}, Foundations of
  Computational Mathematics {12} (2012), no.~2, 229--262.

\bibitem{girolami2011riemann}
Mark Girolami and Ben Calderhead, \emph{Riemann manifold Langevin and
  Hamiltonian Monte Carlo methods}, Journal of the Royal Statistical Society:
  Series B (Statistical Methodology) {73} (2011), no.~2, 123--214.

\bibitem{gozlan2013characterization}
Nathael Gozlan, Cyril Roberto, and P-M Samson, \emph{Characterization of
  Talagrand's transport-entropy inequalities in metric spaces}, The Annals of
  Probability (2013), 3112--3139.

\bibitem{gross1975logarithmic}
Leonard Gross, \emph{Logarithmic Sobolev inequalities}, American Journal of
  Mathematics {97} (1975), no.~4, 1061--1083.

\bibitem{guionnet2003lectures}
Alice Guionnet and B~Zegarlinksi, \emph{Lectures on logarithmic Sobolev
  inequalities}, S{\'e}minaire de probabilit{\'e}s XXXVI, Springer, 2003,
  pp.~1--134.

\bibitem{haario2004Markov}
Heikki Haario, Marko Laine, Markku Lehtinen, Eero Saksman, and Johanna Tamminen, \emph{Markov chain Monte Carlo methods for high dimensional
  inversion in remote sensing}, Journal of the Royal Statistical Society:
  series B (statistical methodology) {66} (2004), no.~3, 591--607.

\bibitem{haario2001adaptive}
Heikki Haario, Eero Saksman, and Johanna Tamminen, \emph{An adaptive Metropolis algorithm}, Bernoulli {7} (2001), no.~2, 223--242.

\bibitem{hoffman2014no}
Matthew~D Hoffman and Andrew Gelman, \emph{The no-U-turn sampler: adaptively
  setting path lengths in Hamiltonian Monte Carlo.}, The Journal of Machine Learning Research
  {15} (2014), no.~1, 1593--1623.

\bibitem{holley1987logarithmic}
Richard Holley and Daniel Stroock, \emph{Logarithmic Sobolev inequalities and
  stochastic Ising models}, Journal of Statistical Physics {46} (1987),
  no.~5, 1159--1194.

\bibitem{jung1901ueber}
Heinrich Jung, \emph{Ueber die kleinste kugel, die eine r{\"a}umliche figur
  einschliesst.}, Journal f{\"u}r die reine und angewandte Mathematik (Crelles
  Journal) {1901} (1901), no.~123, 241--257.

\bibitem{kaipio2006statistical}
Jari Kaipio and Erkki Somersalo, \emph{Statistical and computational inverse
  problems}, vol. 160, Springer Science \& Business Media, 2006.

\bibitem{kallenberg1997foundations}
Olav Kallenberg, \emph{Foundations of modern probability},
  vol.~2, Springer, 1997.
  
\bibitem{kokiopoulou2011trace}
Effrosini Kokiopoulou, Jie Chen and Yousef Saad, \emph{Trace optimization and eigenproblems in dimension reduction methods}, Numerical Linear Algebra with Applications {18} (2011), 565--602.

  
\bibitem{kullback1997information}
Solomon Kullback, \emph{Information Theory and Statistics}, Courier Corporation, 1997.

\bibitem{lamminpaa2019likelihood}
Otto Lamminp{\"a}{\"a}, Marko Laine, Simo Tukiainen, and Johanna Tamminen,
  \emph{Likelihood informed dimension reduction for remote sensing of atmospheric constituent profiles}, 2017 MATRIX Annals, Springer, 2019,
  pp.~65--78.

\bibitem{ledoux1999concentration}
Michel Ledoux, \emph{Concentration of measure and logarithmic Sobolev
  inequalities}, Seminaire de probabilites XXXIII, Springer, 1999,
  pp.~120--216.

\bibitem{li2015note}
Jinglai Li, \emph{A note on the Karhunen--Lo{\`e}ve expansions for
  infinite-dimensional Bayesian inverse problems}, Statistics \& Probability
  Letters {106} (2015), 1--4.

\bibitem{li2006efficient}
Wei Li and Olaf~A Cirpka, \emph{Efficient geostatistical inverse methods for
  structured and unstructured grids}, Water Resources Research {42}
  (2006), no.~6.

\bibitem{lieberman2010parameter}
Chad Lieberman, Karen Willcox, and Omar Ghattas, \emph{Parameter and state
  model reduction for large-scale statistical inverse problems}, SIAM Journal
  on Scientific Computing {32} (2010), no.~5, 2523--2542.

\bibitem{lindgren2011explicit}
Finn Lindgren, H{\aa}vard Rue, and Johan Lindstr{\"o}m, \emph{An explicit link
  between Gaussian fields and Gaussian Markov random fields: the stochastic
  partial differential equation approach}, Journal of the Royal Statistical
  Society: Series B (Statistical Methodology) {73} (2011), no.~4,
  423--498.

\bibitem{manzoni2016accurate}
Andrea Manzoni, Stefano Pagani, and Toni Lassila, \emph{Accurate solution of
  Bayesian inverse uncertainty quantification problems combining reduced basis
  methods and reduction error models}, SIAM/ASA Journal on Uncertainty
  Quantification {4} (2016), no.~1, 380--412.

\bibitem{marzouk2009dimensionality}
Youssef Marzouk and Habib Najm, \emph{Dimensionality reduction and
  polynomial chaos acceleration of Bayesian inference in inverse problems},
  Journal of Computational Physics {228} (2009), no.~6, 1862--1902.

\bibitem{neal2011MCMC}
Radford Neal et~al., \emph{MCMC using Hamiltonian dynamics}, Handbook of
  Markov chain Monte Carlo {2} (2011), no.~11, 2.

\bibitem{otto2000generalization}
Felix Otto and C{\'e}dric Villani, \emph{Generalization of an inequality by
  Talagrand and links with the logarithmic Sobolev inequality}, Journal of
  Functional Analysis {173} (2000), no.~2, 361--400.

\bibitem{mcbook}
Art~B. Owen, \emph{Monte Carlo theory, methods and examples}, 2013.

\bibitem{pinkus2015ridge}
Allan Pinkus, \emph{Ridge functions}, vol.\ 205, Cambridge University Press,
  2015.

\bibitem{reiss2020nonasymptotic}
Markus Rei{\ss}, Martin Wahl, et~al., \emph{Nonasymptotic upper bounds for the
  reconstruction error of PCA}, Annals of Statistics {48} (2020), no.~2,
  1098--1123.

\bibitem{robert2013monte}
Christian Robert and George Casella, \emph{Monte Carlo statistical methods},
  Springer Science \& Business Media, 2013.

\bibitem{rothaus1978lower}
OS~Rothaus et~al., \emph{Lower bounds for eigenvalues of regular
  Sturm-Liouville operators and the logarithmic Sobolev inequality}, Duke
  Mathematical Journal {45} (1978), no.~2, 351--362.

\bibitem{rubio2018fast}
Paul-Baptiste Rubio, Fran{\c{c}}ois Louf, and Ludovic Chamoin, \emph{Fast model
  updating coupling Bayesian inference and PGD model reduction}, Computational
  Mechanics {62} (2018), no.~6, 1485--1509.

\bibitem{russi2010uncertainty}
Trent~Michael Russi, \emph{Uncertainty quantification with experimental data
  and complex system models}, Ph.D. thesis, UC Berkeley, 2010.

\bibitem{schillings2020convergence}
Claudia Schillings, Bj{\"o}rn Sprungk, and Philipp Wacker, \emph{On the
  convergence of the Laplace approximation and noise-level-robustness of
  Laplace-based Monte Carlo methods for Bayesian inverse problems}, Numerische
  Mathematik {145} (2020), no.~4, 915--971.

\bibitem{spantini2015optimal}
Alessio Spantini, Antti Solonen, Tiangang Cui, James Martin, Luis Tenorio, and
  Youssef Marzouk, \emph{Optimal low-rank approximations of Bayesian linear
  inverse problems}, SIAM Journal on Scientific Computing {37} (2015),
  no.~6, A2451--A2487.

\bibitem{stuart2010inverse}
Andrew~M Stuart, \emph{Inverse problems: a Bayesian perspective}, Acta Numerica
  {19} (2010), 451--559.

\bibitem{tamminen2004adaptive}
Johanna Tamminen, \emph{Adaptive Markov chain Monte Carlo algorithms with
  geophysical applications}, Ph.D. thesis, University of Helsinki, Faculty of
  Science, Department of Mathematics, 2004.

\bibitem{tyagi2014learning}
Hemant Tyagi and Volkan Cevher, \emph{Learning non-parametric basis independent
  models from point queries via low-rank methods}, Applied and Computational
  Harmonic Analysis {37} (2014), no.~3, 389--412.

\bibitem{vershynin_2012}
Roman Vershynin, \emph{Introduction to the non-asymptotic analysis of random
  matrices}, p.~210–268, Cambridge University Press, 2012.

\bibitem{zahm2020gradient}
Olivier Zahm, Paul~G Constantine, Clementine Prieur, and Youssef Marzouk,
  \emph{Gradient-based dimension reduction of multivariate vector-valued
  functions}, SIAM Journal on Scientific Computing {42} (2020), no.~1,
  A534--A558.

\end{thebibliography}

\providecommand{\bysame}{\leavevmode\hbox to3em{\hrulefill}\thinspace}
\providecommand{\MR}{\relax\ifhmode\unskip\space\fi MR }
\providecommand{\MRhref}[2]{%
  \href{http://www.ams.org/mathscinet-getitem?mr=#1}{#2}
}
\providecommand{\href}[2]{#2}

\end{document}